\documentclass[reqno,11pt]{article}
\usepackage{graphicx, amsmath, amsthm,amssymb}

\def\normH#1#2{\norm{#2}_{H_0^{#1}}}
\def\TKr{T_{\math{Kr}}}
\def\V{\Vhat}

\input{stochbif.sty}

\begin{document}


\title{Sharp estimates for metastable lifetimes\\ in parabolic SPDEs: 
Kramers' law and beyond
}
\author{Nils Berglund and Barbara Gentz\footnote{Research supported by CRC 701 \lq\lq Spectral Structures and Topological Methods in Mathematics\rq\rq.}}
\date{}

\maketitle

\begin{abstract}
We prove a Kramers-type law for metastable transition times for a class of
one-dimensional parabolic stochastic partial differential equations (SPDEs) with
bistable potential. The expected transition time between local minima of the
potential energy depends exponentially on the energy barrier to overcome,
with an explicit prefactor related to functional determinants. Our results cover
situations where the functional determinants vanish owing to a bifurcation,
thereby rigorously proving the results of formal computations announced
in~\cite{BG09a}. The proofs rely on a spectral Galerkin approximation of the
SPDE by a finite-dimensional system, and on a potential-theoretic approach to
the computation of transition times in finite dimension.
\end{abstract}

\leftline{\small{\it 5 February 2012.\/} Revised {\it 29 November 2012.\/} }
\leftline{\small 2010 {\it Mathematical Subject Classification.\/} 
60H15, 
35K57 (primary),    
60J45, 
37H20 (secondary) 
}
\noindent{\small{\it Keywords and phrases.\/}
Stochastic partial differential equations, 
parabolic equations,
reaction--diffusion equations, 
meta\-stability, 
Kramers' law, 
exit problem, 
transition time,
large deviations, 
Wentzell--Freidlin theory, 
potential theory, 
capacities,
Galerkin approximation, 
subexponential asymptotics,
pitchfork bifurcation.
}  

\tableofcontents


\section{Introduction}
\label{sec_intro}

Metastability is a common physical phenomenon, in which a system
quickly moved across a first-order phase transition line takes a long time to
settle in its equilibrium state. This behaviour has been established rigorously
in two main classes of mathematical models. The first class consists of lattice
models with Markovian dynamics of Metropolis type, such as the Ising model with
Glauber dynamics or the lattice gas with Kawasaki dynamics
(see~\cite{denHollander04,OlivieriVares05} for recent surveys). 

The second class of models consists of stochastic differential equations driven
by weak Gaussian white noise. For dissipative drift, sample paths of such
equations tend to spend long time spans near attractors of the system without
noise, with occasional transitions between attractors. In the particular case
where the drift term is given by minus the gradient of a potential, the
attractors are local minima of the potential, and the mean transition time
between local minima is governed by Kramers' law~\cite{Eyring,Kramers}: In the
small-noise limit, the transition time is exponentially large in the potential
barrier height between the minima, with a multiplicative prefactor depending on
the curvature of the potential at the local minimum the process starts in and at
the highest saddle crossed during the transition. While the exponential
asymptotics was proved to hold by Freidlin and Wentzell using the theory of
large deviations~\cite{VF69,FW}, the first rigorous proof of Kamers' law,
including the prefactor, was obtained more recently by Bovier, Eckhoff, Gayrard
and Klein~\cite{BEGK,BGK} via a potential-theoretic approach.
See~\cite{Berglund_Kramers_11} for a recent review.

The aim of the present work is to extend Kramers' law to a class of parabolic
stochastic partial differential equations of the form 
\begin{equation}
 \label{SPDEin}
\6u_t(x) = \bigbrak{\Delta u_t(x) - U'(u_t(x))} \6t + \sqrt{2\eps}\6W(t,x)\;, 
\end{equation} 
where $x$ belongs to an interval $[0,L]$, $u(x)$ is real-valued and $W(t,x)$
denotes space-time white noise. If the potential $U$ has several local minima
$u_i$, the deterministic limiting system admits several stable stationary
solutions: these are simply the constant solutions, equal  to $u_i$ everywhere.
It is natural to expect that the transition time between these stable solutions
is also governed by a formula of Kramers type. In the case of the double-well
potential $U(u)=\frac14u^4 - \frac12 u^2$, the exponential asymptotics of the
transition time was determined and proved to hold by Faris and
Jona-Lasinio~\cite{Faris_JonaLasinio82}. The prefactor was computed formally, by
analogy with the finite-dimensional case, by Maier and
Stein~\cite{Maier_Stein_PRL_01,Maier_Stein_SPIE_2003,Stein04}, except for
particular interval lengths $L$ at which Kamers' formula breaks down because of
a bifurcation. The behaviour near bifurcation values has been derived
formally in~\cite{BG09a}. 

In the present work, we provide a full proof for Kramers' law for SPDEs of the
form~\eqref{SPDEin}, for a general class of double-well potentials $U$. The
results cover all finite positive values of the interval length, and thus
include bifurcation values. One of the main ingredients 
of the proof is a result by Bl\"omker and Jentzen on spectral Galerkin
approximations~\cite{Blomker_Jentzen_09}, which allows us to reduce the system
to a finite-dimensional one. This reduction requires some a priori bounds on
moments of transition times, which we obtain by large-deviation techniques
(though it might be possible  to obtain them by other methods). Transition
times for the finite-dimensional equation can be accurately estimated by the
potential-theoretic approach of~\cite{BEGK,BGK}, provided one can control
capacities uniformly in the dimension. Such a control has been achieved
in~\cite{BBM2010} in a particular case, the so-called synchronised regime of a
chain of coupled bistable particles introduced in~\cite{BFG06a,BFG06b}. 
Part of the work of the present paper consists in establishing such a control
for a general class of systems. We note that although we limit ourselves to the
one-dimensional case, there seems to be no fundamental obstruction to extending
the technique to SPDEs in higher dimensions driven by a $Q$-Wiener process. 
Very recently, Barret has independently obtained an
alternative proof of Kramers' law for non-bifurcating one-dimensional
SPDEs, using a different approach based on approximations by finite
differences~\cite{Barret_2012}.

The remainder of this paper is organised as follows. Section~\ref{sec_results}
contains the precise definition of the model, an overview of needed properties
of the deterministic system, and the statement of all results.
Section~\ref{sec_proof} outlines the essential steps of the proofs. Technical
details of the proofs are deferred to subsequent sections. Section~\ref{sec_det}
contains the needed estimates on the deterministic partial differential
equation, including an infinite-dimensional normal-form analysis of
bifurcations. In Section~\ref{sec_ld} we derive the required a priori estimates
for the stochastic system, mainly based  on large-deviation principles.
Section~\ref{sec_cap} contains the sharp estimates of capacities, while
Section~\ref{sec_times} combines the previous results to obtain precise
estimates of expected transition times in finite dimension. 


\bigskip
\noindent
{\bf Acknowledgements:}
We would like to thank Florent Barret, Dirk Bl\"omker, Martin Hairer, Arnulf
Jentzen and Dan Stein for helpful discussions. 
BG thanks the MAPMO, Orl\'eans, and NB thanks the CRC 701 {\it
Spectral  Structures and Topological Methods in Mathematics\/} at the  
University of Bielefeld, for kind hospitality and financial support.


\newpage



\section{Results}
\label{sec_results}


\subsection{Parabolic SPDEs with bistable potential}
\label{ssec_def}

Let  $L$ be a positive constant, and let $E=\cC([0,L],\R)$ denote the Banach
space of continuous functions $u:[0,L]\to\R$, equipped with the
sup norm $\norm{\cdot}_{L^\infty}$. 

We consider the parabolic SPDE
\begin{equation}
\6u_t(x) = \bigbrak{\Delta u_t(x) - U'(u_t(x))} \6t + \sqrt{2\eps}\6W(t,x)\;,
\qquad t\in\R_+\;, \ x\in[0,L]
 \label{SPDE}
\end{equation} 
with 
\begin{itemiz}
\item	either periodic boundary conditions (b.c.)
\begin{equation}
 \label{def01}
u(0) = u(L)\;, 
\end{equation} 
\item	or zero-flux Neumann boundary conditions
\begin{equation}
 \label{def02}
\dpar ux(0) = \dpar ux(L) = 0\;,
\end{equation} 
\end{itemiz}
and initial condition $u_0\in E$, satisfying the same boundary conditions.

In~\eqref{SPDE}, $\Delta$ denotes the second derivative (the one-dimensional
Laplacian), $\eps>0$ is a small parameter, and $W(t,x)$ denotes space--time
white noise, defined as the cylindrical Wiener process compatible with the b.c. 
The local potential $U:\R\to\R$ will be assumed to satisfy a certain number of
properties, which are detailed below. When considering a general class of local
potentials, it is useful to keep in mind the example
\begin{equation}
 \label{def03}
U(u) = \frac14 u^4 - \frac12 u^2\;. 
\end{equation} 
Observe that $U$ has two minima, located in $u=-1$ and $u=+1$, and a
local maximum in $u=0$. Furthermore, the quartic growth as $u\to\pm\infty$
makes $U$ a confining potential. As a result, for small $\eps$, solutions
of~\eqref{SPDE}
will be localised with high probability, with a preference for staying near
$u=1$ or $u=-1$. 

The bistable and confining nature of $U$ are two essential features that we
want to keep for all considered local potentials. A first set of assumptions on
$U$ is the following:

\begin{assump}[Assumptions on the class of potentials $U$] \hfill
\label{assump_U}
\begin{enumH}
\item[U1:]	$U: \R\to\R$ is of class $\cC^3$. 
In some cases (namely, when $L$ is
close to $\pi$ for Neumann b.c.\ and close to $2\pi$ for periodic b.c.), our
results require $U$ to be of class $\cC^5$.\footnote{Actually, for $L$ near
a critical value, the results hold under the assumption~$U\in\cC^{4}$, with a
weaker control on the error terms.}\

\item[U2:]	$U$ has exactly two local minima and one local maximum, and
$U''$ is nonzero at all three stationary points. Without loss of 
generality, we may assume that the local maximum is in $u=0$ and that
$U''(0)=-1$. The positions of the minima will be denoted by $u_-<0<u_+$.
 
\item[U3:]	There exist constants $M_0>0$ and $p_0\geqs2$ such that
the potential and its derivatives satisfy 
$\abs{U^{(j)}(u)} \leqs M_0(1+\abs{u}^{2p_0-j})$ for $j=0,1,2,3$ and all
$u\in\R$.

\item[U4:]	There exist constants $\alpha\in\R, \beta>0$ such that
$U(u)\geqs \beta u^2 - \alpha$ for all $u\in\R$.

\item[U5:]	For any $\gamma>0$, there exists an $M_1(\gamma)$ such
that $U'(u)^2 \geqs \gamma u^2 - M_1(\gamma)$ for all $u\in\R$. In addition,
there exist constants $a>0$ and $b,c\in\R$ such that 
\begin{equation}
 \label{def04}
 U'(u+v) - U'(u) \geqs a v^{2p_0-1} + b \abs{u}^{2p_0-1} + c
\end{equation} 
holds for all $u\in\R$ and all $v\geqs0$. 

\item[U6:]	There exists a constant $M_2$ such that $U''(u)\geqs -M_2$ and 
\begin{equation}
 \label{def05}
 u^{2(p_0-1)}U''(u) - 2(p_0-1) u^{2p_0-3}U'(u) \geqs -M_2
\end{equation} 
for all $u\in\R$.
\end{enumH}
\end{assump}

\begin{remark}
 \label{rem_ass_U}
 A sufficient condition for U3--U6 of Assumption~\ref{assump_U} to hold is that 
 the potential can be written as $U(u)=p(u)+U_0(u)$, where $p$ is a polynomial
of even degree $2p_0\geqs4$, with strictly positive leading coefficient, and
$U'_0$ is a Lipschitz continuous function, cf.~\cite[Remark~2.6]{Cerrai_1999}. 
\end{remark}

Let us recall the definition of a mild solution of~\eqref{SPDE}. We denote by
$\e^{\Delta t}$ the Markov semigroup of the heat equation $\sdpar ut = \Delta
u$, defined by the convolution 
\begin{equation}
 \label{heat_semigroup}
(\e^{\Delta t}u)(x) = \int_0^L G_t(x,y)u(y)\6y\;. 
\end{equation} 
Here $G_t(x,y)$ denotes the Green function of the Laplacian compatible with the
considered boundary conditions. It can be written as
\begin{equation}
 \label{mild01}
G_t(x,y) = \sum_k \e^{-\nu_k t}e_k(x)\cc{e_k(y)}\;, 
\end{equation} 
where the $e_k$ form a complete orthonormal basis of eigenfunctions of the
Laplacian, with eigenvalues $-\nu_k$. 
That is, 
\begin{itemiz}
\item	for periodic b.c., 
\begin{equation}
 \label{cob_periodic}
e_k(x) = \frac{1}{\sqrt{L}}\e^{2k\pi \icx x/L}\;,\quad k\in\Z\;,
\qquad 
\text{and } \nu_k=\biggpar{\frac{2k\pi}{L}}^2\;;
\end{equation} 
\item	for Neumann b.c., 
\begin{equation}
 \label{cob_Neumann}
e_0(x) = \frac{1}{\sqrt{L}}\;,\;
e_k(x) = \sqrt{\frac{2}{L}}\cos\biggpar{\frac{k\pi x}{L}}\;,\quad
k\in\N\;,
\qquad 
\text{and } \nu_k=\biggpar{\frac{k\pi}{L}}^2\;.
\end{equation} 
\end{itemiz}

A \emph{mild solution}\/ of the SPDE~\eqref{SPDE} is by definition a solution to
the integral equation 
\begin{equation}
 \label{mild_solution}
u_t = \e^{\Delta t}u_0 
- \int_0^t \e^{\Delta(t-s)}U'(u_s)\6s
+ \sqrt{2\eps} \int_0^t \e^{\Delta(t-s)}\6W(s)\;.
\end{equation} 
Here the stochastic integral can be represented as a series of one-dimensional
It\^o integrals 
\begin{equation}
 \label{mild_02} 
\int_0^t \e^{\Delta(t-s)}\6W(s)
= \sum_k \int_0^t \e^{-\nu_k(t-s)} \6W^{(k)}_s e_k \;,
\end{equation} 
where the $W^{(k)}_t$ are independent standard Wiener processes (see for
instance~\cite{Jetschke_86}). It is known that for a confining local potential
$U$,~\eqref{SPDE} admits a pathwise unique mild
solution in $E$~\cite{DaPrato_Jabczyk_92}. 

\begin{figure}
\centerline{\includegraphics*[clip=true,width=\textwidth]{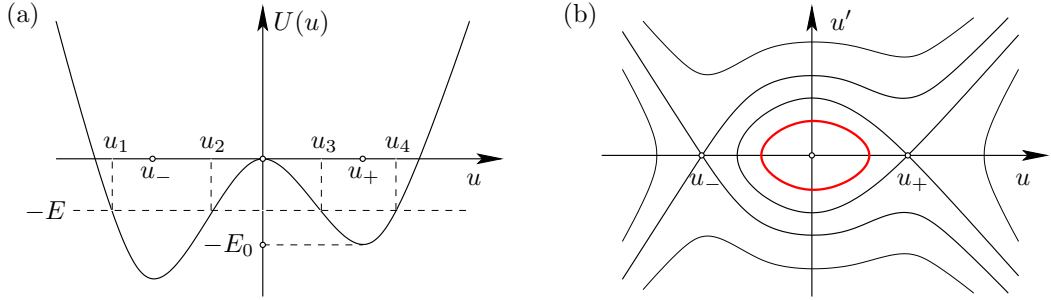}}
\caption[]{(a) Example of a local potential $U$. (b) Level sets of
the first integral $H(u,u')$. Bounded stationary solutions compatible with the
boundary conditions only exist in the inner region $\set{H(u,u')<E_0}$. The
periodic orbit contained in $\set{H(u,u')=E}$ crosses the $u$-axis at points
$u=u_2(E)$ and $u=u_3(E)$, defined in (a).
}
\label{fig_potential}
\end{figure}


\subsection{The deterministic equation}
\label{ssec_det}

Consider for a moment the deterministic partial differential equation 
\begin{equation}
 \label{PDE}
\sdpar{u_t}{t}(x) = \Delta u_t(x) - U'(u_t(x))\;. 
\end{equation} 
Stationary solutions of~\eqref{PDE} have to satisfy the second-order ordinary
differential equation 
\begin{equation}
 \label{ODE}
u''(x) = U'(u(x)) \;,
\end{equation} 
together with the boundary conditions. Note that this equation describes the
motion of a particle of unit mass in the inverted potential $-U$. 
There are exactly three stationary solutions which do not depend on $x$, given
by 
\begin{equation}
 \label{ODE01}
u^*_-(x) \equiv u_-\;, \qquad 
u^*_+(x) \equiv u_+\;, \qquad 
u^*_0(x) \equiv 0\;. 
\end{equation} 
Depending on the boundary conditions and the value of $L$, there may be
additional, non-constant stationary
solutions. They can be found by observing that~\eqref{ODE} is a Hamiltonian
system, with first integral 
\begin{equation}
 \label{ODE02}
H(u,u') = \frac12 (u')^2 - U(u)\;. 
\end{equation} 
Orbits of~\eqref{ODE} belong to level sets of $H$ (\figref{fig_potential}b). 
Bounded orbits only exist for $H<E_0$, where~\footnote{Here and below, we use the shorthands $a\vee b \defby \max\set{a,b}$ and $a\wedge b \defby \min\set{a,b}$.}%
\begin{equation}
 \label{ODE03}
E_0 = -(U(u_-)\vee U(u_+))\;. 
\end{equation} 
For any $E\in(0,E_0)$, there exist exactly four values 
$u_1(E)<u_2(E)<0<u_3(E)<u_4(E)$ of $u$ for which $U(u)=-E$
(\figref{fig_potential}a). The periodic solution corresponding to $H=E$ crosses
the $u$-axis at $u=u_2(E)$ and $u=u_3(E)$, and has a period 
\begin{equation}
 \label{ODE04}
T(E) = 2 \int_{u_2(E)}^{u_3(E)} \frac{\6u}{\sqrt{E+U(u)}}\;. 
\end{equation} 
The fact that $U''(0)=-1$ implies that $\lim_{E\to 0} T(E)=2\pi$ (in this
limit, stationary solutions approach those of a harmonic oscillator with unit
frequency). In addition, we have $\lim_{E\to E_0}T(E)=+\infty$, because the
level set $H=E_0$ is composed of homoclinic orbits (or heteroclinic orbits if
$U(u_-)=U(u_+)$).
 
We will make the following assumption, which imposes an additional condition on
the local potential:

\begin{assump}
\label{assump_U02}
The period $T(E)$ is strictly increasing on $[0,E_0)$.
\end{assump}

\begin{remark}
\label{rem_T(E)} 
A normal-form analysis (cf.~Section~\ref{ssec_normalform}) shows that if
$U\in\cC^5$, then $T(E)$ is increasing near $E=0$ if and only if 
\begin{equation}
 \label{ODE05b}
U^{(4)}(0) > -\frac53 U'''(0)^2\;. 
\end{equation}  
Furthermore, a sufficient (but not necessary) condition for
Assumption~\ref{assump_U02} to hold true is that 
\begin{equation}
 \label{ODE05}
U'(u)^2 - 2U(u)U''(u) > 0 
\qquad
\text{for all $u\in(u_-,u_+)\setminus\set{0}$} 
\end{equation} 
(see Appendix~\ref{app_A}).
Note that this condition is satisfied for the particular
potential~\eqref{def03}.
\end{remark}

Under Assumption~\ref{assump_U02}, nonconstant stationary solutions satisfying
periodic b.c.\ only exist for $L>2\pi$, while stationary solutions satisfying Neumann
b.c.\ only exist for $L>\pi$; they are obtained by taking the top or bottom half
of a closed curve with constant $H$. Additional stationary solutions appear
whenever $L$ crosses a multiple of $2\pi$ or $\pi$. More precisely
(\figref{fig_bifurcation}), 
\begin{itemiz}
\item	for periodic b.c., there exist $n$ families of nonconstant stationary
solutions whenever $L\in(2n\pi,2(n+1)\pi]$ for some $n\geqs1$, where members of
a same family are of the form $u^*_{n,\varphi}(x)=u^*_{n,0}(x+\varphi)$,
$0\leqs\varphi<L$; 
\item	for Neumann b.c., there exist $2n$ nonconstant stationary
solutions whenever $L\in(n\pi,(n+1)\pi]$ for some $n\geqs1$, where solutions
occur in pairs $u^*_{n,\pm}(x)$ related by the symmetry
$u^*_{n,-}(x)=u^*_{n,+}(L-x)$. 
\end{itemiz}

\begin{figure}
\centerline{\includegraphics*[clip=true,width=\textwidth]{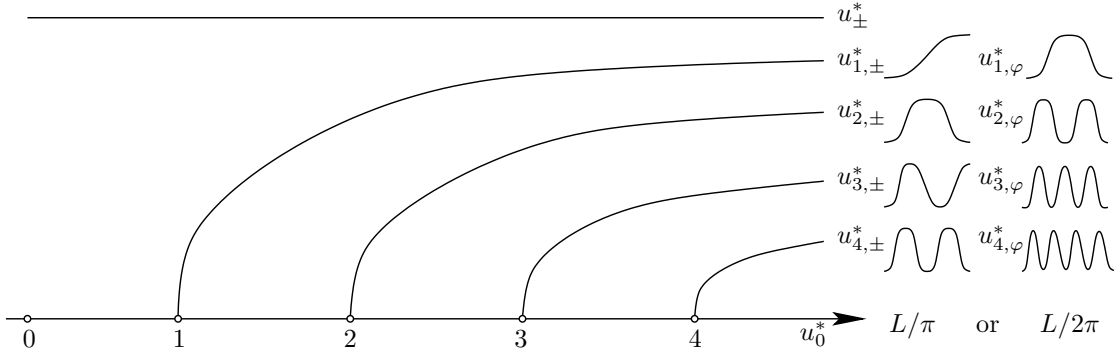}}
\caption[]{Schematic representation of the deterministic bifurcation diagram.
Nonconstant stationary solutions appear whenever $L$ is a multiple of $\pi$ for
Neumann b.c., and of $2\pi$ for periodic b.c. For Neumann b.c., the stationary
solutions $u^*_{n,\pm}$ contain $n$ kinks. For periodic b.c.,
all members of the family $\set{u^*_{n,\varphi}, 0\leqs\varphi<L}$ contain $n$
kink--antikink pairs. The transition states ($n=1$) are also called
instantons~\cite{Maier_Stein_SPIE_2003}. 
}
\label{fig_bifurcation}
\end{figure}

Next we examine the stability of these stationary solutions. Stability of a
stationary solution $u_0$ is determined by the variational equation
\begin{equation}
 \label{var_eqn}
\sdpar{v_t}{t}(x) = \Delta v_t(x) - U''(u_0(x)) v_t(x) 
\bydef Q[u_0]v_t(x)\;,
\end{equation} 
by way of the sign of the eigenvalues of the linear operator $Q[u_0] = \Delta
- U''(u_0(\cdot))$. For the space-homogeneous stationary
solutions~\eqref{ODE01}, the eigenvalues of $Q$ are simply shifted eigenvalues
of the Laplacian. Thus 
\begin{itemiz}
\item	For periodic b.c., the eigenvalues of $Q[u^*_0]$ are given by
$-\lambda_k$, where 
\begin{equation}
 \label{ODE07}
\lambda_k = \nu_k - 1 = \biggpar{\frac{2k\pi }{L}}^2 - 1\;, 
\quad k\in\Z\;. 
\end{equation} 
It follows that $u^*_0$ is always unstable: it has one positive eigenvalue for
$L\leqs2\pi$,
and the number of positive eigenvalues increases by $2$ each time $L$ crosses a
multiple of $2\pi$. The eigenvalues of $Q[u^*_\pm]$ are given by 
$-\nu_k - U''(u_\pm)$ and are always negative, implying that $u^*_+$ and $u^*_-$
are stable. 
\item	For Neumann b.c., the eigenvalues of $Q[u^*_0]$ are given by
$-\lambda_k$, where 
\begin{equation}
 \label{ODE08}
\lambda_k = \nu_k - 1 = \biggpar{\frac{k\pi}{L}}^2 - 1\;, 
\quad k\in\N_0\;. 
\end{equation} 
Again $u^*_0$ is always unstable: it has one positive eigenvalue for
$L\leqs\pi$,
and the number of positive eigenvalues increases by $1$ each time $L$ crosses a
multiple of $\pi$. As before, $u^*_+$ and $u^*_-$ are always stable. 
\end{itemiz}
The problem of determining the stability of the nonconstant stationary solutions
is equivalent to characterising the spectrum of a Schr\"odinger operator, and
thus to solving a Sturm--Liouville problem. In general, there is no
closed-form expression for the eigenvalues. However, a bifurcation analysis for
$L$ equal to multiples of $2\pi$ or $\pi$ (cf. Section~\ref{ssec_normalform})
shows that 
\begin{itemiz}
\item	for periodic b.c., the stationary solutions $u^*_{n,\varphi}$ appearing
at $L=2n\pi$ have $2n-1$ positive eigenvalues and one eigenvalue equal to
zero (associated with translation symmetry), the other eigenvalues being
negative;
\item	for Neumann b.c., the stationary solutions $u^*_{n,\pm}$ appearing
at $L=n\pi$ have $n$ positive eigenvalues while the other eigenvalues are
negative.
\end{itemiz}

A last important object for the analysis is the potential energy  
\begin{equation}
 \label{pot_energy}
V[u] = \int_0^L \biggbrak{\frac12 u'(x)^2 + U(u(x))}\6x\;. 
\end{equation} 
For $u+v$ satisfying the b.c., the Fr\'echet derivative of $V$ at $u$ in the
direction $v$ is given by 
\begin{align}
\nonumber 
\nabla_v V[u] &\defby \lim_{\eps\to0} \frac1\eps \bigpar {V[u+\eps v] -V[u]}\\
\nonumber 
&= \int_0^L \bigbrak{u'(x)v'(x) + U'(u(x))v(x)} \6x \\
&= \int_0^L \bigbrak{-u''(x) + U'(u(x))}v(x) \6x\;.
\label{pot_Frechet}
\end{align} 
Thus stationary solutions of the deterministic equation~\eqref{PDE} are also
stationary points of the potential energy. A similar computation shows that the
second-order Fr\'echet derivative of $V$ at $u$ is the bilinear map 
\begin{equation}
 \label{pot_Frechet2}
\nabla^2_{v_1,v_2} V[u] :
(v_1,v_2) \mapsto -\int_0^L (Q[u]v_1)(x)v_2(x)\6x\;. 
\end{equation} 
Hence the eigenvalues of the second derivative coincide, up to their sign, with
those of the Sturm--Liouville problem for the variational
equation~\eqref{var_eqn}. In particular, the stable stationary solutions $u^*_+$
and $u^*_-$ are local minima of the potential energy. 

We call \emph{transition states}\/ between $u^*_+$ and $u^*_-$ the stationary
points of $V$ at which $\nabla^2 V$ has one and only one negative eigenvalue.
Thus 
\begin{itemiz}
\item	for periodic b.c., $u^*_0$ is the only transition state for $L\leqs
2\pi$, while for $L>2\pi$, all members of the family $u^*_{1,\varphi}$ are
transition states;
\item	for Neumann b.c., $u^*_0$ is the only transition state for $L\leqs
\pi$, while for $L>\pi$, the transition states are the two stationary solutions
$u^*_{1,\pm}$.
\end{itemiz}
Note that for given $L$ and given b.c., $V$ has the same value at all
transition states. Transition states are characterised by the following
property: Consider all continuous paths $\gamma$ in $E$ connecting $u^*_-$ to
$u^*_+$. For each of these paths, determine the maximal value of $V$ along the
path, and call \emph{critical}\/ those paths for which that value is the
smallest possible. Then for any critical path, the maximal value of $V$ is
assumed on a transition state. 


\subsection{Main results}
\label{ssec_res}

We can now state the main results of this work. We start with the case of
Neumann b.c. We fix parameters $r, \rho>0$ and an initial condition $u_0$ such
that $\norm{u_0-u^*_-}_{L^\infty}\leqs r$. Let 
\begin{equation}
 \label{main01}
\tau_+ = \inf\bigsetsuch{t>0}{\norm{u_t-u^*_+}_{L^\infty} < \rho}\;.
\end{equation} 
We are interested in sharp estimates of the expected first-hitting time 
$\expecin{u_0}{\tau_+}$ for small values of $\eps$. 

\begin{figure}
\centerline{\includegraphics*[clip=true,width=\textwidth]{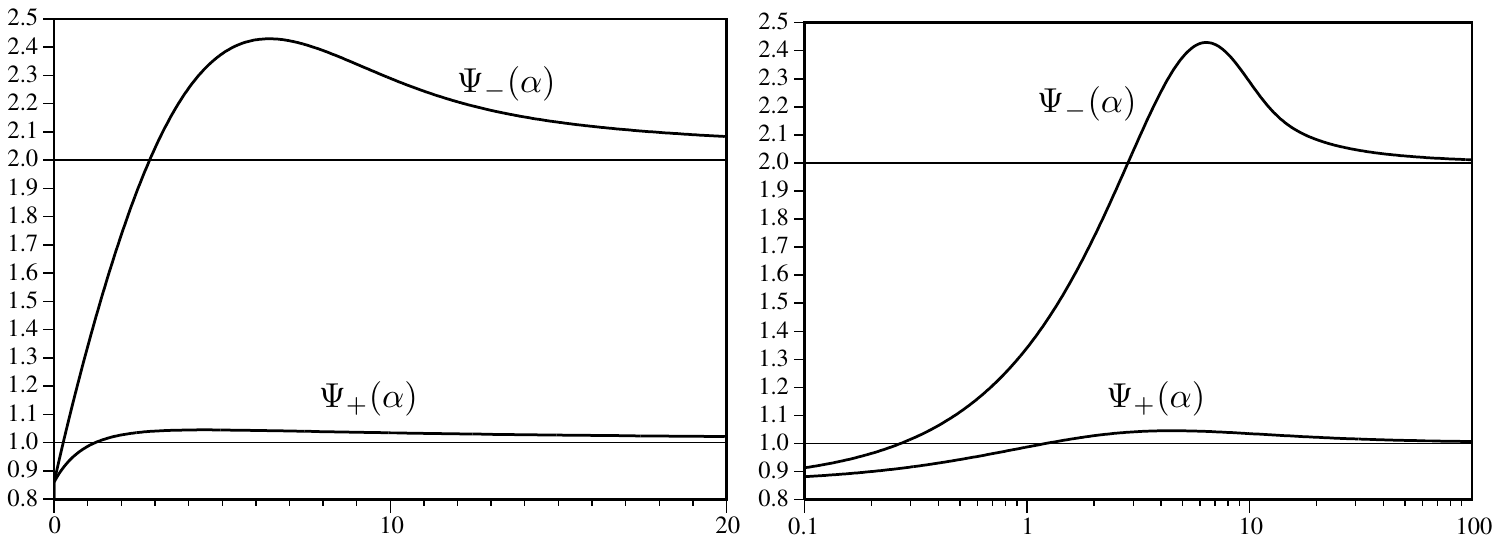}}
\caption[]{The functions $\Psi_\pm(\alpha)$ shown on a linear and on a
logarithmic scale. 
}
\label{fig_Psi}
\end{figure}

Recall from~\eqref{ODE08} that the eigenvalues of the variational equation at
$u^*_0\equiv0$ are given by $-\lambda_k$ where $\lambda_k = (k\pi/L)^2-1$. Those
at $u^*_-$ are given by $-\nu_k^-$ where 
\begin{equation}
 \label{main02}
\nu_k^- = \biggpar{\frac{k\pi}{L}}^2 + U''(u_-)\;.  
\end{equation} 
When $L>\pi$, we denote the eigenvalues at the transition states $u^*_{1,\pm}$
by $-\mu_k$ where 
\begin{equation}
 \label{main03}
\mu_0 < 0 < \mu_1 < \mu_2 < \dots
\end{equation}
We further introduce two functions $\Psi_\pm:\R_+\to\R_+$, which play a r\^ole
for the behaviour of $\expecin{u_0}{\tau_+}$ when $L$ is close to $\pi$. They
are given by 
\begin{align}
\label{psiplus}
\Psi_+(\alpha) &= \sqrt{\frac{\alpha(1+\alpha)}{8\pi}}
\e^{\alpha^2/16} K_{1/4} \biggpar{\frac{\alpha^2}{16}}\;,\\
\Psi_-(\alpha) &= \sqrt{\frac{\pi\alpha(1+\alpha)}{32}}
\e^{-\alpha^2/64} 
\biggbrak{I_{-1/4} \biggpar{\frac{\alpha^2}{64}}
+ I_{1/4} \biggpar{\frac{\alpha^2}{64}}}\;,
\label{psiminus} 
\end{align}
where $I_{\pm1/4}$ and $K_{1/4}$ denote modified Bessel functions of
first and second kind. The functions $\Psi_\pm$ are bounded below and above by
positive constants, and satisfy 
\begin{equation}
 \label{psilimit}
\lim_{\alpha\to+\infty} \Psi_+(\alpha) = 1\;, 
\qquad
\lim_{\alpha\to-\infty} \Psi_-(\alpha) = 2\;, 
\end{equation} 
and
\begin{equation}
 \label{psi0}
\lim_{\alpha\to0} \Psi_+(\alpha) = \lim_{\alpha\to0} \Psi_-(\alpha)
= \frac{\Gamma(1/4)}{2^{5/4}\sqrt{\pi}}\;.
\end{equation} 
See~\figref{fig_Psi} for plots of these functions. 

\begin{theorem}[Neumann boundary conditions]
\label{thm_Neumann}
For Neumann b.c.\ and sufficiently small $r, \rho$ and $\eps$, the following
holds true.
\begin{enum}
\item	If $L<\pi$ and $L$ is bounded away from $\pi$, then 
\begin{equation}
 \label{Kramers_Neumann_smallL}
\bigexpecin{u_0}{\tau_+} = 
2\pi \biggpar{\frac{1}{\abs{\lambda_0}\nu^-_0} \prod_{k\geqs1}
\frac{\lambda_k}{\nu^-_k}}^{1/2} 
\e^{(V[u^*_0]-V[u^*_-])/\eps}
\bigbrak{1+\Order{\eps^{1/2}\abs{\log\eps}^{3/2}}}\;.
\end{equation} 
\item	If $L>\pi$ and $L$ is bounded away from $\pi$, then 
\begin{equation}
 \label{Kramers_Neumann_largeL}
\bigexpecin{u_0}{\tau_+} = 
\pi \biggpar{\frac{1}{\abs{\mu_0}\nu^-_0} \prod_{k\geqs1}
\frac{\mu_k}{\nu^-_k}}^{1/2} 
\e^{(V[u^*_{1,\pm}]-V[u^*_-])/\eps}
\bigbrak{1+\Order{\eps^{1/2}\abs{\log\eps}^{3/2}}}\;.
\end{equation} 
\item	If $L\leqs\pi$ and $L$ is in a neighbourhood of $\pi$, then 
\begin{equation}
 \label{Kramers_Neumann_Lleqpi}
 \bigexpecin{u_0}{\tau_+} = 
2\pi \biggpar{\frac{\lambda_1+\sqrt{C\eps}}
{\abs{\lambda_0}\nu^-_0\nu^-_1}
\prod_{k\geqs2}
\frac{\lambda_k}{\nu^-_k}}^{1/2} 
\frac{\e^{(V[u^*_0]-V[u^*_-])/\eps}}
{\Psi_+(\lambda_1/\sqrt{C\eps})}
\bigbrak{1+R_+(\eps,\lambda_1)}\;,
\end{equation} 
where 
\begin{equation}
 \label{Kramers_defC}
C = \frac1{4L} \Bigbrak{ U^{(4)}(0) 
+ \frac{8\pi^2 - 3L^2}{4\pi^2 - L^2} U'''(0)^2}\;, 
\end{equation}
and the remainder $R_+$ satisfies 
\begin{equation}
 \label{Kramers_Rplus}
R_+(\eps,\lambda) = \biggOrder{\biggbrak{\frac{\eps\abs{\log\eps}^3}
{\max\set{\abs{\lambda},\sqrt{\eps\abs{\log\eps}}}}}^{1/2}}\;.
\end{equation} 
\item	If $L\geqs\pi$ and $L$ is in a neighbourhood of $\pi$, then 
\begin{equation}
 \label{Kramers_Neumann_Lgeqpi}
 \bigexpecin{u_0}{\tau_+} = 
2\pi \biggpar{\frac{\mu_1+\sqrt{C\eps}}
{\abs{\mu_0}\nu^-_0\nu^-_1}
\prod_{k\geqs2}
\frac{\mu_k}{\nu^-_k}}^{1/2} 
\frac{\e^{(V[u^*_{1,\pm}]-V[u^*_-])/\eps}}
{\Psi_-(\mu_1/\sqrt{C\eps})}
\bigbrak{1+R_-(\eps,\mu_1)}\;,
\end{equation} 
where $C$ is given by~\eqref{Kramers_defC}, and the remainder $R_-$ is of the
same order as $R_+$. 
\end{enum}
\end{theorem}

Note that~\eqref{psi0} (together with the fact that $\mu_k(L)\to\lambda_k(\pi)$
as $L\to\pi_+$) shows that $\expecin{u_0}{\tau_+}$ is indeed continuous
at $L=\pi$. In a neighbourhood of order $\sqrt{\eps}$ of $L=\pi$, the prefactor
of the transition time is of order $\eps^{1/4}$, while it is constant to
leading order when $L$ is bounded away from $\pi$.

We have written here the different expressions for the expected transition time
in a generic way, in terms of eigenvalues and potential-energy
differences. Note however that several quantities appearing in the theorem admit
more explicit expressions:
\begin{itemiz}
\item	We have $V[u^*_0]=0$ and $V[u^*_-]=U(u_-)$, while $V[u^*_{1,\pm}]$ is
determined by solving~\eqref{ODE} with the help of the first
integral~\eqref{ODE02}. For the symmetric double-well potential~\eqref{def03},
it can be expressed explicitly in terms of elliptic integrals. 

\item	The two identities 
\begin{equation}
 \label{product_formulas}
\prod_{k=1}^\infty \biggpar{1 - \frac{x^2}{k^2}} = \frac{\sin(\pi x)}{\pi x}\;,
\qquad 
\prod_{k=1}^\infty \biggpar{1 + \frac{x^2}{k^2}} = \frac{\sinh(\pi x)}{\pi x} 
\end{equation} 
imply that the prefactor in~\eqref{Kramers_Neumann_smallL} is given by 
\begin{equation}
 \label{prefactor_Neumann_smallL} 
2\pi \biggpar{\frac{1}{\abs{\lambda_0}\nu^-_0} \prod_{k\geqs1}
\frac{\lambda_k}{\nu^-_k}}^{1/2} 
= 2\pi \Biggpar{\frac{\sin
L}{\sqrt{U''(u_-)}\,\sinh\bigpar{L\sqrt{U''(u_-)}\,}}}^{1/2}\;.
\end{equation} 

\item	Since there is no closed-form expression for the eigenvalues $\mu_k$,
it might seem impossible to compute the prefactor appearing
in~\eqref{Kramers_Neumann_largeL}. In fact, techniques developed for the
computation of Feynman integrals allow to compute the product of such ratios of
eigenvalues, also called a ratio of functional determinants,
see~\cite{Forman1987,McKane_Tarlie_1995,ColinVerdiere1999,
Maier_Stein_PRL_01,Maier_Stein_SPIE_2003} .
\end{itemiz}

\begin{figure}
\centerline{\includegraphics*[clip=true,width=\textwidth]{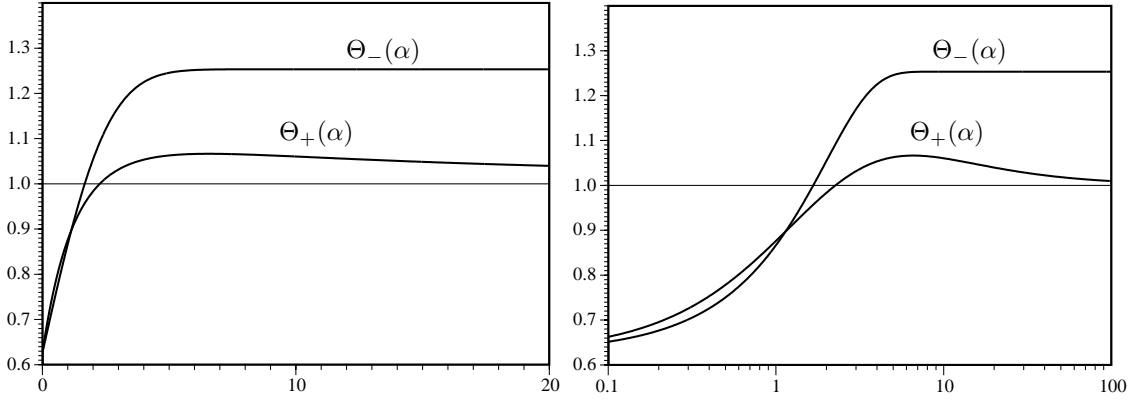}}
\caption[]{The functions $\Theta_\pm(\alpha)$ shown on a linear and on a
logarithmic scale. 
}
\label{fig_Theta}
\end{figure}

We now turn to the case of periodic b.c. In that case, the
eigenvalues of the variational equation at
$u^*_0\equiv0$ are given by $-\lambda_k$ where $\lambda_k = (2k\pi/L)^2-1$.
Those at $u^*_-$ are given by $-\nu_k^-$ where 
\begin{equation}
 \label{main12}
\nu_k^- = \biggpar{\frac{2k\pi}{L}}^2 + U''(u_-)\;.  
\end{equation} 
When $L>2\pi$, we denote the eigenvalues at the family of transition states
$u^*_{1,\varphi}$ by $-\mu_k$ where 
\begin{equation}
 \label{main13}
\mu_0 < \mu_{-1}=0 < \mu_{1} < \mu_2, \mu_{-2} <  \dots
\end{equation}
We further introduce two functions $\Theta_\pm:\R_+\to\R_+$, which play a r\^ole
for the behaviour of $\expecin{u_0}{\tau_+}$ when $L$ is close to $2\pi$. They
are given by 
\begin{align}
\label{thetaplus}
\Theta_+(\alpha) &= \sqrt{\frac{\pi}{2}} (1+\alpha) \e^{\alpha^2/8}
\Phi\biggpar{-\frac{\alpha}{2}}\;,\\
\Theta_-(\alpha) &= \sqrt{\frac{\pi}{2}} 
\Phi\biggpar{\frac{\alpha}{2}}\;,
\label{thetaminus} 
\end{align}
where $\Phi(x)=(2\pi)^{-1/2}\int_{-\infty}^x \e^{-t^2/2}\6t$ denotes the
distribution function of a standard Gaussian random variable. 
The functions $\Theta_\pm$ are bounded below and above by positive constants,
and satisfy 
\begin{equation}
 \label{thetalimit}
\lim_{\alpha\to+\infty} \Theta_+(\alpha) = 1\;, 
\qquad
\lim_{\alpha\to-\infty} \Theta_-(\alpha) = \sqrt{\frac{\pi}{2}} \;, 
\end{equation} 
and
\begin{equation}
 \label{theta0}
\lim_{\alpha\to0} \Theta_+(\alpha) = \lim_{\alpha\to0} \Theta_-(\alpha)
= \sqrt{\frac{\pi}{8}}\;.
\end{equation} 
See~\figref{fig_Theta} for plots of these functions.

\begin{theorem}[Periodic boundary conditions]
\label{thm_Kramers_periodic_bc} 
For periodic b.c.\ and sufficiently small $r, \rho$ and $\eps$, the following
holds true.
\begin{enum}
\item	If $L<2\pi$ and $L$ is bounded away from $2\pi$, then 
\begin{equation}
 \label{Kramers_periodic_smallL}
\bigexpecin{u_0}{\tau_+} = 
\frac{2\pi}{\sqrt{\abs{\lambda_0}\nu^-_0}} \biggpar{\prod_{k\geqs1}
\frac{\lambda_k}{\nu^-_k}}
\e^{(V[u^*_0]-V[u^*_-])/\eps}
\bigbrak{1+\Order{\eps^{1/2}\abs{\log\eps}^{3/2}}}\;.
\end{equation} 
\item	If $L\leqs2\pi$ and $L$ is in a neighbourhood of $2\pi$, then 
\begin{equation}
 \label{Kramers_periodic_Lleqpi}
\bigexpecin{u_0}{\tau_+} = 
\frac{2\pi}{\sqrt{\abs{\lambda_0}\nu^-_0}} 
\frac{\lambda_1+\sqrt{2C\eps}}{\nu_1^-}
\biggpar{\prod_{k\geqs2}
\frac{\lambda_k}{\nu^-_k}}
\frac{\e^{(V[u^*_0]-V[u^*_-])/\eps}}{\Theta_+(\lambda_1/\sqrt{2C\eps}\,)}
\bigbrak{1+R_+(\eps,\lambda_1)}\;,
\end{equation} 
where 
\begin{equation}
 \label{Kramers_defC2}
C = \frac1{4L} \Bigbrak{ U^{(4)}(0) 
+ \frac{32\pi^2 - 3L^2}{16\pi^2 - L^2} U'''(0)^2}\;, 
\end{equation}
and the remainder $R_+$ satisfies~\eqref{Kramers_Rplus}.

\item	If $L\geqs2\pi$ and $L$ is in a neighbourhood of $2\pi$, then 
\begin{equation}
 \label{Kramers_periodic_Lgeqpi}
 \bigexpecin{u_0}{\tau_+} = 
\frac{2\pi}{\sqrt{\abs{\mu_0}\nu^-_0}} 
\frac{\sqrt{2C\eps}}{\nu_1^-}
\biggpar{\prod_{k\geqs2}
\frac{\sqrt{\mu_k\mu_{-k}}}{\nu^-_k}}
\frac{\e^{(V[u^*_{1,\varphi}]-V[u^*_-])/\eps}}
{\Theta_-(\mu_1/\sqrt{8C\eps}\,)}
\bigbrak{1+R_-(\eps,\mu_1)}\;,
\end{equation} 
where $C$ is given by~\eqref{Kramers_defC2}, and the remainder $R_-$ is of the
same order as $R_+$. 

\item	If $L>2\pi$ and $L$ is bounded away from $2\pi$, then 
\begin{equation}
 \label{Kramers_periodic_LargeL}
 \bigexpecin{u_0}{\tau_+} = 
\frac{2\pi}{\sqrt{\abs{\mu_0}\nu^-_0}} 
\frac{\sqrt{2\pi\eps\mu_1}}{\nu_1^- }
\biggpar{\prod_{k\geqs2}
\frac{\sqrt{\mu_k\mu_{-k}}}{\nu^-_k}}
\frac{\e^{(V[u^*_{1,0}]-V[u^*_-])/\eps}}
{L \norm{(u^*_{1,0})'}_{L^2}}
\bigbrak{1+\Order{\eps^{1/2}\abs{\log\eps}^{3/2}}}\;.
\end{equation}
\end{enum}
\end{theorem}

Note that for $L\geqs 2\pi - \Order{\sqrt{\eps}}$, the prefactor of
$\expecin{u_0}{\tau_+}$ is proportional to $\sqrt{\eps}/L$. This is due to the
existence of the continuous family of transition states 
\begin{equation}
 \label{main20}
u^*_{1,\varphi}(x) = u^*_{1,0}(x+\varphi)\;, 
\qquad
0\leqs \varphi < L
\end{equation} 
owing to translation symmetry. The quantity 
\begin{equation}
 \label{main21}
 L \norm{(u^*_{1,0})'}_{L^2}
= L \Biggbrak{\int_0^L \biggpar{\dtot{}{x}u^*_{1,0}(x)}^2 \6x}^{1/2}
\end{equation} 
plays the r\^ole of the \lq\lq length of the saddle\rq\rq. One shows (cf.\
Section~\ref{ssec_periodic}) that for $L$ close to $2\pi$, $\mu_1$ is close to
$-2\lambda_1$ and 
\begin{equation}
 \label{main22}
 L \norm{(u^*_{1,0})'}_{L^2}
= 2\pi \sqrt{\frac{\mu_1}{8C}} + \Order{\mu_1}\;, 
\end{equation} 
which implies
shows that~\eqref{Kramers_periodic_Lgeqpi} and~\eqref{Kramers_periodic_LargeL} are 
indeed compatible. 

As in the case of Neumann b.c., several of the above quantities admit more
explicit expressions. For instance, the identities~\eqref{product_formulas}
imply that the prefactor in~\eqref{Kramers_periodic_smallL} is given by  
\begin{equation}
 \label{prefactor_periodic_smallL} 
\frac{2\pi}{\sqrt{\abs{\lambda_0}\nu^-_0}} \biggpar{\prod_{k\geqs1}
\frac{\lambda_k}{\nu^-_k}}
= \frac{2\pi \sin(L/2)}
{\sinh\bigpar{\sqrt{U''(u_-)}\,L/2}}\;.
\end{equation} 
See~\cite{Stein_JSP_04} for an explicit expression of the prefactor for
$L>2\pi$, for a particular class of double-well potentials.



\newpage


\section{Outline of the proof}
\label{sec_proof}


\subsection{Potential theory}
\label{ssec_pot}

A first key ingredient of the proof is the potential-theoretic approach to
metastability of finite-dimensional SDEs developed in~\cite{BEGK,BGK}. 
Given a confining potential $V:\R^d\to\R$, consider the diffusion defined by  
\begin{equation}
 \label{pot01}
\6x_t = -\nabla V(x_t)\6t + \sqrt{2\eps}\6W_t\;, 
\end{equation} 
where $W_t$ denotes $d$-dimensional Brownian motion. The diffusion is
reversible with respect to the invariant measure 
\begin{equation}
 \label{pot02}
\mu(\6x) = \frac{1}{Z} \e^{-V(x)/\eps}\6x\;, 
\end{equation} 
where $Z$ is the normalisation. This follows from the fact that its
infinitesimal generator 
\begin{equation}
 \label{pot03}
\cL = \eps\Delta - \nabla V(x)\cdot\nabla
= \eps \e^{V/\eps} \nabla \cdot \e^{-V/\eps} \nabla 
\end{equation} 
is self-adjoint in $L^2(\R^d, \mu(\6x))$. 

Let $A, B, C\subset\R^d$ be measurable sets which are regular (that is, their
complement is a region with continuously differentiable boundary). We are
interested in the expected first-hitting time 
\begin{equation}
 \label{pot04}
w_A(x) = \bigexpecin{x}{\tau_A}\;. 
\end{equation} 
Dynkin's formula shows that $w_A(x)$ solves the Poisson problem 
\begin{align}
\nonumber
\cL w_A(x) &= -1 & &x\in A^c\;, \\
w_A(x) &= 0 & &x\in A\;.
\label{pot05} 
\end{align}
The solution of~\eqref{pot05} can be expressed in terms of the Green function
$G_{A^c}(x,y)$ as 
\begin{equation}
 \label{pot06}
w_A(x) = -\int_{A^c} G_{A^c}(x,y)\6y\;. 
\end{equation} 
Reversibility implies that the Green function satisfies the symmetry 
\begin{equation}
 \label{pot07}
\e^{-V(x)/\eps} G_{A^c}(x,y) =  \e^{-V(y)/\eps} G_{A^c}(y,x)\;.
\end{equation} 
Another important quantity is the \defwd{equilibrium potential} 
\begin{equation}
 \label{pot08}
h_{A,B}(x) = \bigprobin{x}{\tau_A<\tau_B}\;. 
\end{equation} 
It satisfies the Dirichlet problem 
\begin{align}
\nonumber
\cL h_{A,B}(x) &=0 & &x\in (A\cup B)^c\;, \\
\nonumber
h_{A,B}(x) &= 1 & &x\in A\;, \\
h_{A,B}(x) &= 0 & &x\in B\;,
\label{pot09} 
\end{align}
whose solution can be expressed in terms of the Green function and an 
\defwd{equilibrium measure}\/ $e_{A,B}(\6y)$ on $\partial A$ defined by 
\begin{equation}
 \label{pot10}
h_{A,B}(x) = \int_{\partial A} G_{B^c}(x,y) e_{A,B}(\6y)\;.  
\end{equation} 
Finally, the \defwd{capacity}\/ between $A$ and $B$ is defined as
\begin{equation}
 \label{pot11}
\capacity_A(B) = -\int_{\partial A} \e^{-V(y)/\eps} e_{A,B}(\6y)\;.  
\end{equation}  
The key observation is that the relations~\eqref{pot10}, \eqref{pot07}
and~\eqref{pot06} can be combined to yield 
\begin{align}
\nonumber
\int_{A^c} h_{C,A}(y)\e^{-V(y)/\eps}\6y 
&= \int_{A^c} \int_{\partial C} G_{A^c}(y,z) e_{C,A}(\6z) \e^{-V(y)/\eps}\6y \\
&= - \int_{\partial C} w_A(z) \e^{-V(z)/\eps} e_{C,A}(\6z)\;.
\label{pot12}
\end{align}
The approach used in~\cite{BEGK} is to take $C$ to be a ball of radius $\eps$
centred in $x$, and to use Harnack inequalities to show that $w_A(z)\simeq
w_A(x)$ on $C$. It then follows from~\eqref{pot11} that 
\begin{equation}
 \label{pot13}
\int_{A^c} h_{C,A}(y)\e^{-V(y)/\eps}\6y 
\simeq w_A(x) \capacity_C(A)\;. 
\end{equation} 
The left-hand side can be estimated using a priori bounds on the equilibrium
potential. Thus a sufficiently precise estimate of the capacity
$\capacity_C(A)$ will yield a good estimate for
$\expecin{x}{\tau_A}=w_A(x)$. Now it follows from Green's identities that the
capacity can also be expressed as a Dirichlet form evaluated at the equilibrium potential:
\begin{equation}
 \label{pot14}
\capacity_A(B) = \eps \int_{(A\cup B)^c} \norm{\nabla h_{A,B}(x)}^2
\e^{-V(x)/\eps}\6x\;. 
\end{equation} 
Even more useful is the variational representation 
\begin{equation}
 \label{pot15}
\capacity_A(B) = \eps \inf_{h\in\cH_{A,B}}  \int_{(A\cup B)^c} \norm{\nabla
h(x)}^2 \e^{-V(x)/\eps}\6x\;,
\end{equation}
where $ \cH_{A,B}$ denotes the set of twice weakly differentiable functions
satisfying the boundary conditions in~\eqref{pot09}. Indeed, inserting a
sufficiently good guess for the equilibrium potential on the right-hand side
immediately yields a good upper bound. A matching lower bound can be obtained
by a slightly more involved argument. 

Several difficulties prevent us from applying the same strategy directly to the
infinite-dimensional equation~\eqref{SPDE}. It is possible, however, to
approximate~\eqref{SPDE} by a finite-dimensional system, using a spectral
Galerkin method, to estimate first-hitting times for the finite-dimensional
system using the above ideas, and then to pass to the limit.  


\subsection{Spectral Galerkin approximation}
\label{ssec_galerkin}

Let $P_d:E\to E$ be the projection operator defined by 
\begin{equation}
 \label{Galerkin1}
(P_du)(x) = \sum_{\abs{k}\leqs d} 
y_k e_k(x)\;,
\qquad 
y_k = y_k[u] = \int_0^L \cc{e_k(y)} u(y)\6y\;,
\end{equation} 
where the $e_k$ are the basis vectors compatible with the boundary
conditions, given by~\eqref{cob_periodic} or~\eqref{cob_Neumann}. 
We denote by $E_d$ the finite-dimensional image of $E$ under $P_d$. 
Let $u_t(x)$ be the mild solution of the SPDE~\eqref{SPDE}
and let $u^{(d)}_t(x)$ be the solution of the projected equation 
\begin{equation}
 \label{Galerkin3}
\6u^{(d)}_t(x) = P_d\bigbrak{\Delta u^{(d)}_t(x) - U'(u^{(d)}_t(x))}\6t +
\sqrt{2\eps}P_d\6W(t,x)\;, 
\end{equation} 
called Galerkin approximation of order $d$. 
It is known (see, for instance, \cite{Jetschke_86}) that~\eqref{Galerkin3} is
equivalent to the finite-dimensional system of SDEs
\begin{equation}
 \label{Galerkin4}
\6 y_k(t) = -\dpar{}{y_{k}} \Vhat (y(t))\6t + \sqrt{2\eps} \6W_k(t)\;,
\qquad
\abs{k}\leqs d\;,
\end{equation} 
where the $W_k(t)$ are independent standard Brownian motions, and 
the potential is given by  
\begin{equation}
 \label{Galerkin4V}
\Vhat(y) = V\Biggbrak{\sum_{\abs{k}\leqs d} y_k e_k}\;. 
\end{equation} 
We will need an estimate of the deviation of the Galerkin
approximation $u_t^{(d)}$ from $u_t$. Such estimates are available in the
numerical
analysis literature. For instance, \cite{Liu_CMS_2003} provides an estimate for
the Sobolev norm $\normH{r}{u}^2 = \sum_k (1+k^2)^r \abs{y_k}^2$,
with $r<1/2$. We shall use the more precise results
in~\cite{Blomker_Jentzen_09}, which allow for a control in the (stronger)
sup norm. Namely, we have the following result:

\begin{theorem}
\label{thm_Galerkin}
Fix a $T>0$. Let $U'$ be locally Lipschitz, and assume 
\begin{equation}
 \label{Galerkin5}
\sup_{d\in\N} \sup_{0\leqs t\leqs T} \norm{u_t^{(d)}(\omega)}_{L^\infty} <
\infty
\end{equation}  
for all $\omega\in\Omega$. Then, for any $\gamma\in(0,1/2)$, there exists an almost 
surely finite random variable $Z:\Omega\to\R_+$ such that 
\begin{equation}
 \label{Galerkin6}
 \sup_{0\leqs t\leqs T} \norm{u_t(\omega) - u_t^{(d)}(\omega)}_{L^\infty}
\leqs Z(\omega) d^{-\gamma}
\end{equation} 
for all $\omega\in\Omega$.
\end{theorem}
\begin{proof}
The result follows directly from~\cite[Theorem 3.1]{Blomker_Jentzen_09},
provided we verify the validity of four assumptions given in~\cite[Section 
2]{Blomker_Jentzen_09}. 
\begin{itemiz}
\item	Assumption 1 concerns the regularity of the semigroup $\e^{\Delta t}$
associated with the heat kernel, and is satisfied as shown in~\cite[Lemma 
4.1]{Blomker_Jentzen_09}.
\item	Assumption 2 is the local Lipschitz condition on $U'$. 
\item	Assumption 3 concerns the deviation of $P_d W(t,x)$ from $W(t,x)$ and
is satisfied according to~\cite[Proposition 4.2]{Blomker_Jentzen_09}.
\item	Assumption 4 is~\eqref{Galerkin5}. 
\qed
\end{itemiz}
\renewcommand{\qed}{}
\end{proof}


\subsection{Proof of the main result}
\label{ssec_proofmain}

For $r, \rho>0$ sufficiently small constants we define the balls 
\begin{align}
A =A(r) &=  \setsuch{u\in E}{\norm{u-u^*_-}_{L^{\infty}}\leqs r}\;, \\
B = B(\rho) &=  \setsuch{u\in E}{\norm{u-u^*_+}_{L^{\infty}}\leqs \rho}\;.
 \label{proofmain01}
\end{align} 
If $u^*_{\text{ts}}$ stands for a transition state between $u^*_-$ and
$u^*_+$, we denote by
\begin{equation}
 \label{proofmain00}
H_0 = V[u^*_{\text{ts}}] - V[u^*_-] 
\end{equation} 
the communication height from $u^*_-$ to $u^*_+$.
We fix an initial condition $u_0 \in A$, and write $u_0^{(d)}=P_du_0$ for its
projection on the finite-dimensional space $E_d$.  Finally we set $A_d=A\cap
E_d$. Consider the first-hitting times 
\begin{align}
\nonumber
\tau^{(d)}_B &= \inf\setsuch{t>0}{u^{(d)}_t\in B}\;,\\
\tau_B &= \inf\setsuch{t>0}{u_t\in B}\;.
\label{infd7}
\end{align}
We first need some a priori bounds on moments of these hitting times. They
are stated in the following result, which is proved in Section~\ref{sec_ld}. 

\begin{prop}[A priori bound on moments of hitting times]
\label{prop_apriori_hit} 
For any $\eta>0$, there exist constants $\eps_{0}=\eps_0(\eta), T_{1}=T_1(\eta), H_1 > 0$ such
that for all $\eps\in(0,\eps_0)$, there exists a $d_{0}(\eps)>0$ such that \begin{equation}
 \label{infd9} 
\sup_{v_0\in A}
\expecin{v_0}{\tau_B^2} \leqs T_1^2\e^{2(H_0+\eta)/\eps}
\qquad\text{and}\qquad
\sup_{d\geqs d_{0}}\sup_{v_0\in A}
\bigexpecin{v_0^{(d)}}{(\tau^{(d)}_{B})^2} \leqs T_1^2\e^{2H_1/\eps}\;.
\end{equation}
\end{prop}

The next result applies to all finite-dimensional Galerkin approximations, and
is based on the potential-theoretic approach. The detailed proof is given 
in Sections~\ref{sec_cap} and~\ref{sec_times}.

\begin{prop}[Bounds on expected hitting times in finite dimension] 
\label{prop_bounds_finite} 
There exists $\eps_{0}>0$ such that for any $0<\eps<\eps_{0}$ there exists a 
$d_0=d_{0}(\eps)<\infty$ such that for all $d\geqs d_0$,
there exists a probability measure $\nu_{d,B}$ supported on $\partial A_d$ such
that 
\begin{equation}
 \label{infd8}
C(d,\eps) \e^{H(d)/\eps} \bigbrak{1-R^-_{d,B}(\eps)}
\leqs \int_{\partial A_d}
\bigexpecin{v_0}{\tau^{(d)}_B} \nu_{d,B}(\6v_0) 
\leqs 
C(d,\eps) \e^{H(d)/\eps} \bigbrak{1+R^+_{d,B}(\eps)}\;,
\end{equation} 
where the quantities $C(d,\eps)$, $H(d)$ and $R^\pm_{d,B}(\eps)$ are explicitly
known. 
They satisfy 
\begin{itemiz}
\item	$\lim_{d\to\infty}C(d,\eps)\bydef C(\infty,\eps)$ exists and is finite;
\item	$\lim_{d\to\infty}H(d) = H_0$ is given by the communication height;
\item	the remainders $R^\pm_{d,B}(\eps)$ are uniformly bounded in $d$ and 
$R^\pm_B(\eps)=\sup_d R^\pm_{d,B}(\eps)$ satisfies
$\lim_{\eps\to0}\smash{R^\pm_B(\eps)}=0$.
\end{itemiz}
\end{prop}

Then we have the following result.

\begin{prop}[Averaged bounds on the expected first-hitting time in infinite
dimension]
\label{prop_infd}
Pick a $\delta \in (0,\rho)$. 
There exist $\eps_0>0$ and probability measures $\nu_+$ and $\nu_-$ on 
$\partial A$ such that for $0<\eps<\eps_0$, 
\begin{align}
\nonumber
\int_{\partial A} \bigexpecin{v_0}{\tau_{B(\rho)}} \nu_+(\6v_0) 
&\leqs 
C(\infty,\eps) \e^{H_0/\eps} \bigbrak{1+2R^+_{B(\rho-\delta)}(\eps)}\;, \\
\int_{\partial A} \bigexpecin{v_0}{\tau_{B(\rho)}} \nu_-(\6v_0) 
&\geqs 
C(\infty,\eps) \e^{H_0/\eps} \bigbrak{1-2R^-_{B(\rho+\delta)}(\eps)}\;.
 \label{infd10}
\end{align}
\end{prop}
\begin{proof}
To ease notation, we write $B=B(\rho)$, $B_\pm=B(\rho\pm\delta)$ and 
$\TKr = C(\infty,\eps) \e^{H_0/\eps}$. For given $v_0\in\partial
A$ and $K>0$, define the
event
\begin{equation}
 \label{infd10:1}
\Omega_{K,d} \defby \biggset{\sup_{t\in[0,K\TKr]}
\norm{v_t-v^{(d)}_t}_{L^\infty} \leqs
\delta,\, \tau^{(d)}_{B_-}\leqs K\TKr}\;,
\end{equation} 
where $v_t$ and $v_t^{(d)}$ denote the solutions of the original and the
projected
equation with respective initial conditions $v_0$ and $P_d v_0$. 
Theorem~\ref{thm_Galerkin} and Markov's inequality imply 
\begin{equation}
 \label{infd10:2}
\fP(\Omega_{K,d}^c) \leqs 
\prob{Z>\delta d^\gamma}
+ \frac{\bigexpecin{v_0^{(d)}}{\tau^{(d)}_{B_-}}}{K\TKr}\;.
\end{equation} 
Choosing $\eps_{0}$ and $d_0(\eps)$ such that
Proposition~\ref{prop_apriori_hit} applies, the last summand can be bounded  
using~\eqref{infd9} and the Cauchy--Schwarz inequality. This yields 
\begin{equation}
 \label{infd10:2b}
\limsup_{d\to\infty} \fP(\Omega_{K,d}^c) \leqs \frac{M(\eps)}K
\qquad\text{where } 
M(\eps) = \frac{T_1}{C(\infty,\eps)}\e^{(H_1-H_0)/\eps}\;.
\end{equation} 
We decompose 
\begin{equation}
 \label{infd10:3}
\bigexpecin{v_0}{\tau_B} 
=  \bigexpecin{v_0}{\tau_B\indicator{\Omega_{K,d}}} + 
\bigexpecin{v_0}{\tau_B\indicator{\Omega_{K,d}^c}}\;.
\end{equation} 
In order to estimate the first summand, we note that by definition of
$B$, $B_+$ and $B_-$,  
\begin{equation}
 \label{infd10:4}
 \tau^{(d)}_{B_+} \leqs \tau_B \leqs
\tau^{(d)}_{B_-}
\qquad
\text{on $\Omega_{K,d}$\;.}
\end{equation} 
It follows that 
\begin{align}
\nonumber
 \bigexpecin{v_0^{(d)}}{\tau^{(d)}_{B_+}} -
\bigexpecin{v_0^{(d)}}{\tau^{(d)}_{B_+}\indicator{\Omega_{K,d}^c}}
&= \bigexpecin{v_0^{(d)}}{\tau^{(d)}_{B_+}\indicator{\Omega_{K,d}}}
\\
\nonumber
&\leqs \bigexpecin{v_0}{\tau_B\indicator{\Omega_{K,d}}} \\
&\leqs
\bigexpecin{v_0^{(d)}}{\tau^{(d)}_{B_-}\indicator{\Omega_{K,d}}} 
\leqs \bigexpecin{v_0^{(d)}}{\tau^{(d)}_{B_-}}\;.
 \label{infd10:5}
\end{align} 
The second summand in~\eqref{infd10:3} can be bounded by Cauchy--Schwarz:
\begin{equation}
 \label{infd10:6}
 0 \leqs \bigexpecin{v_0}{\tau_B\indicator{\Omega_{K,d}^c}} 
\leqs \sqrt{\expecin{v_0}{\tau_B^2}} \sqrt{\fP(\Omega_{K,d}^c)}\;.
\end{equation} 
This shows that
\begin{equation}
 \label{infd10:10}
\bigexpecin{v_0^{(d)}}{\tau^{(d)}_{B_+}} - 
\sqrt{\expecin{v_0}{\bigpar{\tau^{(d)}_{B_+}}^2} \fP(\Omega_{K,d}^c)}
\leqs \bigexpecin{v_0}{\tau_B} \leqs 
\bigexpecin{v_0^{(d)}}{\tau^{(d)}_{B_-}} + 
\sqrt{\expecin{v_0}{\tau_B^2} \fP(\Omega_{K,d}^c)}\;.
\end{equation} 
Proposition~\ref{prop_bounds_finite} shows that 
\begin{equation}
 \label{infd10:11}
\limsup_{d\to\infty} 
\int_{\partial A_d}  \bigexpecin{v_0}{\tau^{(d)}_{B_-}}
\nu_{d,B_-}(\6v_0) \leqs \TKr \bigbrak{1 + R^+_{B_-}(\eps)}\;,
\end{equation} 
while Lebesgues's dominated convergence theorem and~\eqref{infd10:2b} yield 
\begin{equation}
 \label{infd10:12}
\limsup_{d\to\infty} 
\int_{\partial A_d}  \sqrt{\expecin{v_0}{\tau_B^2} \fP(\Omega_{K,d}^c)}
\,\nu_{d,B_-}(\6v_0) 
\leqs 
\sqrt{\sup_{v_0\in A}\expecin{v_0}{\tau_B^2} \frac{M(\eps)}K }\;.
\end{equation} 
Inserting~\eqref{infd10:11} and~\eqref{infd10:12} in~\eqref{infd10:10} shows
that 
\begin{equation}
\limsup_{d\to\infty} 
\frac{1}{\TKr} \int_{\partial A_d} \expecin{v_0}{\tau_B} \nu_{d,B_-}(\6v_0)
\leqs 
1+R^+_{B_-}(\eps) + \sqrt{\frac{\sup_{v_0\in A}\expecin{v_0}{\tau_B^2}}{\TKr^2}
\frac{M(\eps)}K}\;.
 \label{infd10:7}
\end{equation} 
Taking $K$ sufficiently large, the third summand can be made smaller than the
second one. The upper bound in~\eqref{infd10} then follows with
$\nu_+=\nu_{d,B_-}$ for $d$ sufficiently large. The proof of the lower bound is
analogous. 
\end{proof}

To finish the proof of the main result, we need 

\begin{theorem}
\label{thm_MOS}
There exist constants $\eps_0, \kappa, t_0, m, c>0$ such that 
such that for all $\eps < \eps_0$, 
\begin{equation}
 \label{infd11}
\biggprob{\sup_{u_0,v_0\in A}
\frac{\norm{u_t-v_t}_{L^\infty}}{\norm{u_0-v_0}_{L^\infty}} < c\e^{-mt} 
\; \forall t>t_0} \geqs 1 - \e^{-\kappa/\eps}\;. 
\end{equation}  
Here $u_t$ and $v_t$ denote the mild solutions of the
SPDE~\eqref{SPDE} with respective initial conditions $u_0$ and $v_0$. 
\end{theorem}

This result has been proved 
in~\cite[Corollary 3.1]{Martinelli_Olivieri_Scoppola_89} in the case of
Dirichlet b.c., under the condition $L\neq\pi$. The reason for this restriction is
that for $L=\pi$, the Hessian at the potential minimum has a zero eigenvalue. 
For Neumann and periodic b.c., this difficulty does not occur, because the
potential minima always have only strictly negative eigenvalues. 

\begin{prop}[Main result]
\label{prop_main_result} 
Pick $\delta\in(0,\rho/2)$. 
There exists $\eps_0>0$ such that for all $\eps<\eps_0$, 
\begin{equation}
 \label{infd12}
C(\infty,\eps) \e^{H_0/\eps} \bigbrak{1-3R^-_{B(\rho+2\delta)}(\eps)}
\leqs \bigexpecin{u_0}{\tau_B(\rho)} \leqs 
C(\infty,\eps) \e^{H_0/\eps} \bigbrak{1+3R^+_{B(\rho-2\delta)}(\eps)}\;. 
\end{equation} 
\end{prop}
\begin{proof}
As before we write $B=B(\rho)$, $B_\pm=B(\rho\pm\delta)$. 
Given a constant $T\geqs t_0$, consider the event 
\begin{equation}
 \label{infd14:01}
\Omega_T = \biggset{\tau_{B_{+}}\geqs T, \sup_{v_0\in A} 
\frac{\norm{v_t-u_t}_{L^\infty}}{\norm{v_0-u_0}_{L^\infty}}
< c\e^{-mt} \;\forall t\geqs T}\;.
\end{equation} 
Then Theorem~\ref{thm_MOS} and the standard large-deviation estimate Corollary~\ref{cor_ldp_Vbar} show that 
for any $T>0$, there exist constants $\eps_0,\kappa_1>0$ such that for
$\eps<\eps_0$, 
\begin{equation}
 \label{infd14:02}
\fP(\Omega_T^c) \leqs \e^{-\kappa_1/\eps}\;.
\end{equation} 
Note that for all $v_0\in A$, 
\begin{equation}
 \label{infd14:03}
\tau^{v_0}_{B_+} \leqs \tau^{u_0}_B \leqs \tau^{v_0}_{B_-}
\qquad \text{on $\Omega_T$\;,}
\end{equation} 
provided $T$ is large enough that $rc\e^{-mT} \leqs \delta/2$.
In order to prove the upper bound, we start by observing that 
\begin{equation}
 \label{infd14:04}
\bigexpecin{u_0}{\tau_B \indicator{\Omega_T}}
=  \bigexpecin{u_0}{\tau_B \indicator{\Omega_T}}
\int_{\partial A} \nu_+(\6v_0)
\leqs \int_{\partial A} \bigexpecin{v_0}{\tau_{B_-}}\nu_+(\6v_0)\;,
\end{equation} 
which can be bounded above with Proposition~\ref{prop_infd}. 
Furthermore, by Cauchy--Schwarz, we have 
\begin{equation}
 \label{infd14:05}
\bigexpecin{u_0}{\tau_B \indicator{\Omega_T^c}}
\leqs \sqrt{\expecin{u_0}{\tau_B^2}} \sqrt{\fP(\Omega_T^c)}
\leqs T_1 \e^{(H_0+\eta-\kappa_1/2)/\eps}\;.
\end{equation} 
For the lower bound, we use the decomposition
\begin{align}
\nonumber
\bigexpecin{u_0}{\tau_B} 
&\geqs \bigexpecin{u_0}{\tau_B\indicator{\Omega_T}}
\int_{\partial A} \nu_-(\6v_0) \\
&\geqs \int_{\partial A}
\bigexpecin{v_0}{\tau_{B_+}}\nu_-(\6v_0) - \int_{\partial A}
\bigexpecin{v_0}{\tau_{B_+}\indicator{\Omega_T^c}}\nu_-(\6v_0)\;.
 \label{infd14:06}
\end{align}  
The first term on the right-hand side can be bounded below with
Proposition~\ref{prop_infd}, while the second one is bounded above by 
\begin{equation}
 \label{infd14:07}
\sqrt{\sup_{v_0\in \partial A}\bigexpecin{v_0}{\tau_{B_+}^2}} 
\sqrt{\fP(\Omega_T^c)} \leqs T_1 \e^{(H_0+\eta-\kappa_1/2)/\eps}\;.
\end{equation} 
This concludes the proof, provided we choose $\eta < \kappa_1/2$ when applying
Proposition~\ref{prop_infd}.
\end{proof}




\newpage

\section{Deterministic system}
\label{sec_det}


This section gathers a number of needed results on the deterministic
partial differential equation~\eqref{PDE}. Some general properties of
the equation are discussed e.g.\ in~\cite{ChafeeInfante_74,Jolly_89}. 

In Section~\ref{ssec_function} we introduce various function spaces and
inequalities required in the analysis. In Section~\ref{ssec_bounds_pot},
we establish some general bounds on the potential energy $V$ and its
derivative. Section~\ref{ssec_normalform} analyses the behaviour of the
potential energy at bifurcation points, and Section~\ref{ssec_tpot}
contains a result on the relation between $V$ and its restrictions to
finite-dimensional subspaces.  


\subsection{Function spaces}
\label{ssec_function} 

We introduce two scales of function spaces that will play a r\^ole in the
sequel. Let $I$ denote either a compact interval $[0,L]\subset\R$ or the circle
$\T^1=\R/(2\pi\Z)$.

We denote by $\cC^{0} = \cC^{0}(I)$ the space of all continuous functions
$u:I\to\R$. Note that $I$ is compact so that the functions from $\cC^{0}$ are bounded. 
When equipped with the sup norm 
$\norm{u}_{\cC^0} = \sup_{x\in I} \abs{u(x)}$, $\cC^0$ is a Banach space.
For $\alpha>0$, we define the H\"older seminorm 
\begin{equation}
 \label{function01}
[u]_\alpha = \sup_{x\neq y} \frac{\abs{u(x)-u(y)}}{\abs{x-y}^\alpha}\;,
\end{equation} 
the H\"older norm $\norm{u}_{\cC^\alpha} = \norm{u}_{\cC^0} +
[u]_\alpha$, and write $\cC^\alpha =
\setsuch{u\in\cC^{0}}{\norm{u}_{\cC^\alpha}<\infty}$ for the associated Banach
space. 

For $1\leqs p\leqs \infty$, $L^p = L^p(I)$ denotes the space of all $u:I\to\R$
with bounded $L^p$-norm. Note that for $u\in \cC^0$,
$\norm{u}_{L^\infty}=\norm{u}_{\cC^0}$. When $u\in L^2(\T^1)$, we write its
Fourier series as
\begin{equation}
 \label{function02}
u(x) = \sum_{k\in\Z} y_k e_k(x)\;, 
\qquad
e_k(x) =\frac{ \e^{\icx k x}}{\sqrt{2\pi}}\;.
\end{equation} 
For $s\geqs 0$, we define the Sobolev norm 
\begin{equation}
 \label{function03}
\norm{u}_{H^s}^2 \equiv \norm{y}_{H^s}^2
= \sum_{k\in\Z} (1+k^2)^s \abs{y_k}^2\;,
\end{equation} 
and denote by $H^s=H^s(\T^1)=\setsuch{u\in L^2(\T^1)}{\norm{u}_{H^s}<\infty}$
the fractional Sobolev space (also called Bessel
potential space). Note
that $H^s$ is a 
Hilbert space, and $H^0=L^2$. 
The norm $\norm{u}_{H^1}$ can be equivalently defined by 
$\norm{u}_{H^1}^2 = \norm{u}_{L^2}^2 + \norm{u'}_{L^2}^2$, where $u'$ is the
weak derivative of $u$.

\begin{lemma}
 \label{lemma_Morrey}
For any $\alpha\geqs0$ and $s>\alpha+1/2$, there exists a constant
$C=C(\alpha,s)$ such that 
\begin{equation}
 \label{function04}
\norm{u}_{\cC^\alpha} \leqs C \norm{u}_{H^s} 
\qquad
\forall u\in H^s(\T^1)\;. 
\end{equation}  
As a consequence, we have $H^s(\T^1) \subset \cC^\alpha(\T^1)$. 
\end{lemma}

\begin{remark}
 \label{rem_Morrey} 
In the particular case $s=1$, \eqref{function04} can be strengthened to
Morrey's inequality 
\begin{equation}
 \label{function05}
 \norm{u}_{\cC^{1/2}} \leqs C \norm{u}_{H^1} 
\qquad
\forall u\in H^1(\T^1)\;.
\end{equation} 
\end{remark}

Let $p, q$ satisfy $1\leqs p\leqs 2 \leqs q\leqs\infty$ and
$\frac1p+\frac1q=1$. Then the Hausdorff--Young
inequalities~\cite{Dunford_Schwartz_I} state that there exist constants $C_1(p)$
and $C_2(p)$ such that 
\begin{equation}
 \label{norm04}
\norm{u}_{L^q} \leqs C_1\norm{y}_{\ell^p} 
\qquad\text{and}\qquad
\norm{y}_{\ell^q} \leqs C_2\norm{u}_{L^p}\;.
\end{equation} 
We consider now some properties of convolutions $y*z$ defined by 
\begin{equation}
 \label{norm01}
(y*z)_k = \sum_{l\in\Z} y_l z_{k-l}\;. 
\end{equation} 
Young's inequality states that for $1\leqs p,q,r\leqs\infty$ such that 
$\frac1p+ \frac1q = \frac1r + 1$, 
\begin{equation}
 \label{norm03}
\norm{y*z}_{\ell^r} \leqs \norm{y}_{\ell^p} \norm{z}_{\ell^q}\;. 
\end{equation} 

\begin{lemma}
\label{lem_Saitoh_B}
Let $r, s, t \in (0,1/2)$ be such that $t < r+s-1/2$. Then 
there exists a constant $C=C(r,s,t)$ such that 
\begin{equation}
 \label{norm20}
\norm{y*z}_{H^t} \leqs C \norm{y}_{H^r} \norm{z}_{H^s}\;. 
\end{equation} 
\end{lemma}
\begin{proof}
Define $w_k$ by 
\begin{equation}
 \label{norm21:01}
\frac1{w_k} = \sum_{l\in Z} \frac{1}{(1+l^2)^r} \frac{1}{(1+(k-l)^2)^s}\;. 
\end{equation} 
Splitting the sum at $-\abs{k}, 0, \abs{k}/2, \abs{k}$ and $2\abs{k}$, and
bounding each sum by an integral, one easily shows that 
\begin{equation}
 \label{norm21:02}
\frac1{w_k} \leqs \frac{C}{(1+k^2)^t}\;. 
\end{equation} 
Let $\tilde y_k = (1+k^2)^{r/2} \abs{y_k}$ and 
$\tilde z_k = (1+k^2)^{s/2} \abs{z_k}$. 
Adapting a computation in~\cite{Saitoh_2000}, we write 
\begin{equation}
\bigabs{(y*z)_k}
\leqs \sum_{l\in\Z} \abs{y_l}\abs{z_{k-l}} 
= \sum_{l\in\Z} \frac{1}{(1+l^2)^{r/2} (1+(k-l)^2)^{s/2}}
\tilde y_l \tilde z_{k-l}\;.
 \label{norm21:03}
\end{equation} 
By the Cauchy--Schwarz inequality, 
\begin{equation}
 \label{norm21:04}
\bigpar{(y*z)_k}^2 
\leqs \frac{1}{w_k} \sum_{l\in\Z}  
\tilde y_l^2 \tilde z_{k-l}^2\;.
\end{equation} 
If $\tilde y^2, \tilde z^2$ denote the vectors with components $\tilde y_l^2$
and $\tilde z_l^2$, it follows from~\eqref{norm21:02} that 
\begin{align}
 \label{norm21:05}
\norm{y*z}_{H^t}^2 
&\leqs \sum_{k\in\Z} C \sum_{l\in\Z}  
\tilde y_l^2 \tilde z_{k-l}^2
= C \sum_{k\in\Z} (\tilde y^2*\tilde z^2)_k
= C \norm{\tilde y^2*\tilde z^2}_{\ell^1} 
\leqs C \norm{\tilde y^2}_{\ell^1}\norm{\tilde z^2}_{\ell^1}
\end{align} 
by Young's inequality, and the results follows since
$\norm{\tilde y^2}_{\ell^1}=\norm{y}_{H^r}^2$, 
$\norm{\tilde z^2}_{\ell^1}=\norm{z}_{H^s}^2$. 
\end{proof}

Finally, the following estimate allows to bound the usual $\ell^r$-norm in
terms of Sobolev norms. 

\begin{lemma}
\label{lem_weighted}
Fix $1\leqs r<2$. For any $s>1/r-1/2$, there exists a finite $C(s)$ such that 
\begin{equation}
 \label{norm12}
\norm{y}_{\ell^r} \leqs C(s) \norm{y}_{H^{s}}\;. 
\end{equation}  
\end{lemma}
\begin{proof}
Apply H\"older's inequality, with $p=2/(2-r)$ and $q=2/r$, to the decomposition 
$\abs{y_k}^r = (1+k^2)^{-rs/2} \cdot (1+k^2)^{rs/2}\abs{y_k}^r$. 
\end{proof}

By the Hausdorff--Young inequality~\eqref{norm04}, this implies the Sobolev embedding theorem
\begin{equation}
 \label{norm14}
\norm{u}_{L^p} \leqs C(s,p) \norm{y}_{H^s} 
\end{equation} 
whenever $p\geqs2$ and $s> 1/2 - 1/p$. 


\subsection{Bounds on the potential energy}
\label{ssec_bounds_pot}

In this subsection, we derive some bounds involving the potential energy 
\begin{equation}
 \label{bp01}
V[u] = \int_0^L \biggbrak{\frac12 u'(x)^2 + U(u(x))} \6x
\end{equation}
and its gradient. Periodic and Neumann boundary conditions can be treated in a
unified way by writing the Fourier series as 
\begin{equation}
 \label{bp02} 
u(x) = \sum_{k\in\Z} z_k \frac{\e^{\icx b k \pi x/L}}{\sqrt{L}}\;,
\qquad z_k=\cc{z_{-k}}\;,
\end{equation} 
where $b=1$ and $z_k=z_{-k}$ for Neumann b.c., and $b=2$ for periodic b.c.
The value of the potential expressed in Fourier variables becomes 
\begin{equation}
 \label{bp03}
\Vhat(z) = V[u(\cdot)] 
= \frac12 \sum_{k\in\Z}\nu_k \abs{z_k}^2 
+ \int_0^L U \Biggpar{\sum_{k\in\Z} z_k \frac{\e^{\icx b k \pi
x/L}}{\sqrt{L}}}\6x\;,
\end{equation} 
where $\nu_k = (b\pi k)^2/L^2$. 

\begin{lemma}[Bounds on \protect{$\Vhat$}]
\label{lem_V1} 
There exist constants $\alpha', \beta', M'_0>0$ such that 
\begin{equation}
 \label{bp05}
\beta' \norm{z}_{H^1}^2 - \alpha' \leqs 
\Vhat(z) \leqs M'_0(1+\norm{z}_{H^1}^2)^{p_0}\;. 
\end{equation} 
\end{lemma}
\begin{proof}
By Assumption~\ref{assump_U}~(U3), we have 
\begin{equation}
 \label{bp06:1}
V[u] \leqs \frac12 \norm{u'}_{L^2}^2 + M_0(1+\norm{u}_{L^{2p_0}}^{2p_0})\;, 
\end{equation} 
where 
\begin{equation}
 \label{bp06:2}
\norm{u'}_{L^2}^2 = \frac{b^2\pi^2}{L^2} \sum_{k\in\Z} k^2 z_k^2 \leqs
\frac{b^2\pi^2}{L^2} \norm{z}_{H^1}^2\;.
\end{equation} 
By~\eqref{norm14} we have 
\begin{equation}
 \label{bp06:3}
 \norm{u}_{L^{2p_0}} \leqs C(1,2p_0)\norm{z}_{H^1}\;,
\end{equation} 
which implies the upper bound. The lower bound is obtained in a similar way,
using Assumption~\ref{assump_U}~(U4). 
\end{proof}

The gradient of $\Vhat(z)$ and the Fr\'echet derivative of $V[u]$ are related
by  
\begin{equation}
 \label{bp11}
\dpar \Vhat{z_k}(z) = \nabla_{e_k} V[u]\;,
\end{equation} 
where $e_{k}$ is defined in \eqref{cob_periodic} or \eqref{cob_Neumann}, respectively.
Thus, by~\eqref{pot_Frechet} and Parseval's identity, 
\begin{equation}
 \label{bp12}
\norm{\nabla \Vhat(z)}_{\ell^2}^2 
= \sum_{k\in\Z} \dpar{\Vhat}{z_k} \cc{\dpar{\Vhat}{z_k}} (z) 
= \int_0^L \bigbrak{-u''(x) + U'(u(x))}^2 \6x\;.
\end{equation} 

\begin{lemma}[Lower bound on \protect{$\norm{\nabla \Vhat}_{\ell^2}^2$}]
\label{lem_V2} 
For any $\rho>0$ there exists $M'_1(\rho)$ such that 
\begin{equation}
 \label{bp13}
\norm{\nabla \Vhat(z)}_{\ell^2}^2 \geqs \rho \norm{z}_{H^1}^2 -
M'_1(\rho)\;.
\end{equation} 
\end{lemma}
\begin{proof}
We expand the square in~\eqref{bp12} and evaluate the terms separately. Using
Assumptions~\ref{assump_U}~(U5) und (U6) and integration by parts, we have
for any $\gamma>0$ 
\begin{align}
\nonumber
\int_0^L u''(x)^2 \6x 
&= \sum_{k\in\Z} \frac{b^4\pi^4}{L^4} k^4 \abs{z_k}^2 \;, \\
\nonumber
\int_0^L U'(u(x))^2 \6x
&\geqs \gamma \sum_{k\in\Z} \abs{z_k}^2 - LM_1(\gamma)\;, \\
-2\int_0^L u''(x) U'(u(x)) \6x 
&= 2 \int_0^L u'(x)^2 U''(u(x)) \6x
\geqs -2M_2 \sum_{k\in\Z} \frac{b^2\pi^2}{L^2} k^2 \abs{z_k}^2\;,
 \label{bp14:02}
\end{align} 
so that 
\begin{equation}
 \label{bp14:03}
\norm{\nabla \V(z)}_{\ell^2}^2 \geqs \sum_{k\in\Z} \biggbrak{\frac{b^4\pi^4}{L^4}
k^4 -2M_2 \frac{b^2\pi^2}{L^2} k^2 + \gamma} \abs{z_k}^2 - M_1(\gamma)\;. 
\end{equation} 
For any $\rho>0$, we can find a $\gamma$ such that the term in brackets is
bounded below by $2\rho(1+k^2)$, uniformly in $k$.
This proves the result.
\end{proof}

\begin{cor}
\label{cor_V1}
For any $\delta>0$, there exists $H=H(\delta)$ such that 
$\norm{\nabla \V(z)}_{\ell^2}^2 \geqs \delta^2$ whenever $\V(z)\geqs H$. 
As a consequence, all stationary points of $\V$ belong to 
$\setsuch{z}{\V(z)\leqs H(0)}$. 
\end{cor}
\begin{proof}
This immediately follows from 
$\norm{\nabla \V(z)}_{\ell^2}^2 \geqs \alpha'[(\V(z)/M_0')^{1/p_{0}}-1]-M'_1(\alpha')$.
\end{proof}


\subsection{Normal forms}
\label{ssec_normalform}

We will rely on normal forms when analysing the system for $L$ near a critical value. In this situation we will always assume that the local potential $U$ is in $\cC^{5}$, so that we can write its Taylor expansion as 
\begin{equation}
 \label{nf02}
U(u) = -\frac12 u^2 + \frac{a_3\sqrt{L}}{3} u^3 + \frac{a_4 L}{4} u^4 +
\Order{u^5}\;. 
\end{equation} 
Then for small $u$, the potential energy admits the expansion 
\begin{equation}
 \label{nf03}
V[u] = \frac 12 \norm{u'}_{L^2}^2 - \frac 12 \norm{u}_{L^2}^2
+ \frac{a_3 \sqrt{L}}{3} \int_0^L u(x)^3\6x
+ \frac{a_4 L}{4} \norm{u}_{L^4}^4 
+ \Order{\norm{u}_{L^5}^5}\;.
\end{equation}
Equation~\eqref{bp03} shows that the potential energy in Fourier variables can
be decomposed as 
\begin{equation}
 \label{nf06}
\V(z) =  \V_2(z) + \V_3(z) + \V_4(z) + R(z)\;,
\end{equation} 
where the $\V_n(z)$ are given by the convolutions 
\begin{align}
\nonumber
\V_2(z) &= \frac{1}{2} \sum_{k\in\Z} \lambda_k z_k z_{-k} \;, \\
\nonumber
\V_3(z) &= \frac{a_3}{3} 
\sum_{\substack{k_1,k_2,k_3\in\Z \\ k_1+k_2+k_3=0}} z_{k_1} z_{k_2} z_{k_3} \;,
\\
\V_4(z) &= \frac{a_4}{4} 
\sum_{\substack{k_1,k_2,k_3,k_4\in\Z \\ k_1+k_2+k_3+k_4=0}} z_{k_1} z_{k_2}
z_{k_3} z_{k_4} \;, 
\label{nf07}  
\end{align}
where $\lambda_k = \nu_k-1$. It follows from~\eqref{norm14}, applied for $p=5$, 
that the remainder satisfies $R(z)= \Order{\norm{z}_{H^{s}}^5}$ for all $s>3/10$.


\begin{prop}
\label{prop_NormalForm}
Let $L$ be such that $\lambda_k$ is bounded away from $0$ for all $k\neq\pm1$.
Then there exists a map $g:\R^\Z\to\R^\Z$ such that 
\begin{equation}
 \label{nf08}
\V(z+g(z)) = \V_2(z) + C_4 z_1^2z_{-1}^2 + R_1(z)\;,
\end{equation}  
where 
\begin{equation}
 \label{nf09}
C_4 = \frac32a_4 + 2a_3^2 \biggpar{\frac{1}{\abs{\lambda_0}} -
\frac{1}{2\lambda_2}}\;, 
\end{equation} 
and the remainder satisfies 
\begin{equation}
 \label{nf10}
R_1(z) = \Order{\norm{z}_{H^s}^5}
\end{equation} 
for all $5/12 < s < 1/2$. 
Furthermore, $\norm{g(z)}_{H^t}=\Order{\norm{z}_{H^s}^2}$ for $t<2s-1/2$ 
and the Jacobian of the transformation $z\mapsto z+g(z)$ satisfies 
\begin{equation}
 \label{nf11}
\det(\one+\sdpar gz(z)) = 1 + 
\Order{a_3\norm{z}_{H^s}}
+\Order{a_4\norm{z}_{H^s}^2}
\end{equation} 
on the set $\set{z_k=z_{-k}}$.
\end{prop}
\begin{proof}
In the course of the proof, we will need Sobolev norms with indices $q, r, t$,
satisfying the relations 
$$
0<q,r<t<s<1/2\;, 
\quad
t<2s-1/2\;,
\quad
r<2t-1/2
\quad\text{and}\quad
q<3t-1\;. 
$$
This is always possible for $5/12 < s < 1/2$.
In this proof, we will denote by $C_0$ any constant appearing when
applying Lemma~\ref{lem_Saitoh_B}. Its value may change from one line to the
next one. 
\begin{enum}
\item	Let $g^{(2)}: \R^\Z\to\R^\Z$ be homogeneous of degree $2$, and satisfy
$g^{(2)}_{-k}(z)=g^{(2)}_k(z)$. Then, expanding and grouping terms of equal
order we get 
\begin{align}
\label{nf12A}
\V(z+g^{(2)}(z)) = \V_2(z) 
&+  \sum_k \lambda_k z_k g^{(2)}_{-k}(z) \!\!\!\!
 &{+}& \frac{1}{2} \sum_k
\lambda_k g^{(2)}_k(z) g^{(2)}_{-k}(z) \!\!\!\! \\
\nonumber
&+ \V_3(z) &{+}& \sum_k \dpar{\V_3}{z_k}(z) g^{(2)}_k(z) &{+}& r_1(z) \\
\nonumber
&             &{+}& \V_4(z) &{+}& r_2(z) \\
\nonumber
&             & &              &{+}& R(z+g^{(2)}(z))\;,
\end{align}
with remainders that can be written as
\begin{align}
\nonumber
r_1(z) &= \frac12 \sum_{k,l} \dpar{^2\V_3}{z_k\partial z_l}
\bigpar{z+\theta_1g^{(2)}(z)} g^{(2)}_k(z)g^{(2)}_l(z) \\
r_2(z) &= \sum_{k} \dpar{\V_4}{z_k}
\bigpar{z+\theta_2g^{(2)}(z)} g^{(2)}_k(z)
\label{nf12B} 
\end{align}
for some $\theta_1,\theta_2\in[0,1]$.

We want to choose $g^{(2)}(z)$ in such a way that the terms of order $3$ in
$\V(z+g^{(2)}(z))$ cancel. This can be achieved by taking 
\begin{equation}
 \label{nf12C}
\lambda_k g^{(2)}_k(z) \bydef \tilde g^{(2)}_k(z) = 
\begin{cases}
\vrule height 12pt depth 12pt width 0pt
0 & \text{if $\abs{k}=1\;,$} \\
\vrule height 12pt depth 12pt width 0pt
\displaystyle
-\frac{a_3}{3} \sum_{k_1+k_2=k} b_{k_1,k_2}z_{k_1}z_{k_2} 
& \text{otherwise\;,}
\end{cases} 
\end{equation} 
for appropriate coefficients $b_{k_1,k_2}$ satisfying
$b_{k_1,k_2}=b_{-k_1,-k_2}$. The choice of these coefficients is not unique, but
we can make it unique by imposing the symmetry conditions 
\begin{equation}
 \label{nf12D}
b_{k,l} = b_{l,k} = b_{l,-k-l} = b_{-k-l,l} = b_{k,-k-l} = b_{-k-l,k}\;.
\end{equation} 
Indeed, these are all the terms contributing to the monomial $z_kz_lz_{-k-l}$
in the first sum in~\eqref{nf12A}. Then simple combinatorics show that all
$b_{k,l}$ belong to the interval $[1/6,6]$. This choice has the further
advantage that on the set $\set{z_k=z_{-k}}$, 
\begin{equation}
 \label{nf12E}
\dpar{\tilde g^{(2)}_k}{z_l}(z) 
= -\frac{2a_3}{3} b_{l,k-l} z_{k-l} 
= -\frac{2a_3}{3} b_{k,l-k} z_{l-k} 
= \dpar{\tilde g^{(2)}_l}{z_k}(z)
\end{equation} 
whenever $\abs{k},\abs{l}\neq1$.

The term of order $4$ of $\V(z+g^{(2)}(z))$ is given by 
\begin{equation}
 \label{nf12F}
\Vtilde_4(z) = \V_4(z) + \frac{1}{2} \sum_k
\lambda_k g^{(2)}_k(z) g^{(2)}_{-k}(z)
+ \sum_k \dpar{\V_3}{z_k}(z) g^{(2)}_k(z)\;.
\end{equation} 
Note that the convolution structure is preserved. 
In order to show that the sums indeed converge, we first note that since 
$t < 2s - 1/2$, we have 
\begin{equation}
 \label{nf12G}
\norm{\tilde g^{(2)}(z)}_{H^t} 
\leqs 2\abs{a_3} \norm{z*z}_{H^t} 
\leqs C_0\abs{a_3} \norm{z}_{H^s}^2 
\end{equation} 
by Lemma~\ref{lem_Saitoh_B}. Since $\abs{\lambda_k}^{-1}\leqs1\leqs 
(1+k^2)^t$ $\forall k\neq\pm1$, the first sum in~\eqref{nf12F} can be
bounded by 
\begin{equation}
 \label{nf12H}
\sum_k
\frac{\tilde g^{(2)}_k(z)^2}
{\abs{\lambda_k}}
\leqs \norm{\tilde g^{(2)}(z)}_{H^t}^2
\leqs C_0a_3^2 \norm{z}_{H^s}^4\;.
\end{equation} 
The second sum in~\eqref{nf12F} can be bounded as follows:
\begin{align}
\nonumber
\biggabs{\sum_k \dpar{\V_3}{z_k}(z) g^{(2)}_k(z)}
&\leqs \sum_{k} \biggbrak{\frac{\abs{a_3}}{3}
\sum_{k_1+k_2=-k} 3 
\abs{z_{k_1}}\abs{z_{k_2}}} \bigabs{g^{(2)}_k(z)} \\ 
\nonumber
&= \abs{a_3}
\Bigbrak{\abs{z}*\abs{z}*\bigabs{g^{(2)}(z)}}_0 \\
\nonumber
&\leqs \abs{a_3}
\Bignorm{\abs{z}*\abs{z}*\bigabs{g^{(2)}(z)}}_{H^r} \\
\nonumber
&\leqs C_0\abs{a_3}\norm{z}^2_{H^s}
\bignorm{g^{(2)}(z)}_{H^t} \\
&\leqs C_0a_3^2 \norm{z}_{H^s}^4\;.
\label{nf12I}
\end{align}
This shows that $\Vtilde_4(z)$ indeed exists, and satisfies 
\begin{equation}
 \label{nf12J}
\bigabs{\Vtilde_4(z)} \leqs C_0\biggpar{\frac{a_4}{4} +
6a_3^2}\norm{z}_{H^s}^4\;.
\end{equation} 
Next we estimate the remainders. 
The remainder $r_1(z)$ can be bounded as follows: 
\begin{align}
\nonumber
\abs{r_1(z)} &\leqs \frac12 \sum_{k,l} 6 \frac{\abs{a_3}}{3}
\bigabs{z_{-k-l}+\theta_1g^{(2)}_{-k-l}(z)} 
\bigabs{g^{(2)}_k(z)} \bigabs{g^{(2)}_l(z)} \\ 
\nonumber
&\leqs \abs{a_3}
\Bigbrak{\bigabs{z+\theta_1g^{(2)}(z)}
*\bigabs{g^{(2)}(z)}*\bigabs{g^{(2)}(z)}}_0 \\
\nonumber
&\leqs \abs{a_3}
\Bignorm{\bigabs{z+\theta_1g^{(2)}(z)}*\bigabs{g^{(2)}(z)}
*\bigabs{g^{(2)}(z)}}_{H^q} \\
\nonumber
&\leqs C_0\abs{a_3}
\bignorm{z+\theta_1g^{(2)}(z)}_{H^t}
\bignorm{g^{(2)}(z)}_{H^t}^2 \\
&\leqs C_0\abs{a_3}^3 \norm{z}_{H^s}^5 
\Bigbrak{1 + 2\abs{a_3} \norm{z}_{H^s}}\;.
\label{nf12K}
\end{align}
A similar computation yields 
\begin{equation}
 \label{nf12L}
 \abs{r_2(z)} \leqs C_0 \abs{a_3a_4} \norm{z}_{H^s}^5 
\Bigbrak{1 + 8\abs{a_3}^3 \norm{z}_{H^s}^3}\;.
\end{equation} 
Finally, clearly 
\begin{equation}
 \label{nf12M}
\bigabs{R(z+g^{(2)}(z))} = \Order{\norm{z}_{H^s}^5}\;. 
\end{equation} 
We have thus obtained 
\begin{equation}
 \label{nf12N}
\V(z+g^{(2)}(z)) \bydef \Vtilde(z)
= V_2(z) + \Vtilde_4(z) +  \Order{\norm{z}_{H^s}^5}\;.
\end{equation}

\item	The Jacobian matrix of $z\mapsto z+g^{(2)}(z)$ is given by $\one+A(z)$,
where the elements of $A(z)=\sdpar{g^{(2)}}{z}(z)$ can be deduced
from~\eqref{nf12E}. By construction, $A(z)$ is self-adjoint with respect to the
scalar product weighted by the $\lambda_k$, and thus has real eigenvalues. 
Its $\ell^1$-operator norm satisfies 
\begin{equation}
 \label{nf12O}
\norm{A(z)}_{\ell^1} = \max_l \sum_k \biggabs{\dpar{g^{(2)}_k(y)}{y_l}}
\leqs C_0\abs{a_3} \max_l \sum_k \frac{\abs{z_{k-l}}}{\abs{\lambda_k}}
\leqs \const \norm{z}_{\ell^\infty}
\leqs \const \norm{z}_{H^s}\;.
\end{equation} 
Hence the spectral radius $\rho(z)$ of $A(z)$ has order
$\norm{z}_{H^s}$. Now if $\rho(z)<1$ and we denote the eigenvalues of
$A(z)$ by $a_k(z)$, 
\begin{equation}
 \label{nf12P}
\log\det(\one+A(z)) = \sum_k \log(1+a_k(z)) \leqs \sum_k a_k(z)
=\Tr A(z)\;.
\end{equation} 
It follows that $\abs{\det(\one+A(z))} \leqs \e^{\Tr(A(z))}$, and one easily
shows that $\Tr(A(z))=\Order{a_3 \abs{z_0}}\leqs\Order{a_3
\norm{z}_{H^s}}$. A matching lower bound can be obtained in a similar
way. This proves that the Jacobian of the transformation $z\mapsto
z+g^{(2)}(z)$ is $1+\Order{a_3\norm{z}_{H^s}}$.

\item	Let $g^{(3)}: \R^\Z\to\R^\Z$ be homogeneous of degree $3$, and satisfy
$g^{(3)}_{-k}(z)=g^{(3)}_k(z)$. We choose it of the form
\begin{equation}
 \label{nf12Q}
\lambda_k g^{(3)}_k(z) \bydef \tilde g^{(3)}_k(z) = 
\begin{cases}
\vrule height 12pt depth 12pt width 0pt
0 & \text{if $\abs{k}=1\;,$} \\
\vrule height 12pt depth 12pt width 0pt
\displaystyle
\sum_{k_1+k_2+k_3=k} b_{k_1,k_2,k_3}z_{k_1}z_{k_2}z_{k_3} 
& \text{otherwise\;,}
\end{cases} 
\end{equation} 
where the coefficients $b_{k_1,k_2,k_3}$ are invariant under permutations of
$k_1$, $k_2$ and $k_3$. In addition, we require invariance under sign change
and 
\begin{equation}
 \label{nf12R}
b_{k_1,k_2,k_3} = b_{k_1,k_2,-k_1-k_2-k_3}\;. 
\end{equation}
This guarantees in particular that 
\begin{equation}
 \label{nf12S}
\dpar{\tilde g^{(3)}_k}{z_l}(z) = \dpar{\tilde g^{(3)}_l}{z_k}(z)
\end{equation} 
holds on the set $\set{z_k=z_{-k}}$. 
The coefficients $b_{k_1,k_2,k_3}$ can now be chosen in such a way that 
\begin{equation}
 \label{nf12T}
 \sum_k z_k  \tilde g^{(3)}_k(z) + \tilde V_4(z)
\end{equation} 
contains only one term, proportional to $z_{-1}^2z_1^2$, which 
cannot be eliminated because $\smash{\tilde
g^{(3)}_1(z)=\tilde g^{(3)}_{-1}(z)}=0$. It follows that 
\begin{equation}
 \label{nf12U}
\Vtilde(z+ g^{(3)}(z)) = \V_2(z) + C_4 z_{-1}^2z_1^2 + R_1(z)
\end{equation} 
for some constant $C_4$. Along the lines of the above calculations, one checks
that $R_1(z)=\Order{\norm{z}_{H^s}^5}$, and that the Jacobian of the
transformation $z\mapsto z+g^{(3)}(z)$ is
\begin{equation}
\label{nf12V} 
\det\bigpar{\one+\sdpar{g^{(3)}}{z}(z)} = 
1+\bigOrder{(a_4+a_3^2)\norm{z}_{H^s}^2}\;.
\end{equation}
This proves~\eqref{nf10} and~\eqref{nf11}.

\item	It remains to compute the coefficient $C_4$ of the resonant term. 
To do this, it is sufficient to compute the terms containing $z_{\pm 1}$ of
$\smash{\tilde g^{(2)}_0}(z)$ and $\smash{\tilde g^{(2)}_{\pm2}}(z)$, which are
the only ones contributing to the resonant term. One finds 
\begin{align}
\nonumber
\tilde g^{(2)}_0(\dots,0,z_{-1},0,z_1,0,\dots) &= -2a_3 z_1z_{-1}\;, \\
\nonumber
\tilde g^{(2)}_2(\dots,0,z_{-1},0,z_1,0,\dots) &= -a_3 z_{-1}^2\;, \\
\tilde g^{(2)}_{-2}(\dots,0,z_{-1},0,z_1,0,\dots) &= -a_3 z_1^2\;, 
\label{nf12W} 
\end{align}
and substituting in~\eqref{nf12F} yields the result.
\qed
\end{enum}
\renewcommand{\qed}{}
\end{proof}


This result has important consequences for the behaviour of the potential near
bifurcation points. In the case of Neumann b.c., $\lambda_0=-1$ and
$\lambda_2=(4\pi^2/L^2)-1$. Thus the coefficient $C_4$ of the term
$z_1^2z_{-1}^2$ is given by 
\begin{equation}
 \label{nf20}
C_4(L) = \frac32 a_4 + \frac{8\pi^2-3L^2}{4\pi^2-L^2}a_3^2
=  \frac1{4L} \Bigbrak{ U^{(4)}(0) 
+ \frac{8\pi^2 - 3L^2}{4\pi^2 - L^2} U'''(0)^2}\;.
\end{equation} 
In particular, at the bifurcation point we have 
\begin{equation}
 \label{nf21}
C_4(\pi) = \frac32 a_4 + \frac53a_3^2
=  \frac1{4L} \Bigbrak{ U^{(4)}(0) 
+ \frac53 U'''(0)^2}\;.
\end{equation} 
The expression~\eqref{nf08} for the normal form shows that 
if $C_4(\pi)>0$, the system undergoes a supercritical pitchfork bifurcation at
$L=\pi$. This means that the origin is an isolated stationary point if $L<\pi$,
while for $L>\pi$ two new stationary points appear at a distance of order
$\sqrt{L-\pi}$ from the origin. They correspond to the functions we denoted
$u^*_{1,\pm}$. As a consequence, the period $T(E)$ defined in~\eqref{ODE04}
must grow for small positive $E$, to be compatible with the existence of
nonconstant stationary solutions for $L>\pi$. An analysis of the Hessian
matrices of $\Vhat$ at $u^*_{1,\pm}$ shows that they have one negative
eigenvalue for $L$ slightly larger than $\pi$. This must remain true for all
$L>\pi$ because we know that the stationary solutions $u^*_{1,\pm}$ remain
isolated when $L$ grows. 

In the case of periodic b.c., $\lambda_0=-1$ and
$\lambda_2=(16\pi^2/L^2)-1$. Thus the coefficient $C_4$ of the term
$z_1^2z_{-1}^2$ is given by 
\begin{equation}
 \label{nf22}
C_4(L) = \frac32 a_4 + \frac{32\pi^2-3L^2}{16\pi^2-L^2}a_3^2
=  \frac1{4L} \Bigbrak{ U^{(4)}(0) 
+ \frac{32\pi^2 - 3L^2}{16\pi^2 - L^2} U'''(0)^2}\;.
\end{equation} 
The value $C_4(2\pi)$ at the bifurcation point is equal to the
value~\eqref{nf21} of $C_4(\pi)$ for Neumann b.c. Thus the condition on the
bifurcation being supercritical is exactly the same as before. 
The difference is that instead of being equal, $z_1$ and $z_{-1}$ are only
complex conjugate, and thus the centre manifold at the bifurcation point is
two-dimensional. The invariance of the potential under translations
$u(x)\mapsto u(x+\ph)$ for any $\ph\in\R$ implies that $\Vhat(z)$ is invariant
under $z_k\mapsto \e^{\icx k \ph 2\pi/L}z_k$. This and the
expression~\eqref{nf08} for the normal form show that for $L>2\pi$, there
is a closed curve of stationary solutions at distance of order $\sqrt{L-2\pi}$
from the origin. It corresponds to the family of solutions we denoted
$u^*_{1,\ph}$. An analysis of the Hessian of $\Vhat$ at any $u^*_{1,\ph}$ shows
that it has one negative and one vanishing eigenvalue (due to translation
symmetry). 

Finally note that a similar normal-form analysis can be made for the other
bifurcations, at subsequent multiples of $\pi$ or $2\pi$. We do not detail this
analysis, since only saddles with one negative eigenvalue are important for
metastable transition times.


\subsection{The truncated potential}
\label{ssec_tpot}

Let $\Vhat^{(d)}$ be the restriction of the potential $\Vhat$ to the subspace
of Fourier modes $z_k$ such that $\abs{k}\leqs d$. For given $d$, let us write 
$z=(v,w)$, where $v$ is the vector of Fourier components with $\abs{k}\leqs d$
and $w$ contains the vector of remaining components. Then 
\begin{equation}
 \label{tpot01}
\Vhat^{(d)}(v) = \Vhat(v,0)\;.
\end{equation} 

\begin{prop}
\label{prop_tpot}
There exists $d_0<\infty$ such that for $d\geqs d_0$, the potentials
$\Vhat^{(d)}$ and $\Vhat$ have the same number of nondegenerate critical
points, and with the same number of negative eigenvalues.
\end{prop}
\begin{proof}
A critical point $(v^*,w^*)$ of $\Vhat$ has to satisfy the conditions 
\begin{equation}
 \label{tpot02:1}
\sdpar{\Vhat}{v}(v^*,w^*) = 0\;, 
\qquad
\sdpar{\Vhat}{w}(v^*,w^*) = 0\;,
\end{equation} 
while a critical points $v_*$ of $\Vhat^{(d)}$ has to satisfy 
\begin{equation}
 \label{tpot02:2}
\sdpar{\Vhat}{v}(v_*,0) = 0\;.
\end{equation} 
Lemma~\ref{lem_V2} implies that all critical points of $\Vhat$
have an $H^1$-norm bounded by some constant $M$. Let us
prove that 
\begin{equation}
 \label{tpot02:3}
\norm{\sdpar{\Vhat}{w}(v,0)}_{\ell^2}^2 = \Order{d^{-1}} 
\end{equation} 
for $\norm{v}_{H^1} \leqs M$. Indeed it follows from~\eqref{bp03} that 
\begin{equation}
 \label{tpot02:4}
\dpar{\Vhat}{w_k}(v,0) =  \int_0^L U'(u(x)) \e^{\icx bk\pi x/L}\6x\;, 
\end{equation} 
where $u(x)=\sum_{\abs{\ell}\leqs d} v_\ell \e^{\icx b\ell\pi x/L}$. Since
$\norm{v}_{H^1} \leqs M$ and $U$ is at least continuously differentiable,
the Fourier components of $ U'(u(x))$ decay like $k^{-1}$ at least, which
implies~\eqref{tpot02:3}. 

Let $(v^*,w^*)$ be a critical point of $\Vhat$ and consider the function 
\begin{equation}
 \label{tpot02:5}
F(\xi,w) = \sdpar{\Vhat}{v}(v^*+\xi,w)\;. 
\end{equation} 
Then $F(0,w^*)=0$ and $\sdpar F\xi(0,w^*)=\sdpar{\Vhat}{vv}(v^*,w^*)$. Thus
if $(v^*,w^*)$ is nondegenerate, the implicit function theorem
implies that in a neighbourhood of $w=w^*$, there exists a continuously
differentiable function $h$ with $h(w^*)=0$ and such that all solutions of
$F(\xi,w)=0$ in a neighbourhood of $(0,w^*)$ are given by $\xi=h(w)$. In
particular, choosing $d$ large enough, we can assume that $h$ is defined for $w=0$ , and we get 
\begin{equation}
 \label{tpot02:6}
0 = F(h(0),0) = 
\sdpar{\Vhat}{v}(v^*+h(0),0)\;.
\end{equation} 
This shows that $v_*=v^*+h(0)$ is a stationary point of $\Vhat^{(d)}$, which is
unique in the neighbourhood of $(v^*,w^*)$. 

Conversely, let $v_*$ be a stationary point of $\Vhat^{(d)}$. The same
implicit-function-theorem argument shows that if $v_*$ is nondegenerate, then
there exists a continuously differentiable function $h$, with $h(0)=0$, such
that all solutions of $\sdpar{\Vhat}v(v,w)=0$ near $(v_*,0)$ satisfy
$v=v_*+h(w)$. Now let us consider the function
\begin{equation}
 \label{tpot02:8}
g(w) = \sdpar{\Vhat}w (v_*+h(w),w)\;. 
\end{equation} 
Then $g(0)=\sdpar{\Vhat}w (v_*,0)$ has an $\ell^2$-norm of order $d^{-1/2}$
by~\eqref{tpot02:3}. Furthermore, 
\begin{equation}
 \label{tpot02:9}
\sdpar gw(w) =  \sdpar{\Vhat}{ww} (v_*+h(w),w) 
+ \sdpar{\Vhat}{wv} (v_*+h(w),w)\sdpar hw(w)\;.
\end{equation} 
The first matrix on the right-hand side has eigenvalues of order $d^2$, while
the second one is small as a consequence of~\eqref{tpot02:3}. Thus $\sdpar gw$
is invertible near $w=0$ for sufficiently large $d$, and the local inversion
theorem shows that $g(w)$ has an isolated zero  at a point $w^*$ near $w=0$.
This yields the existence of a unique stationary point $(v^*=h(w^*),w^*)$ of
$\Vhat$ in the vicinity of $(v_*,0)$. 
\end{proof}



\section{A priori estimates}
\label{sec_ld}


This section has two major aims:
\begin{itemiz}
\item	Show that the first-hitting time of a given
set $B$ admits a second moment, bounded uniformly in the dimension $d$;
\item	Derive a priori bounds on the equilibrium potential
$h_{A,B}(x)=\probin{x}{\tau_A<\tau_B}$.
\end{itemiz}
We start in Section~\ref{ssec_infdim} by recalling some general bounds
involving sup and H\"older norms of solutions of the SPDE~\eqref{SPDE}.
In order to estimate moments of first-hitting times, the space being
unbounded, we repeatedly need the Markov property to restart the process
when it hits certain sets. This is most efficiently done using Laplace
transforms, and we prove some useful inequalities in Section~\ref{ssec_laplace}. 
Section~\ref{ssec_lde}
recalls some large-deviation results. Sections~\ref{ssec_moments}
and~\ref{ssec_momfin} contain the main estimates on moments,
respectively, for the infinite-dimensional system and for its Galerkin
approximation. Finally, Section~\ref{ssec_eqpot2} contains the estimates
of the equilibrium potential.


\subsection{A priori bounds on solutions of the SPDE}
\label{ssec_infdim}

The solution of the heat equation $\sdpar ut = \Delta u$ with initial condition
$u_0\in L^2(\T^1)$ can be written 
\begin{equation}
 \label{infdim01}
u_t = \e^{\Delta t} u_0\;, 
\end{equation} 
where $\e^{\Delta t}$ stands for convolution with the heat kernel 
\begin{equation}
 \label{infdim02}
G_t(x,y) = \sum_{k\in\Z} \e^{-k^2 t} e_k(x)e_k(y) \indexfct{t>0}\;.
\end{equation} 
Here the $e_k$ are the eigenfunctions of the Laplacian, defined
in~\eqref{function02}. 

\begin{lemma}[Smoothing effect of the heat semigroup]
\label{lem_infdim01}
For any $s\geqs0$, there is a finite constant $C(s)$ such that for all $u_{0}\in L^2(\T^1)$
\begin{equation}
 \label{infdim02A}
\norm{\e^{\Delta t} u_0}_{H^s} \leqs (1+C(s) t^{-s/2}) \norm{u_0}_{L^2}
\qquad
\forall t>0 \;.
\end{equation}   
\end{lemma}
\begin{proof}
We have 
\begin{equation}
 \label{infdim02B}
\e^{\Delta t} u_0 = \sum_{k\in\Z} \e^{-k^2 t} y_k(0) e_k\;.
\end{equation} 
The result follows by computing the $H^s$-norm, and using the fact that
$(2xt)^s\e^{-2xt}$ is bounded by a constant, depending only on $s$. 
\end{proof}

Note that by Lemma~\ref{lemma_Morrey}, this implies 
\begin{equation}
 \label{infdim02C}
\norm{\e^{\Delta t} u_0}_{\cC^\alpha}
\leqs C(1+ t^{-s/2}) \norm{u_0}_{L^2}
\leqs C(1+ t^{-s/2}) \norm{u_0}_{L^\infty}
\quad\ 
\forall t>0 \;,
\ 
\forall s>\alpha+\frac12\;,
\end{equation} 
where the constant $C$ depends only on $\alpha$
and $s$.

Consider now the stochastic convolution 
\begin{equation}
 \label{infdim03}
W_\Delta(t) = \int_0^t \e^{\Delta(t-s)} \6W(s)\;, 
\end{equation} 
where $W(t)$ is a cylindrical Wiener process on $L^2(\T^1)$. It is known that
\begin{equation}
 \label{infdim03A} 
W_\Delta(t)\in H^s(\T^1)
\qquad
\text{and}
\qquad
W_\Delta(t)\in \cC^\alpha(\T^1)
\end{equation} 
almost surely, for all $t>0$ and all $s<1/2$ and $\alpha<1/2$
(see e.g.~\cite[p.~50]{Hairer_LN_2009}). 

We will need to control the sup and H\"older norms of the rescaled process 
$\sqrt{2\eps}W_\Delta(t)$. 

\begin{prop}[Large-deviation estimate for $W_\Delta$]
\label{prop_ldp_W}  
For any $\alpha\in[0,1/2)$ and $T>0$, there exists a constant $\kappa > 0$ such that for all
$H, \eta > 0$, there exists an $\eps_0>0$ such that for all $\eps<\eps_0$, 
\begin{equation}
 \label{infdim04A}
\biggprob{\sup_{0\leqs t\leqs T}
\bignorm{\sqrt{2\eps}\,W_\Delta(t)}_{\cC^\alpha} > H} \leqs \e^{-(\kappa H^2 -
\eta)/2\eps}\;. 
\end{equation} 
\end{prop}
\begin{proof}
Let $\cH$ denote the Cameron--Martin space of the cylindrical Wiener process,
defined by 
\begin{equation}
 \label{infdim04B:1}
\cH = \biggsetsuch{\varphi}{\varphi_t(x) = \int_0^t \int_0^x
\dot\varphi(u,z)\6u\6z\;,\; \dot\varphi\in L^2([0,T]\times \T^1)}\;. 
\end{equation} 
Schilder's theorem for Gaussian fields shows that the family
$\set{\sqrt{2\eps}\,W}_{\eps>0}$ satisfies a large-deviation principle with good
rate function 
\begin{equation}
 \label{infdim04B:1a} 
I_0(\varphi) = 
\begin{cases}
\frac12 \norm{\dot\varphi}_{L^2([0,T]\times\T^1)}^2
& \text{if $\varphi \in \cH$\;,} \\
+\infty 
& \text{otherwise\;.}
\end{cases}
\end{equation}
Define a map $Z:\cH\to L^2([0,T]\times \T^1)$ by 
\begin{equation}
\label{infdim04B:2}
\varphi \mapsto Z[\varphi]\;, 
\quad
Z[\varphi]_t = \int_0^t \e^{\Delta(t-s)}
\dot{\varphi}_s\6s \;.
\end{equation}
From the large-deviation principle for parabolic SPDEs established 
in~\cite{Faris_JonaLasinio82,ChenalMillet1997}, it follows in particular, 
that the family $\set{\sqrt{2\eps}W_\Delta}_{\eps>0}$ satisfies a large-deviation
principle with good rate function 
\begin{equation}
\label{infdim04B:3} 
I(\psi) = 
\begin{cases}
\inf\Bigsetsuch{I_0(\varphi)}{Z[\varphi]=\psi}
& \text{if $\psi \in \im(Z)$\;,} \\
+\infty 
& \text{otherwise\;.}
\end{cases}
\end{equation} 
Now observe that if $\psi=Z[\varphi]$, for any $T_1\in[0,T]$  and any
$s\in[0,1)$ one has by Lemma~\ref{lem_infdim01}
\begin{align}
\nonumber
\norm{\psi_{T_1}}_{H^s(\T^1)}
&\leqs \int_0^{T_1} \norm{\e^{\Delta(T_1-t)}\dot\varphi_t}_{H^s(\T^1)} \6t \\
\nonumber
&\leqs \int_0^{T_1}
\biggpar{1+\frac{C(s)}{(T_1-t)^{s/2}}}\norm{\dot\varphi_t}_{L^2(\T^1)} \6t \\
&\leqs 
\biggpar{\int_0^{T_1} \biggpar{1+ \frac{C(s)}{(T_1-t)^{s/2}}}^{2} \6t}^{1/2} 
\bigpar{2I_0(\varphi)}^{1/2}\;.
\label{infdim04B:4} 
\end{align}
Since $s<1$, the integral is finite (and increasing in $T_1$). 
Together with Lemma~\ref{lemma_Morrey}, this proves
that 
\begin{equation}
 \label{infdim04B:5}
I(\psi) \geqs \frac{1}{C_1(T_1,\alpha)} \norm{\psi_{T_1}}_{\cC^\alpha(\T^1)}^2 
\end{equation} 
for all $T_1\in[0,T]$, $0\leqs\alpha<1/2$ and $s$ satisfying $\alpha+1/2 <s<1$,
where $C_1$ is increasing in
$T_1$. By a standard application of the
large-deviation principle (see e.g.~\cite{Faris_JonaLasinio82,ChenalMillet1997})
\begin{equation}
 \label{infdim04B:6}
\limsup_{\eps\to0} 2\eps \log 
\biggprob{\sup_{0\leqs t\leqs T} \bignorm{\sqrt{2\eps}W_\Delta(t)}_{\cC^\alpha}
> H} \leqs -\inf\bigsetsuch{I(\psi)}{\exists T_1\in[0,T] :
\norm{\psi_{T_1}}_{\cC^\alpha} > H}\;.
\end{equation}
The bound~\eqref{infdim04B:5} implies that the right-hand side is bounded
above by $-H^2/C_1(T,\alpha)$, which concludes the proof.  
\end{proof}

We now turn to properties of mild solutions of the full nonlinear SPDE, given
by 
\begin{equation}
 \label{infdim06}
u_t = \e^{\Delta t} u_0 + \sqrt{2\eps} \, W_\Delta(t) 
+ \int_0^t \e^{\Delta(t-s)} U'(u_s)\6s\;.
\end{equation} 
Results in \cite{Cerrai_1996,Cerrai_1999} provide estimates on the sup
norm of $u_t$:

\begin{prop}[Uniform bounds on the sup norm]
\label{prop_Cerrai} 
For any $u_0\in \cC^0(\T^1)$ and any $T>0$, there exists a unique mild solution
on $[0,T]$ such that $\expec{\sup_{t\in[0,T]}\norm{u_t}_{L^\infty}^2}<\infty$. 
Furthermore, there is a constant $c$ depending only on $U'$ such that the
following bounds hold: 
\begin{enum}
\item	
There exists $\gamma>0$ such that 
for any $u_0$ and any $t\geqs 0$,
\begin{align}
\nonumber
\norm{u_t}_{L^\infty} \leqs{}& 
\e^{\gamma t} \norm{u_0}_{L^\infty} 
+ \sqrt{2\eps} \sup_{0\leqs s\leqs t} \norm{W_\Delta(s)}_{L^\infty} \\
&{}+ c\e^{\gamma t} \int_0^t
\Bigpar{1+(2\eps)^{(2p_{0}-1)/2}\norm{W_\Delta(s)}_{L^\infty}^{2p_{0}-1}}\6s\;,
 \label{infdim07}
\end{align} 

\item	For any $t>0$, 
\begin{equation}
 \label{infdim08}
\sup_{u_0\in \cC^0(\T^1)} \norm{u_t}_{L^\infty} \leqs 
c \biggpar{1 +  \sqrt{2\eps} \sup_{0\leqs s\leqs t}
\norm{W_\Delta(s)}_{L^\infty}} t^{-1/2(p_{0}-1)}
+  \sqrt{2\eps}\,\norm{W_\Delta(t)}_{L^\infty}\;.
\end{equation} 
\end{enum}
\end{prop}
\begin{proof}
Existence and uniqueness of the mild solution are proved in
\cite[Theorem~7.13]{DaPrato_Jabczyk_92}.
The estimate \eqref{infdim07} is Proposition~3.2 of~\cite{Cerrai_1999}, with
$m=p_{0}-1$, while the uniform estimate \eqref{infdim08} is Proposition~3.4
of~\cite{Cerrai_1999}, c.f.\ also~\cite[Lemma~3.4]{Cerrai_1996}. 
\end{proof}

Observe that in the case $\eps=0$, we  can find a constant $M$ 
\emph{uniform}\/ in $t$ such that $\norm{u_t}_{L^\infty} \leqs
M(1+\norm{u_0}_{L^\infty})$ for all $t\geqs0$. Hence
Proposition~\ref{prop_ldp_W} shows that for all $H_1>0$, 
\begin{align}
\label{infdim08A}
&\biggprobin{u_0}{\sup_{0\leqs t\leqs T} \norm{u_t}_{L^\infty} \geqs 
M(1+\norm{u_0}_{L^\infty})+H_1}
\leqs \e^{-\kappa(T)f(H_{1})/2\eps} \;,\\
\nonumber
\llap{\text{where}}\ \hfill \hskip1.7cm \hfill \\
\label{infdim08A_b}
&f(H_{1})=\min\set{H_{1}^2, H_{1}^{2/(2p_{0}-1)}} \;,
\end{align} 
for some $\kappa(T)>0$ and $\eps$ small enough. 

Combining~\eqref{infdim06} and Proposition~\ref{prop_Cerrai}, we obtain the
following estimate on the H\"older norm of $u_T$ at a given time $T>0$. 

\begin{prop}[Bound on the H\"older norm]
\label{prop_Holder_decay} 
For any $T>0$ and $0<\alpha<1/2$, there exist constants
$\kappa_1(T,\alpha), \kappa_2(T,\alpha)>0, c(\alpha)>0$
such that 
\begin{equation}
 \label{infdim09}
\bigprobin{u_0}{\norm{u_T}_{\cC^\alpha} > H} 
\leqs 
\exp\biggset{-\frac{\kappa_1}{2\eps}
\min\Bigl\{H^{2},
f\Bigpar{(\kappa_2H-1)^{1/(2p_{0}-1)}-M(1+\norm{u_0}_{ L^\infty } )}
 \Bigr\}
}
\end{equation}  
for all $u_0\in L^\infty$ and $H>c(\alpha)(1+T^{-s/2})\norm{u_0}_{L^\infty}$ such that 
\begin{equation}
 \label{infdim099} 
 (\kappa_2H-1)^{1/(2p_{0}-1)}-M(1+\norm{u_0}_{L^\infty}) >0\;,
\end{equation}
and all $\eps < \eps_0(\alpha,T,H)$.
\end{prop}
\begin{proof}
Denote by $u^{(0)}_t$, $u^{(1)}_t$ and $u^{(2)}_t$ the three summands on the
right-hand side of~\eqref{infdim06}. Then the probability~\eqref{infdim09} can
be bounded by $\sum_{i=0}^2P_i$, where
$P_i=\probin{u_0}{\smash{\norm{u^{(i)}_T}_{\cC^\alpha}} > H/3}$.

Pick $s>\alpha+1/2$. Then Lemma~\ref{lemma_Morrey} and 
Lemma~\ref{lem_infdim01} show that there exists $C_1(\alpha,s)$ such
that $P_0=0$, provided we choose $H/3 > C_1(1+T^{-s/2})\norm{u_0}_{L^\infty}$. 
Furthermore, Proposition~\ref{prop_ldp_W} provides a bound on $P_1$ of order
$\e^{-\kappa_1H^2/2\eps}$. 

As for $P_2$, it can be bounded as follows. Since $\abs{U'(u)}\leqs
M_0(1+\abs{u}^{2p_{0}-1})$ for some constant $M_0$, we have by~\eqref{infdim02C} 
\begin{align}
\nonumber
\norm{u^{(2)}_T}_{\cC^\alpha} 
&{}\leqs  \int_0^T \norm{\e^{\Delta(T-t)}U'(u_t)}_{\cC^\alpha}\6t \\
&{}\leqs
\int_0^T C(\alpha, s)(1+(T-t)^{-s/2})\6t \, M_0 \Bigpar{1+\sup_{t\in[0,T]}
\norm{u_t}_{L^\infty}^{2p_{0}-1}}\;.
 \label{infdim10} 
\end{align} 
The integral is bounded provided $s<2$. The result then follows by
using~\eqref{infdim08A}.
\end{proof}


\subsection{Laplace transforms}
\label{ssec_laplace} 

Let $(E, \norm{\cdot})$ be a Banach space, and 
let $(x_t)_{t\geqs0}$ be an $E$-valued Markov process with continuous sample
paths. All subsets of $E$ considered below are assumed to be measurable with respect to the Borel $\sigma$-algebra on $E$.

Recall that the Laplace transform of an almost surely finite positive random
variable $\tau$ is given by 
\begin{equation}
 \label{Laplace0}
\bigexpec{\e^{\lambda\tau}} = 1 + \int_0^\infty \lambda \e^{\lambda t} 
\bigprob{\tau>t} \6t 
\end{equation} 
for any $\lambda\in\C$. There exists a $c\in[0,\infty]$ such that the Laplace
transform is analytic in $\lambda$ for $\re\lambda<c$.  

To control first-hitting times of bounded sets $B\subset E$, we will introduce
an auxiliary set $C$ with bounded complement, $B\cap C=\emptyset$, such that the
process is unlikely to hit $C$ before $B$. On the rare occasions the process
does hit $C$ before $B$, we will use the strong Markov property to restart the
process on the boundary $\partial C$. The following proposition recalls how the
restart procedure is encoded in Laplace transforms.

\begin{prop}[Effect of restart on Laplace transform]
\label{prop_Laplace}
Let $B, C\subset E$ be disjoint sets, and let $x\not\in B\cup C$. Then 
\begin{align}
\label{Laplace1A}
\bigexpecin{x}{\e^{\lambda\tau_B}} 
&= \bigexpecin{x}{\e^{\lambda\tau_{B\cup C}}} + 
\bigexpecin{x}{\e^{\lambda\tau_{B\cup C}}\indexfct{\tau_C<\tau_B} 
\bigbrak{\bigexpecin{x_{\tau_C}}{\e^{\lambda\tau_B}}-1}} \\
&= \bigexpecin{x}{\e^{\lambda\tau_{B\cup C}}\indexfct{\tau_B<\tau_C}} + 
\bigexpecin{x}{\e^{\lambda\tau_{B\cup C}}\indexfct{\tau_C<\tau_B} 
\bigexpecin{x_{\tau_C}}{\e^{\lambda\tau_B}}}\;.
 \label{Laplace1B} 
\end{align}  
\end{prop}

In the same way, or by differentiating~\eqref{Laplace1A} with respect to $\lambda$ and evaluating in $\lambda=0$, the moments of first-hitting times can be expressed. Assuming their existence, for the first two moments we find 
\begin{align}
\label{Laplace2A} 
\bigexpecin{x}{\tau_B} 
&= \bigexpecin{x}{\tau_{B\cup C}}
+ \bigexpecin{x}{\indexfct{\tau_C<\tau_B} \expecin{x_{\tau_C}}{\tau_B}}\;, \\
\bigexpecin{x}{\tau_B^2} 
&= \bigexpecin{x}{\tau_{B\cup C}^2}
+ 2\bigexpecin{x}{\tau_{B\cup C}\indexfct{\tau_C<\tau_B}
\expecin{x_{\tau_C}}{\tau_B}}
+ \bigexpecin{x}{\indexfct{\tau_C<\tau_B} \expecin{x_{\tau_C}}{\tau_B^2}}
\label{Laplace2B} 
\end{align}
for any choice of disjoint sets $B$ and $C$, and any $x\not\in(B\cup C)$.

Below we will use the notations 
\begin{equation}
 \label{ldf20}
\probin{A}{X\in\cdot} = \sup_{y\in A}\probin{y}{X\in\cdot}
\qquad
\text{and}
\qquad
 \expecin{A}{X} = \sup_{y\in A}\expecin{y}{X}\;.
\end{equation}
It follows that for any three pairwise disjoint sets $A$, $B$ and $C$, 
\begin{align}
\nonumber
\bigexpecin{A}{\tau_B} 
&\leqs \bigexpecin{A}{\tau_{B\cup C}}
+ \bigexpecin{A}{\indexfct{\tau_C<\tau_B} \expecin{x_{\tau_C}}{\tau_B}}\\
&\leqs \bigexpecin{A}{\tau_{B\cup C}} + \bigprobin{A}{\tau_C<\tau_B}
\bigexpecin{\partial C}{\tau_B}\;,
 \label{Laplace2C} 
\end{align}  
and a similar relation holds for the second moment.

\begin{lemma}[Moment estimate based on the Markov property]
\label{lem_ldf3} 
Let $B\subset E$ be such that 
\begin{equation}
 \label{ldf21A} 
  \bigprobin{B^c}{\tau_B> T} < 1
\end{equation}
for some $T>0$. Then for any $n\in\N$,  
\begin{equation}
 \label{ldf21}
\bigexpecin{B^c}{\tau_B^n} 
\leqs \frac{n!T^n}{\displaystyle
\bigpar{1- \bigprobin{B^c}{\tau_B > T}}^n}\;. 
\end{equation} 
\end{lemma}
\begin{proof}
The Markov property implies that for any $m\in\N$ and any $x\in B^c$,
\begin{equation}
 \label{ldf21:2}
 \bigprobin{x}{\tau_B>(m+1)T}
= \bigexpecin{x}{\indexfct{\tau_B>mT}\probin{y_{mT}}{\tau_B>T}}
\leqs \probin{B^c}{\tau_B>T} \probin{x}{\tau_B>mT}\;,
\end{equation} 
so that $\probin{x}{\tau_B>mT}\leqs(\probin{B^c}{\tau_B>T})^m$. 
Integration by parts shows that 
\begin{equation}
 \label{ldf21:1}
\bigexpecin{x}{\tau_B^n} = 
n\int_0^\infty t^{n-1} \bigprobin{x}{\tau_B>t}\6t 
\leqs nT^n \sum_{m=0}^\infty (m+1)^{n-1} \bigprobin{x}{\tau_B>mT}\;. 
\end{equation} 
The result is thus a consequence of the inequality
\begin{equation}
 \label{ldf21:3}
\sum_{m=0}^\infty (m+1)^{n-1}p^m \leqs \frac{(n-1)!}{(1-p)^n}
\qquad
\forall p\in[0,1)\;,
\end{equation} 
which follows from properties of the polylogarithm function and Eulerian
numbers.
\end{proof}

\begin{remark}
\label{rem_ldf1} 
Equation~\eqref{ldf21} implies that if~\eqref{ldf21A} holds, the Laplace transform
of $\tau_B$ exists for 
\begin{equation}
 \label{ldf21B} 
\lambda<\frac{1}{T}\bigpar{1- \bigprobin{B^c}{\tau_B > T}}
\end{equation} 
and satisfies
\begin{equation}
 \label{ldf21C}
\bigexpecin{B^c}{\e^{\lambda \tau_B}} \leqs 
\frac{1}{1-\lambda T/(1-\probin{B^c}{\tau_B > T})} \;.
\end{equation} 
A sharper bound on the Laplace transform can be obtained by a direct integration by parts, but this does not automatically lead to better bounds on the moments.
\end{remark}

Next, we will iterate the estimate~\eqref{Laplace2C} in order to get a better bound on the moments of first hitting times.

\begin{cor}[Three-set argument]
\label{cor_ldf4}
Let $A, B, C\subset E$ be such that $ A,  B$ and $ C$
are pairwise disjoint, and assume $\bigprobin{A}{\tau_C<\tau_B}<1$, 
$\expecin{A}{\tau^{k}_B}<\infty$ and $\expecin{\partial C}{\tau^{k}_B}<\infty$
for $k=1,2$. Then
\begin{align}
\label{ldf30A} 
\bigexpecin{A}{\tau_B} 
&\leqs \frac{\bigexpecin{A}{\tau_{B\cup C}} +
\bigprobin{A}{\tau_C<\tau_B}\bigexpecin{\partial C}{\tau_{A\cup B}}}
{1-\bigprobin{A}{\tau_C<\tau_B}}\;,\\
\bigexpecin{A}{\tau_B^2} 
&\leqs 4 \frac{\bigexpecin{A}{\tau_{B\cup C}^2} +
\bigexpecin{\partial C}{\tau_{A\cup B}^2}}
{\bigpar{1-\bigprobin{A}{\tau_C<\tau_B}}^2}\;.
\label{ldf30B} 
\end{align}
\end{cor}
\begin{proof}
We introduce the shorthands 
\begin{equation}
 \label{ldf30:1}
X_k = \bigexpecin{A}{\tau_{B\cup C}^k}\;,
\qquad
Y_k = \bigexpecin{\partial C}{\tau_{A\cup B}^k}\;,
\qquad
p = \bigprobin{A}{\tau_C<\tau_B}\;,
\end{equation} 
for $k=1,2$. Note that $X_{k}, Y_{k}<\infty$ and $p<1$ according to our assumptions. Applying~\eqref{Laplace2C}, once to the triple $(A,B,C)$ and once
to the triple $(\partial C,B,A)$, yields 
\begin{align}
\nonumber
\bigexpecin{A}{\tau_B} &\leqs X_1 + p \bigexpecin{\partial C}{\tau_B} \;,\\
\bigexpecin{\partial C}{\tau_B} &\leqs Y_1 + \bigexpecin{\partial A}{\tau_B}
= Y_{1} +  \bigexpecin{A}{\tau_B} \;,
\label{ldf30:2}
\end{align} 
where we have bounded $\probin{\partial C}{\tau_A<\tau_B}$ by $1$. In addition, we used that hitting $B$ requires first exiting from $A$ which is necessarily realized by passing through $\partial A$. 
This implies 
\begin{align}
\nonumber
\bigexpecin{A}{\tau_B} &\leqs \frac{X_1+pY_1}{1-p}\;, \\
\bigexpecin{\partial C}{\tau_B} &\leqs \frac{X_1+Y_1}{1-p}\;,
\label{ldf30:3}
\end{align}
which proves~\eqref{ldf30A}.
Starting from~\eqref{Laplace2B}, we find
\begin{align}
\nonumber
\bigexpecin{A}{\tau_B^2} &\leqs X_2 + 2X_1 \bigexpecin{\partial C}{\tau_B}
+ p\bigexpecin{\partial C}{\tau_B^2}\;, \\
\bigexpecin{\partial C}{\tau_B^2} &\leqs Y_2 + 2Y_1\bigexpecin{A}{\tau_B} 
+ \bigexpecin{A}{\tau_B^2}\;.
\label{ldf30:4}
\end{align} 
Together with~\eqref{ldf30:3}, this gives 
\begin{equation}
 \label{ldf30:5}
(1-p)^2 \bigexpecin{A}{\tau_B^2} \leqs 
(1-p)(X_2+pY_2) + 2(X_1^2+(1+p)X_1Y_1+p^2Y_1^2)\;,
\end{equation} 
and the result follows after some algebra, using Jensen's inequality. Note that we have
overestimated some terms in order to obtain a more compact expression.
\end{proof}


\subsection{Large deviations}
\label{ssec_lde}

As shown in~\cite{Faris_JonaLasinio82,Freidlin88}, the family
$\set{u_t}_{\eps>0}$ of mild solutions of the SPDE with initial condition
$u_0\in E=\cC^0(\T^1)$ satisfies a large-deviation principle in $E$, equipped
with the sup norm, with rate function
\begin{equation}
 \label{new_eqpot01B}
I_{[0,T]}(\ph) = 
\begin{cases}
\displaystyle 
\frac12 \int_0^{T} \int_0^L \Bigbrak{\dot\ph_t(x)
-\ph''_t(x) + U'(\ph_t(x))}^2\6x \6t  
& \text{if the integral is finite\;,}\\
\vrule height 18pt depth 8pt width 0pt
+\infty &\text{otherwise\;.}
\end{cases}
\end{equation} 
For $u_0\in E$ and $A\subset E$, let
\begin{equation}
 \label{new_eqpot01C}
H(u_0,A) = \frac12
\inf_{T>0} \biggpar{\inf_{\ph:\ph(0)=u_0, \exists t\leqs T \text{ s.t.\ } \ph(t)\in
A} I_{[0,T]}(\ph)}\;, 
\end{equation} 
where the second infimum runs over all continuous paths $\varphi:[0,T]\to E$
connecting $u_0$ in a time $t\leqs T$ to a point in $A$  (if $u_0$ is a local minimum,
then $u\mapsto 2H(u_0,{u})$ is called \defwd{quasipotential}\/). 

We define the \defwd{relative communication height}\/ between
$u_0$ and $A$ by 
\begin{equation}
 \label{new_eqpot02}
\Vbar(u_0,A) =  \inf_{\psi:\psi(0)=u_0, \psi(1)\in A} 
\biggpar{\sup_{t\in[0,1]} V[\psi_t] - V[u_0]} \;,
\end{equation} 
where the infimum now runs over all continuous paths $\psi:[0,1]\to E$
connecting $u_0$ to an endpoint in $A$ (the parameter $t$ need not be
associated to time in this definition). Note that $\Vbar(u_0,A)=0$ if and only
if one can find a path from $u_0$ to $A$ along which the potential is
nonincreasing. This holds in particular when $u_0$ lies in the basin of
attraction of $A$. If $\Vbar(u_0,A)>0$, then one has to cross a potential
barrier in order to reach $A$ from $u_0$. 

The following classical result shows that $H(u_0,A)$ can be estimated below in
terms of the relative communication height.

\begin{lemma}
\label{lem_ldp_Vbar}
For any $u_0\in E$ and any $A\subset E$, we have 
\begin{equation}
 \label{new_eqpot03}
H(u_0,A) \geqs \Vbar(u_0,A)\;. 
\end{equation}  
\end{lemma}
\begin{proof}
Let $\varphi$ be a path connecting $u_0$ to $A$ in time $T$. By definition of
the communication height, the potential on any such path has to reach the value
$\Vbar(u_0,A)+V[u_0]$ at least once. Denoting by $T_1$ the first time this
happens, we have 
\begin{align}
\nonumber
I_{[0,T]}(\ph) 
\geqs{}&  \frac12 \int_0^{T_1} \int_0^L \Bigbrak{\dot\ph_t(x)
+\ph''_t(x) - U'(\ph_t(x))}^2\6x \6t \\
\nonumber
{}&+ 2 \int_0^{T_1} \int_0^L \Bigbrak{
-\ph''_t(x) + U'(\ph_t(x))} \dot\ph_t(x) \6x 
 \6t \\
\nonumber
\geqs{}&2 \int_0^{T_1} \int_0^L \Bigbrak{\ph'_t(x)\dot\ph'_t(x) +
U'(\ph_t(x))\dot\ph_t(x)} \6x\6t \\
=& 2 \Bigbrak{V[\ph_{T_1}] - V[\ph_0]} 
= 2\Vbar(u_0,A)\;.
 \label{new_eqpot03:04}
\end{align}
The right-hand side being independent of $T$, the result follows. 
\end{proof}

By a direct application of the large-deviation principle to the set of paths
starting in~$u_0$ and reaching $A$ in a time less or equal to $T$, we obtain the
following estimate. 

\begin{cor}
\label{cor_ldp_Vbar} 
For any $\eta>0$, there exists $\eps_0(\eta)>0$ such that 
\begin{equation}
 \label{eqpot31}
 \bigprobin{u_0}{\tau_A < T} \leqs 
 \e^{-(H(u_0,A)-\eta)/\eps}
\end{equation} 
for all $\eps<\eps_0$. 
\end{cor}


\subsection{Bounds on moments of $\tau_B$ in infinite dimension}
\label{ssec_moments}

We will now apply the results of the previous sections to the mild solution
$u_t$ of the SPDE, with $E=\cC^0(\T^1)$ equipped with the sup norm. 
We fix $\alpha\in(0,1/2)$ and introduce two families of sets 
\begin{align}
\nonumber
A_1(R) &= \bigsetsuch{u\in\cC^0(\T^1)}{\norm{u}_{L^\infty}\leqs R}\;, \\ 
A_2(R) &= \bigsetsuch{u\in\cC^0(\T^1)}{\norm{u}_{\cC^\alpha}\leqs R}\;.
\label{moments01} 
\end{align}
Note  that $A_{2}\subset A_{1}$, and that $A_2$ is a compact subset of $E$, while
$A_1$ is not compact as a subset of $E$. 

Let $B\subset\cC^0(\T^1)$ be a non-empty, bounded open set in the
$\norm{\cdot}_{L^\infty}$-topology. We let 
\begin{equation}
 \label{moments00}
H_0 = H_0(B) = H(u^*_-,B)\vee H(u^*_+,B) 
\end{equation} 
be the cost, in terms of the rate function, to reach the set $B$ from either one of
the local minima. Our aim is to estimate the first two moments
of $\tau_B$, using the three-set argument Corollary~\ref{cor_ldf4} for
$A=A_2(R_2)\setminus B$ and $C=A_1(R_0)^c$, with appropriately chosen $R_0$ and
$R_2$. We thus proceed to estimating the  quantities appearing in the
right-hand side of~\eqref{ldf30A} and~\eqref{ldf30B}. 

We will use repeatedly the fact that 
Proposition~\ref{prop_ldp_W} and~\eqref{infdim08} yield the estimate 
\begin{equation}
 \label{moments99}
 \sup_{u_0\in E}
 \biggprobin{u_0}{\norm{u_T}_{L^\infty} > 
 H + \frac{c(1+H)}{T^{1/2(p_0-1)}}}
 \leqs \e^{-(\kappa H^2-\eta)/2\eps}\;,
\end{equation} 
valid for all $\eps < \eps_0(T,\eta,H)$. 

\begin{prop}[Bounds on moments of $\tau_{A_2}$]
\label{prop_moments1}
For any sufficiently large $R_2$, there exist $T_0 <\infty$, $\eps_0>0$ such
that for $\eps<\eps_0$, 
\begin{equation}
\label{moments02}
 \bigexpecin{A_2(R_2)^c}{\tau_{A_2(R_2)}^n} \leqs n!T_0^n
\end{equation} 
holds for all $n\geqs 1$. 
\end{prop}
\begin{proof}
Choose a fixed $T_1>0$. Then the Markov property applied at time $T_1/2$ shows that for any $R_{1}>0$,
\begin{equation}
 \label{moments03}
\bigprobin{A_2(R_2)^c}{\tau_{A_2(R_2)}> T_1} \leqs 
\bigprobin{A_2(R_2)^c}{\norm{u_{T_1/2}}_{L^\infty} > R_1} + 
\bigprobin{A_1(R_1)}{\norm{u_{T_1/2}}_{\cC^\alpha} > R_2}\;.
\end{equation} 
The estimate~\eqref{moments99} shows that for 
$R_1=c(2/T_1)^{1/2(p_0-1)}+\theta$ with $\theta>0$, the
first term on the right-hand side is smaller than $1/4$ for
$\eps\leqs \eps_{0}(T_{1})$, uniformly in the initial condition. By
Proposition~\ref{prop_Holder_decay}, we can find $R_2$ such that the second
term is also smaller than $1/4$.  This shows
\begin{equation}
 \label{moments03_b}
\bigprobin{A_2(R_2)^c}{\tau_{A_2(R_2)}> T_1} \leqs \frac12\;.
\end{equation} 
By Lemma~\ref{lem_ldf3}, this
yields~\eqref{moments02} with $T_0=2T_1$.
\end{proof}

\begin{prop}
\label{prop_moments2}
For any $R_2, \eta>0$, there exists a constant 
$T(\eta)\in(0,\infty)$  
such that 
\begin{equation}
 \label{moments04}
\bigprobin{A_2(R_2)}{\tau_B > T} \leqs 
1 - \frac12\e^{-(H_0+\eta)/\eps}
\end{equation}  
for sufficiently small $\eps$. 
\end{prop}
\begin{proof}
We start by fixing an initial condition
$u_0\in A_2(R_2)$. By the large-deviation principle, we have 
\begin{equation}
 \label{moments05:3}
\liminf_{\eps\to0} 2\eps \log\bigprobin{u_0}{\tau_B \leqs T} 
\geqs - \inf\bigsetsuch{I(\varphi)}{\varphi_0=u_0,\; \varphi_T\in B}\;. 
\end{equation} 
Following a classical procedure similar to the one in the proof
of~\cite[Theorem~9.1]{Faris_JonaLasinio82}, we construct a path
$\varphi^*_{u_0}$ connecting $u_0$ to a point $u^*\in B$ such that
$I(\varphi^*_{u_0}) \leqs 2H_0 + \eta$. This can be done by following
the deterministic
flow from $u_0$ to the neighbourhood of a stationary solution of the
deterministic PDE at zero cost, then connecting to that stationary solution at
finite cost. Any two stationary solutions and $u^*$ can also be connected at
finite cost, and $\varphi^*_{u_0}$ is obtained by concatenation. It follows that
there exists $\eps_0(\eta,u_0)>0$ such that 
\begin{equation}
 \label{moments05:4}
 \bigprobin{u_0}{\tau_B \leqs T} \geqs \e^{-(I(\varphi^*_{u_0})+\eta)/2\eps}
\qquad
\forall \eps < \eps_0(\eta,u_0)\;.
\end{equation} 
The set $A_2(R_2)$ being compact, we can find, for any $\delta>0$, a finite
cover of $A_2(R_2)$ with $N(\delta)$ balls of the form
$D_n=\setsuch{u\in\cC^0(\T^1)}{\norm{u-u_n}_{L^\infty}<\delta}$. Hence 
\begin{equation}
 \label{moments05:5}
\max_{1\leqs n\leqs N(\delta)} 
 \bigprobin{u_n}{\tau_B \leqs T} \geqs \e^{-(I^*(\delta)+\eta)/2\eps}
\qquad
\forall \eps < \eps_1(\eta,\delta)\;,
\end{equation} 
where $I^*(\delta)=\max_n I(\varphi^*_{u_n})<2H_0+\eta$  and
$\eps_1(\eta,\delta) = \min_n\eps_0(\eta,u_n)>0$. 

Consider now two solutions $u_t^{(1)}, u_t^{(2)}$ of the SPDE 
with initial conditions $u_0^{(1)}, u_0^{(2)} \in D_n$. By a Gronwall-type
argument similar to the one given in~\cite[Theorem~5.10]{Faris_JonaLasinio82}
and the  bound~\eqref{infdim09}, for any $\kappa_1>0$ there exist
$K(\kappa_1)>0$ such that 
\begin{equation}
 \label{moments05:6}
\biggprob{\sup_{0\leqs t\leqs T} \norm{u_t^{(1)} - u_t^{(2)}}_{L^\infty} 
> \e^{KT} \norm{u_0^{(1)} - u_0^{(2)}}_{L^\infty}} 
\leqs \e^{-\kappa_1/2\eps}\;.
\end{equation} 
The set $B$ being open, it contains a ball
$\setsuch{u\in\cC^0(\T^1)}{\norm{u-u^*}_{L^\infty}<\rho}$ with $\rho>0$. Thus
choosing $\delta=\e^{-KT}\rho$, combining~\eqref{moments05:5} and~\eqref{moments05:6}, we get 
\begin{equation}
 \label{moments05:7}
 \bigprobin{A_2(R_2)}{\tau_B> T} \leqs 
1 - \e^{-(H_0+\eta)/\eps} + \e^{-\kappa_1/2\eps}
\end{equation} 
for all sufficiently small $\eps$. Now choosing, e.g.,
$\kappa_1=I(\varphi^*_{u_1})$ and  $\delta=\e^{-K(\kappa_1)T}\rho$
guarantees that the term $\e^{-\kappa_1/2\eps}$ is negligible. 
\end{proof}

\begin{prop}
\label{prop_moments3}
For every $\eta>0$ and sufficiently large $R_0, R_2$ satisfying $R_{0}>R_{2}$,
there exists $T_0(\eta)\in(0,\infty)$
such that 
\begin{equation}
 \label{moments06}
\bigexpecin{A_1(R_0)\cap B^c}{\tau_{B\cup A_1(R_0)^c}^n} 
\leqs n!T_0^n \e^{n(H_0+\eta)/\eps}
\end{equation}  
for all $n\geqs1$. 
\end{prop}
\begin{proof}
For any $T_1>1$ and $R_1>0$, we have
\begin{align}
\nonumber
\bigprobin{A_1(R_0)\cap B^c}{\tau_{B\cup A_1(R_0)^c} > T_1}
\leqs{}&  
\bigprobin{A_1(R_0)\cap B^c}{\norm{u_{T_1/3}}_{L^\infty} > R_1} \\
\nonumber
&{}+ \bigprobin{A_1(R_1)}{\norm{u_{T_1/3}}_{\cC^\alpha} > R_2} \\
&{}+ \bigprobin{A_2(R_2)}{\tau_B > T_1/3}\;.
 \label{moments07:1}
\end{align} 
The third term on the right-hand side is bounded by
$1-\frac12\e^{-(H_0+\eta)/\eps}$ by Proposition~\ref{prop_moments2}.
Proposition~\ref{prop_Holder_decay} shows that 
\begin{equation}
 \label{moments07:2}
\bigprobin{A_1(R_1)}{\norm{u_{T_1/3}}_{\cC^\alpha} > R_2}
\leqs 
\exp\biggset{-\frac{\kappa_1}{2\eps}\min\Bigset{R_{2}^{2},f\Bigpar{(\kappa_2
R_2-1)^{1/(2p_{0}-1)}-M(1+R_1)}} }\;,
\end{equation} 
while~\eqref{moments99} shows that there exists $\kappa_3(T_1)>0$ such that
\begin{equation}
 \label{moments07:3}
 \bigprobin{A_1(R_0)\cap B^c}{\norm{u_{T_1/3}}_{L^\infty} > R_1}
\leqs 
\exp\biggset{-\frac{\kappa_3}{2\eps} \biggpar{R_1-c\,\biggpar{\frac{T_1}{3}}^{-1/2(p_{0}-1)}}^2}\;. 
\end{equation} 
We have estimated the probability in~\eqref{moments07:1} by three terms. By first choosing 
$R_1$ large enough so that the exponent in~\eqref{moments07:3} is
 smaller than $-(H_0+2\eta)/\eps$, and then $R_2$ sufficiently large for the
exponent in~\eqref{moments07:2} to be smaller than $-(H_0+2\eta)/\eps$ as well, we see that the third summand in~\eqref{moments07:1}  is of leading order. This shows that the probability in~\eqref{moments07:1} is smaller than
$1-\frac14\e^{-(H_0+\eta)/\eps}$ for sufficiently small $\eps$, and therefore, the result
follows from Lemma~\ref{lem_ldf3} with $T_0=4T_1$. 
\end{proof}

\begin{prop}
\label{prop_moments4}
For any $R_2>0$, there exists a constant $R_0>R_{2}$ such that 
\begin{equation}
 \label{moments104}
\bigprobin{A_2(R_2)}{\tau_{A_1(R_0)^c}\leqs\tau_B} \leqs 
\frac12
\end{equation}  
for sufficiently small $\eps$. 
\end{prop}
\begin{proof}
For any $T>0$ and $n\in\N$, we have 
\begin{equation}
 \label{moments105:1}
 \bigprobin{A_2(R_2)}{\tau_{A_1(R_0)^c}\leqs\tau_B} \leqs
\bigprobin{A_2(R_2)}{\tau_{A_1(R_0)^c}\leqs nT} +
\bigprobin{A_2(R_2)}{\tau_B> nT}\;.
\end{equation}
We introduce the quantities 
\begin{align}
\nonumber
p_n &= \bigprobin{A_2(R_2)}{\tau_B> nT}\;, \\
\nonumber
q_n &= \bigprobin{A_2(R_2)}{u_{nT}\not\in A_2(R_2)} \;, \\
r_n &= \bigprobin{A_2(R_2)}{\tau_{A_1(R_0)^c}\leqs nT}\;.
\label{moments105:2} 
\end{align}
Using the Markov property, they can all be expressed in terms of $p_1, q_1$ and
$r_1$. Namely, 
\begin{align}
\nonumber
q_{n+1} 
&\leqs \bigprobin{A_2(R_2)}{u_{nT}\notin A_2(R_2)} + 
\bigexpecin{A_2(R_2)}{\indexfct{u_{nT}\in A_2(R_2)} \probin{u_{nT}}{u_T\notin
A_2(R_2)}} \\
&\leqs q_n + q_1(1-q_n)\;,
\label{eqpot21:04} 
\end{align}
and one easily shows by induction that 
\begin{equation}
 \label{eqpot21:05}
q_n \leqs 1 - (1-q_1)^n \leqs n q_1\;. 
\end{equation} 
Splitting again according to whether $u_{nT}$ belongs to $A_2(R_2)$ or not, we
get
\begin{equation}
r_{n+1} \leqs q_n + r_1(1-q_n) \leqs r_1 + n q_1\;.
\label{eqpot21:10} 
\end{equation}
In a similar way, we have 
\begin{align}
\nonumber
p_{n+1} \leqs{}& 
\bigexpecin{A_2(R_2)}{\indexfct{u_{nT}\in A_2(R_2),\tau_{B}>nT} 
\probin{u_{nT}}{\tau_{B}>T}} \\
\nonumber
&{}+ \bigprobin{A_2(R_2)}{u_{nT}\notin A_2(R_2), \tau_{B}>nT} \\
\nonumber
\leqs{}& p_1 \bigprobin{A_2(R_2)}{u_{nT}\in A_2(R_2), \tau_{B}>nT}
+ \bigprobin{A_2(R_2)}{u_{nT}\notin A_2(R_2), \tau_{B}>nT} \\
\nonumber
=& p_1 p_n + (1-p_1)\bigprobin{A_2(R_2)}{u_{nT}\notin A_2(R_2),
\tau_{B}>nT} \\
\leqs{}& p_1 p_n + (1-p_1) q_n\;. 
\label{eqpot21:07} 
\end{align}
It follows by induction that 
\begin{equation}
 \label{eqpot21:08}
p_n \leqs p_1^n + n^2 q_1(1-p_1)\;. 
\end{equation} 
Putting together the different estimates, we obtain 
\begin{equation}
 \label{eqpot21:11}
\bigprobin{A_2(R_2)}{\tau_{A_1(R_0)^c}\leqs\tau_B}
\leqs p_n + r_n 
\leqs p_1^n + r_1 + n q_1\bigbrak{1+n(1-p_1)}\;.
\end{equation} 
It remains to estimate $p_1$, $q_1$ and $r_1$. Proposition~\ref{prop_moments2}
shows that 
\begin{equation}
 \label{moments105:10}
p_1 \leqs 1 - \frac12\e^{-H_1/\eps}\;, 
\end{equation} 
where $H_1=H_0+\eta$. We can estimate $q_1$ by 
\begin{equation}
 \label{moments105:11}
q_1 \leqs  \bigprobin{A_2(R_2)}{\norm{u_{T/2}}_{L^\infty} > R_1}
+ \bigprobin{A_1(R_1)}{\norm{u_{T/2}}_{\cC^\alpha} > R_2}\;,
\end{equation} 
and both terms can be bounded as in the proof of
Proposition~\ref{prop_moments3}. An appropriate choice of $R_1, R_2$ ensures
that $q_1 \leqs \e^{-H_1/\eps}$. Finally, by~\eqref{infdim08A} we also have 
\begin{equation}
 \label{moments105:20} 
r_1 = 
\bigprobin{A_2(R_2)}{\tau_{A_1(R_0)^c}\leqs T} \leqs 
\exp\biggset{-\frac{\kappa(T)}{2\eps} f(R_0 - M(1+R_2))}
\leqs \e^{-H_1/\eps}
\end{equation} 
for sufficiently large $R_0$. The choice 
\begin{equation}
 \label{moments105:21}
n = \Bigintpartplus{(4 \log 2) \e^{H_1/\eps}} 
\end{equation} 
yields $\log(p_1^n) \leqs -\frac12n\e^{-H_1/\eps} \leqs -2\log(2)$ so that
$p_1^n\leqs 1/4$, while the other terms in~\eqref{eqpot21:11} are
exponentially small. This concludes the proof. 
\end{proof}

Combining Propositions~\ref{prop_moments1}, \ref{prop_moments3} and
\ref{prop_moments4}, we finally get the main result of this section. 

\begin{cor}[Main estimate on the moments of $\tau_B$]
\label{cor_tauB_inf} 
Let $B\subset\cC^0(\T^1)$ be a non-empty, bounded open set in the
$\norm{\cdot}_{L^\infty}$-topology. Then for all $R_0,\eta>0$, there exist
constants $\eps_0 > 0$ and $T_0 < \infty$
such that 
\begin{equation}
 \label{moments08}
\bigexpecin{A_1(R_0)}{\tau_B} \leqs T_0\e^{(H_0(B)+\eta)/\eps}
\qquad
\text{and}
\qquad
\bigexpecin{A_1(R_0)}{\tau_B^2} \leqs T_0^2\e^{2(H_0(B)+\eta)/\eps} 
\end{equation} 
for all $\eps < \eps_0$. 
\end{cor}
\begin{proof}
Making $R_0$ larger if necessary, we choose $R_2$ and $R_0 > R_2$ large enough
for the three previous results to hold, and such that $B \subsetneq A_1(R_0)$.
We apply the three-set argument Corollary~\ref{cor_ldf4} with
$A=A_2(R_2)\setminus B$ and $C=A_1(R_0)^c$.
Noting that $A\subset A_1(R_0)\setminus B$, we have 
\begin{equation}
 \label{moments08:01}
 \bigexpecin{A}{\tau^{n}_{B\cup C}}
 \leqs \bigexpecin{A_1(R_0)\setminus B}{\tau^{n}_{B\cup C}}
 \leqs n!T_0^n \e^{n(H_0+\eta)/\eps}
\end{equation} 
by Proposition~\ref{prop_moments3}. 
Since $\tau_{A\cup B}\leqs\tau_{A_2(R_2)}$ and $A_1(R_0)^c \subset 
A_2(R_2)^c$, we have 
\begin{equation}
 \label{moments08:02}
 \bigexpecin{C}{\tau^{n}_{A\cup B}}
 \leqs \bigexpecin{A_2(R_2)^c}{\tau_{A_2(R_2)}^n} \leqs n!T_0^n
\end{equation} 
by Proposition~\ref{prop_moments1}. Finally, 
 $\probin{A}{\tau_C<\tau_B} \leqs 1/2$ by 
Proposition~\ref{prop_moments4}. This shows the result for initial conditions
in $A_2(R_2)\setminus B$. Now we can easily extend these bounds to all initial
conditions in $A_1(R_0)$ by using Proposition~\ref{prop_moments1} and
restarting the process when it first hits $A_2(R_2)\setminus B$. 
\end{proof}

Note that in the proof of Proposition~\ref{prop_apriori_hit},
we apply this result when $B$ is a
neighbourhood of $u^*_+$. In that case, $H_0(B)$ is equal to the potential
difference between the transition state and the local minimum $u^*_{-}$. 


\subsection{Uniform bounds on moments of $\tau_B$ in finite dimension}
\label{ssec_momfin}

In this section, we derive bounds on the moments of first-hitting times,
similar to those in Corollary~\ref{cor_tauB_inf}, for the finite-dimensional
process, uniformly in the dimension. 

For $E=\cC^0(\T^1)$ and $d\in\N$, we denote by $E_d$ the finite-dimensional space
\begin{equation}
 \label{momfin01}
E_d = \biggsetsuch{u\in\cC^0(\T^1)}
{u(x)=\sum_{k\colon \abs{k}\leqs d} y_k e_k(x)\,,\; y_k\in\R}\;,
\end{equation} 
and by $\set{u_t^{(d)}}_{t\geqs0}$ the solution of the projected equation,
cf.~\eqref{Galerkin3}. Given a set $A\subset E$, we write $A_d=A\cap E_d$, and
denote by $\smash{\tau^{(d)}_{A_d}}$ the first time $u_t^{(d)}$ hits $A_d$,
while $\tau_A$ denotes as before the first time the infinite-dimensional process
$u_t$ hits $A$. 

\begin{prop}[Main estimate on moments of $\tau^{(d)}_{B_d}$]
\label{prop_tauB_finite} 
Let $B\subset E$ be an open ball of radius $r$ in the
$\norm{\cdot}_{L^{\infty}}$-norm. We assume that the centre of $B$ is some $w\in
E_{d}$. As in Corollary~\ref{cor_tauB_inf}, we define $A=A_{2}(R_{2})\setminus
B$. Then, there exist constants 
$\eps_0>0$ and $H_1, T_1<\infty$ such that for any $\eps\in(0,\eps_0)$, there is
a $d_0(\eps)\in\N$ such that 
\begin{equation}
 \label{momfin02}
\bigexpecin{A_d}{\tau^{(d)}_{B_d}} \leqs T_1\e^{H_1/\eps} 
\qquad
\text{and}
\qquad
\bigexpecin{A_d}{(\tau^{(d)}_{B_d})^2} \leqs T_1^2 \e^{2H_1/\eps} 
\end{equation} 
for all $d\geqs d_0(\eps)$. 
\end{prop}
\begin{proof}
We fix constants $\delta, T>0$, and let $\Omega_d$ be the event 
\begin{equation}
 \label{momfin03:1}
\Omega_d = \biggset{\sup_{0\leqs t\leqs T} \norm{u_t^{(d)} - u_t}_{L^\infty}
\leqs \delta}\;.
\end{equation} 
Theorem~\ref{thm_Galerkin} shows that for given $\gamma<1/2$, there exists an almost surely finite
random variable $Z$ such that 
\begin{equation}
\label{momfin03:2}
\fP(\Omega_d^c) \leqs \bigprob{Z>\delta d^\gamma} \;.
\end{equation} 
Given $D\subset\cC^0(\T^1)$, we define the sets 
\begin{align}
\nonumber
D_{d,+} &= \bigsetsuch{u\in E_d}{\exists v\in D \text{ s.t. }
\norm{v-u}_{L^\infty} \leqs \delta}\;, \\
D_{d,-} &= \bigsetsuch{u\in E_d}{\setsuch{v\in E}{\norm{v-u}_{L^\infty} \leqs \delta}
\subset D}\;, 
 \label{momfin03:3}
\end{align} 
which satisfy $D_{d,-}\subset D_d\subset D_{d,+}$. Then for any initial
condition $u_0\in E_d$, we have the two inequalities 
\begin{align}
\nonumber
\bigprobin{u_0}{\tau^{(d)}_{D_{d,+}} > T} 
&\leqs  \bigprobin{u_0}{\tau_D > T} + \fP(\Omega_d^c)\;, \\
\bigprobin{u_0}{\tau^{(d)}_{D_{d,-}} \leqs T} 
&\leqs  \bigprobin{u_0}{\tau_D \leqs T} + \fP(\Omega_d^c)\;.
\label{momfin03:4}
\end{align} 
Let $R_0$ be as in the proof of Corollary~\ref{cor_tauB_inf}, and define the
sets 
\begin{align}
\nonumber
C &= \bigsetsuch{u\in E}{\norm{u}_{L^\infty} \geqs R_0}\;, \\
\nonumber
C'&= \bigsetsuch{u\in E}{\norm{u}_{L^\infty} \geqs R_0+2\delta}\;, \\
B'&= \bigsetsuch{u\in E}{\setsuch{v\in E}{\norm{v-u}_{L^\infty} \leqs \delta}
\subset B}\;.
\label{momfin03:5}
\end{align}
We assume $\delta$ to be small enough for $B'$ to be non-empty. Note that $C$
and $C'$ are the complements of open balls in the
$\norm{\cdot}_{L^{\infty}}$-norm while $B'$ is the open ball of radius
$r-\delta$ around the center $w$ of $B$.

Applying the three-set argument~\eqref{ldf30A} to the triple $(A_{d,+}\setminus B_{d},B_{d},C_{d,-})$, where the sets are disjoint for sufficiently large $R_{0}$, we get
\begin{equation}
\label{momfin03:6}
\bigexpecin{A_d}{\tau^{(d)}_{B_d}}
\leqs 
\frac{\bigexpecin{A_{d,+}} {\tau^{(d)}_{B_{d}\cup C_{d,-}}} +
\bigprobin{A_{d,+}} {\tau^{(d)}_{C_{d,-}} < \tau^{(d)}_{B_{d}}} 
\bigexpecin{C_{d,-}} {\tau^{(d)}_{A_{d,+}\cup B_{d}}} }
{1 - \bigprobin{A_{d,+}} {\tau^{(d)}_{C_{d,-}} < \tau^{(d)}_{B_{d}}} }\;. 
\end{equation} 
Using the facts that $A_d\subset A_{d,+}$, 
$(B')_{d,+} = B_d$
and $(C')_{d,+}\subset C_{d,-}$, we now reduce the estimation of each of the
terms on the right-hand side of~\eqref{momfin03:6} to probabilities that can be
controlled, via~\eqref{momfin03:4}, in terms of the infinite-dimensional
process.

Since $C_{d,-}\subset E_d\setminus (A_{2}(R_{2}))_{d,+}$ for sufficiently large $R_{0}$ and $A_{d,+}\cup B_{d}\supset (A_{2}(R_{2}))_{d,+}$, we have by Lemma~\ref{lem_ldf3}
\begin{equation}
 \label{momfin03:7}
\bigexpecin{C_{d,-}} {\tau^{(d)}_{A_{d,+}\cup B_{d}}}
\leqs
\bigexpecin{E_d\setminus  (A_{2}(R_{2}))_{d,+}} {\tau^{(d)}_{ (A_{2}(R_{2}))_{d,+}}}
\leqs 
\frac{T}{1-\bigprobin{E_d\setminus (A_{2}(R_{2}))_{d,+}} {\tau^{(d)}_{ (A_{2}(R_{2}))_{d,+}} > T} }\;.
\end{equation} 
By~\eqref{momfin03:4} we have for any $u_0\in E_d\setminus  (A_{2}(R_{2}))_{d,+}$ 
\begin{equation}
 \label{momfin03:8}
\bigprobin{u_0} {\tau^{(d)}_{ (A_{2}(R_{2}))_{d,+}} > T}
\leqs \bigprobin{u_0} {\tau_{A_{2}(R_{2})} > T} + \fP(\Omega_d^c)\;. 
\end{equation} 
As we have seen in~\eqref{moments03_b}, 
the first term on the right-hand side can be bounded by $1/2$. As for the
second term,~\eqref{momfin03:2} shows that it is smaller than $1/4$ for $d\geqs d_{0}(\eps)$
large enough. Hence the right-hand side of~\eqref{momfin03:7} can be bounded by
$4T/3$. 

The term $\bigexpecin{A_{d,+}} {\tau^{(d)}_{B_{d}\cup C_{d,-}}}\leqs \bigexpecin{A_{d,+}} {\tau^{(d)}_{(B')_{d,+}\cup (C')_{d,+}}}$ can be estimated in a similar way, by comparing with 
$\probin{A_{d,+}}{\tau_{B'\cup C'}>T}$ and proceeding as in the proof of Proposition~\ref{prop_moments3}, cf.~\eqref{moments07:1}. 

Finally, we have the bounds 
\begin{align}
\nonumber
\bigprobin{A_{d,+}} {\tau^{(d)}_{C_{d,-}} < \tau^{(d)}_{B_{d}}}
&\leqs \bigprobin{A_{d,+}} {\tau^{(d)}_{C_{d,-}} \leqs T} +
\bigprobin{A_{d,+}} {\tau^{(d)}_{(B')_{d,+}} > T} \\
&\leqs \bigprobin{A_{d,+}}{\tau_C\leqs T} + \bigprobin{A_{d,+}}{\tau_{B'}> T} +
2\fP(\Omega_d^c)\;.
\label{momfin03:9}
\end{align}
Proposition~\ref{prop_moments2} and~\eqref{moments105:20} show that 
the sum of the first two terms on the right-hand side can be bounded by
$1-\frac12\e^{-H_1/\eps}$.
The third term can be bounded by $\frac14 \e^{-H_1/\eps}$, provided $d$ is
larger than some (possibly large) $d_0(\eps)$. 

This completes the bound on the first moment, and the second moment can be
estimated in the same way.
\end{proof}

\subsection{Bounds on the equilibrium potential in finite dimension}
\label{ssec_eqpot2}

The aim of this subsection is to obtain bounds on the equilibrium potential 
\begin{equation}
 \label{new_eqpot01}
h^{(d)}_{A_d,B_d}(u_0) = 
\bigprobin{u_0}{\tau^{(d)}_{A_d} < \tau^{(d)}_{B_d}}\;, 
\end{equation} 
when $A$ and $B$ are small open balls, in the $L^\infty$-norm, around the local
minima $u^*_\pm$ of the potential $V$, and as before $A_d=A\cap E_d$ and
$B_d=B\cap E_d$. We denote the centre of $A$ by $u^*_1$ and the centre of $B$
by $u^*_2$, where either $u^*_1=u^*_-$ and $u^*_2=u^*_+$ or vice versa.

We now derive a bound on $h^{(d)}_{A_d,B_d}(u_0)$, which is useful when $u_0$
lies in the basin of attraction of $B_d$. 

\begin{prop}
\label{new_prop_eqpot1}
Let $u^*_1$ and $u^*_2$ be two different local minima of $V$ and consider the
balls $A=\set{\norm{u-u^*_1}_{L^\infty}<r}$ and
$B=\set{\norm{u-u^*_2}_{L^\infty}<r}$, where $r$ is small enough to guarantee
$\overbar A\cap\overbar B=\emptyset$. Then for any $\eta>0$, there
exist $\eps_0=\eps_0(\eta)>0$, $d_0=d_0(\eta,\eps)<\infty$ and $H_0=H_0(\eta)>0$
such that 
\begin{equation}
 \label{new_eqpot04}
h^{(d)}_{A_d,B_d}(u_0) \leqs 4\e^{-H_0/\eps}
\end{equation} 
holds for all $\eps<\eps_0$, all $d\geqs d_0$ and all $u_0$ satisfying 
$H(u_0,A)\geqs\eta$ and $H(u_0,B)=0$. The result holds uniformly for $u_{0}$
from a  $\norm{\cdot}_{L^{\infty}}$-bounded subset of $E_d$.
\end{prop}
\begin{proof}
Fix a $u_{0}$ such that $H(u_0,A)\geqs\eta$ and $H(u_0,B)=0$. 
For any constant $T>0$, we can write 
\begin{equation}
 \label{new_eqpot03:01}
 h^{(d)}_{A_d,B_d}(u_0) \leqs 
\bigprobin{u_0}{\tau^{(d)}_{A_d}\leqs T} + 
\bigprobin{u_0}{\tau^{(d)}_{B_d}> T}\;.
\end{equation} 
For $0<\delta<r$, we define $\Omega_d=\Omega_d(\delta)$ as
in~\eqref{momfin03:1}. Then we have, in a way similar to~\eqref{momfin03:4}, 
\begin{align}
\nonumber
\bigprobin{u_0}{\tau^{(d)}_{A_d}\leqs T} 
&\leqs  \bigprobin{u_0}{\tau_{A_+} \leqs T} + \fP(\Omega_d^c)\;, \\
\bigprobin{u_0}{\tau^{(d)}_{B_d}> T} 
&\leqs  \bigprobin{u_0}{\tau_{B_-} > T} + \fP(\Omega_d^c)\;, 
 \label{new_eqpot03:02}
\end{align} 
where $A_+=\set{\norm{u-u^*_1}_{L^\infty}<r+\delta}$ and
$B_-=\set{\norm{u-u^*_2}_{L^\infty}<r-\delta}$.
We choose $\delta$ small enough that 
$H(u_0,A_+)\geqs \eta/2$. 
The large-deviation principle shows that 
\begin{equation}
 \label{new_eqpot03:03}
\limsup_{\eps\to0} 2\eps \log \bigprobin{u_0}{\tau_{A_+} \leqs T} =
-2H(u_0,A_+) \leqs -\eta\;. 
\end{equation} 
Thus there exists $\eps_0(\eta)>0$, independent
of $T$, such that 
\begin{equation}
 \label{new_eqpot03:05}
 \bigprobin{u_0}{\tau_{A_+} \leqs T} \leqs \e^{-\eta/4\eps}
\end{equation} 
holds for all $\eps < \eps_0$. Choosing $d\geqs d_0(\eta,\eps)$
where $d_0$ is large enough that $\fP(\Omega_{d_0}^c) \leqs \e^{-\eta/4\eps}$,
we have  
\begin{equation}
 \label{new_eqpot03:06}
\bigprobin{u_0}{\tau^{(d)}_{A_d}\leqs T} \leqs 2\e^{-\eta/4\eps}
\end{equation} 
for $\eps < \eps_0$ and $d\geqs d_0$. 

To estimate the second term in~\eqref{new_eqpot03:02}, we assume $\delta<r/2$
and let $B'_-=\set{\norm{u-u^*_2}_{L^\infty}<r-2\delta}$. Assume $T$ is large
enough that the deterministic solution starting in $u_0$ reaches $B'_-$ in time
$T$. Then the large-deviation principle in~\cite{Freidlin88} shows that the
stochastic sample path starting in $u_0$ is unlikely to leave a tube of size
$\delta$ in the $L^\infty$-norm around the deterministic solution before time
$T$, which implies that there exists $\kappa>0$ such that 
\begin{equation}
 \label{new_eqpot03:07}
 \bigprobin{u_0}{\tau_{B_-} > T} \leqs \e^{-\kappa\delta^2/\eps}
\end{equation} 
for $\eps$ small enough. This implies the result, with
$H_0=\kappa\delta^2\wedge\eta/4$. 

As for the uniformity in $u_{0}$, note that after a first finite time $T_{1}$ we may assume that the process has reached a compact subset, cf.~the proof of Proposition~\ref{prop_moments3}. Restarting from this compact subset a standard compactness argument yields the uniformity of $\eps_{0}$, $\delta_{0}$ and $H_{0}$ in $u_{0}$.
\end{proof}

Next we derive a more precise bound, which is useful in situations where we
know $V[u_0]$ explicitly. 

\begin{prop}
\label{new_prop_eqpot2}
Let $u^*_1$ and $u^*_2$ be two different local minima of $V$ and consider the
balls $A=\set{\norm{u-u^*_1}_{L^\infty}<r}$ and
$B=\set{\norm{u-u^*_2}_{L^\infty}<r}$. Assume that $r$ is small enough that
$\overbar A\cap\overbar B=\emptyset$. Then for any $\eta,M>0$, there exist
$\eps_0=\eps_0(\eta,M)>0$ and $d_0=d_0(\eta,M,\eps)<\infty$ such that 
\begin{equation}
 \label{new_eqpot14}
h^{(d)}_{A_d,B_d}(u_0) \leqs 3 \biggpar{\e^{-[\Vbar(u_0,A)-\eta]/\eps} +
\e^{-1/\eta\eps}}
\end{equation} 
holds for all $\eps<\eps_0$, all $d\geqs d_0$ and all $u_0\in E_d$ such 
that $V[u_0]\leqs M$.
\end{prop}
\begin{proof}
Fix a $u_{0}$ with $V[u_0]\leqs M$.
We decompose the equilibrium potential in the same way as
in~\eqref{new_eqpot03:01} and~\eqref{new_eqpot03:02}. 
It follows from~\eqref{new_eqpot03:03} and Lemma~\ref{lem_ldp_Vbar} that 
there exists $\eps_0(\eta)>0$, independent
of $T$, such that 
\begin{equation}
 \label{new_eqpot14:05}
 \bigprobin{u_0}{\tau_{A_+} \leqs T} \leqs \e^{-(\Vbar(u_0,A_+)-\eta/2)/\eps}
\end{equation} 
holds for all $\eps < \eps_0$. We choose $\delta$ in the definition of $A_{+}$ small enough that 
$\Vbar(u_0,A_+)\geqs \Vbar(u_0,A)-\eta/2$, and finally $d\geqs d_0(\eta,\eps)$
where $d_0$ is large enough that $\fP(\Omega_{d_0}^c) \leqs \e^{-1/\eta\eps}$.
This shows that 
\begin{equation}
 \label{new_eqpot14:06}
\bigprobin{u_0}{\tau^{(d)}_{A_d}\leqs T} \leqs \e^{-(\Vbar(u_0,A)-\eta)/\eps} +
\e^{-1/\eta\eps}
\end{equation} 
for $\eps < \eps_0$ and $d\geqs d_0$. 

In order to estimate $\probin{u_0}{\tau_{B_-}>T}$, we let
$D(\kappa)$ be the set of $u\in E$ such that $\Vbar(u,A_+)>0$ 
and $\norm{\nabla V[u]}_{L^2}>\kappa$. Then we can decompose
\begin{equation}
 \label{new_eqpot14:07}
 \bigprobin{u_0}{\tau_{B_-} > T}
\leqs \bigprobin{u_0}{\tau_{D(\kappa)^c} > T} +
\bigprobin{u_0}{\tau_{F(\kappa)} \leqs T}\;.
\end{equation} 
where $F(\kappa) = D(\kappa)^c\cap B_-^c$. 
Now the same argument as above shows that 
\begin{equation}
 \label{new_eqpot14:08}
 \bigprobin{u_0}{\tau_{F(\kappa)} \leqs T} \leqs
\e^{-(\Vbar(u_0,F(\kappa))-\eta/2)/\eps}\;.
\end{equation} 
Note that $\lim_{\kappa\to0}\Vbar(u_0,F(\kappa))=\Vbar(u_0,F(0))$. 
Let us show that $\Vbar(u_0,F(0))=\Vbar(u_0,A_+)$ provided $\delta$ is small
enough. We proceed in two steps:
\begin{enum}
\item 	First we show that $\Vbar(u_0,F(0))\leqs\Vbar(u_0,A_+)$. 
Observe that $F(0)^c= D(0)\cup B_-$. The fact that
$A$ and $B$ have disjoint closure implies that $A_+\cap B_-=\emptyset$ for 
sufficiently small~$\delta$. The fact that $\Vbar(u,A_+)>0$ in $D(0)$ shows  
that $A_+\cap D(0)=\emptyset$. It follows that $A_+\cap F(0)^c=\emptyset$, and
thus $A_+ \subset F(0)$, which implies $\Vbar(u_0,F(0))\leqs\Vbar(u_0,A_+)$. 

\item  Assume by contradiction that $\Vbar(u_0,F(0))<\Vbar(u_0,A_+)$.  Then
there must exist a path $\ph$, connecting $u_0$ to a point $u\in F(0)$, on which
the potential remains strictly smaller than $\Vbar(u_0,A_+)+V[u_0]$. If we can
show that $\Vbar(u,A_+)=0$, then this implies that we can connect $u_0$ to
$A_+$, via $u$, by a path on which the potential remains strictly smaller than
$\Vbar(u_0,A_+)+V[u_0]$, contradicting the definition of $\Vbar(u_0,A)$. 

It thus remains to show that $\Vbar(u,A_+)=0$. Note that 
\begin{equation}
 \label{new_eqpot14:085}
 D(0) = \bigsetsuch{u}{\Vbar(u,A_+)>0, \nabla V[u]\neq 0}\;. 
\end{equation} 
Since $u\in F(0)=D(0)^c\cap B_-^c$, we have $u\neq u^*_2$ and either 
$\Vbar(u,A_+)=0$, or $\nabla V[u]= 0$. However, the assumptions imply that 
$u_{2}^{*}$ is the only stationary point in the
set $\setsuch{u}{\Vbar(u,A_{+})>0}$, so that necessarily $\Vbar(u,A_+)=0$. 
\end{enum}
We have thus proved that $\Vbar(u_0,F(0))=\Vbar(u_0,A_+)$, and it follows that
there exists a $\delta_0(\eta)$ such that for $\delta<\delta_0(\eta)$
\begin{equation}
 \label{new_eqpot14:09}
 \bigprobin{u_0}{\tau_{F(\kappa)} \leqs T} \leqs
\e^{-(\Vbar(u_0,A_+)-\eta)/\eps}\;.
\end{equation} 
for all $\kappa<\kappa_0$. 

It remains to estimate the first term on the right-hand side
of~\eqref{new_eqpot14:07}. Let $\ph$ be a continuous path starting in $u_0$ and
remaining in $\cD(\kappa)$ up to time $T$. Then its rate function satisfies 
\begin{align}
\nonumber
I(\ph) 
&\geqs \int_0^{T} \int_0^L \Bigbrak{
-\ph''_t(x) + U'(\ph_t(x))} \dot\ph_t(x) \6x 
 \6t + \frac12\int_0^{T} \int_0^L \Bigbrak{
-\ph''_t(x) + U'(\ph_t(x))}^2 \6x 
 \6t \\
&\geqs V[\ph_T] - V[u_0] + \frac12 \kappa^2 T\;,
 \label{new_eqpot14:10}
\end{align}
where we have used the fact that the second integral is proportional to
$\norm{\nabla V}_{L^2}^2$, cf.~\eqref{bp12}. The large-deviation principle
implies that 
\begin{equation}
 \label{new_eqpot14:11}
\limsup_{\eps\to0} 2\eps\log \bigprobin{u_0}{\tau_{D(\kappa)^c} > T}
\leqs - \biggbrak{\frac12\kappa^2 T + \inf_{D(\kappa)}V - V[u_0]}\;. 
\end{equation} 
Since $V[u_0]\leqs M$, we can find for any $\kappa>0$, a $T=T(\kappa,M)$
such that the right-hand side is smaller than $-\Vbar(u_0,A)$. 

Finally note that $\setsuch{u_{0}}{V[u_0]\leqs M}$ is contained in a closed ball in the $\cC^{1/2}$-norm, 
so that a standard compactness argument allows to choose $\eps_{0}$ and $d_{0}$ uniformly in $u_{0}$ 
from this set. This concludes the proof.
\end{proof}



\section{Estimating capacities}
\label{sec_cap}



\subsection{Neumann b.c.}
\label{ssec_Neumann}

We consider the potential energy 
\begin{equation}
 \label{capn00}
V[u] =  \int_0^L \biggbrak{\frac12 u'(x)^2 + U(u(x))} \6x
\end{equation}
for functions $u(x)$ containing at most $2d+1$ nonvanishing Fourier modes and
satisfying Neumann boundary conditions, that is 
\begin{equation}
 \label{capn01}
u(x) = \sum_{k=-d}^d z_k \frac{\e^{ik\pi x/L}}{\sqrt{L}} 
= y_0\frac{1}{\sqrt{L}} + \sum_{k=1}^d y_k \sqrt{\frac 2L}\cos(k \pi x/L)\;,
\end{equation} 
where $y_0=z_0$ and $y_k=\sqrt{2}z_k=\sqrt{2}z_{-k}$ for $k\geqs1$. 
The expression $\Vhat$ of the potential in Fourier variables follows 
from~\eqref{nf06} and~\eqref{nf07}, with the sums restricted to $-d\leqs k\leqs
d$. 

Note that 
\begin{equation}
 \label{capn01B}
 u(L-x) = y_0\frac{1}{\sqrt{L}} + \sum_{k=1}^d (-1)^k y_k \sqrt{\frac 2L}\cos(k
\pi x/L)\;,
\end{equation} 
so that the fact that $V[u]=V[u(L-\cdot)]$ implies the symmetry
\begin{equation}
 \label{capn_symmetry}
\Vhat(y_0,y_1,\dots,y_d) = \Vhat(y_0,-y_1,\dots,(-1)^dy_d)\;. 
\end{equation} 
Our aim is to estimate the capacity $\capacity_{A}(B)$, where $A$ is a ball of
radius $r$ in the $L^\infty$-norm around the stationary point $u^*_-$, and $B$
is a ball of radius $r$ around $u^*_+$. Note that $u^*_-$ has $y$-coordinates 
$(u_-\sqrt{L},0,\dots,0)$ and $u^*_+$ has $y$-coordinates 
$(u_+\sqrt{L},0,\dots,0)$. We will rely on the variational
representation of capacities 
\begin{equation}
 \label{capn02}
\capacity_A(B) = \inf_{h\in\cH_{A,B}} \Phi_{(A\cup B)^c}(h)\;, 
\end{equation} 
in terms of the Dirichlet form 
\begin{equation}
 \label{capn03}
\Phi_D(h) = \eps \int_D \e^{-\Vhat(y)/\eps} \norm{\nabla h(y)}_{\ell^2}^2\6y\;,
\end{equation} 
where in~\eqref{capn02}, $\cH_{A,B}$ denotes the set of functions $h$ satisfying the boundary conditions $h=1$ on $\partial A$ and $h=0$ on $\partial B$ for which
$\Phi_{(A\cup B)^c}(h)$ is defined and finite.

\subsubsection{$L<\pi$}

We consider first the case where $L\leqs\pi-c$ for some constant $c>0$. We know
that in this case, $\Vhat$ has only three stationary points, all lying on the
$y_0$-axis. One of them is the origin $O$, where the Hessian of $\Vhat$ has
eigenvalues 
\begin{equation}
 \label{cap_s01}
\lambda_k = -1 + \biggpar{\frac{k\pi}L}^2\;,
\qquad
k=0,\dots,d\;. 
\end{equation} 
Thus $O$ is a saddle with one-dimensional unstable manifold, which in this
case is contained in the $y_0$-axis. Let $\cW^{\math s}(O)$ denote the
$d$-dimensional stable manifold of the origin. 

\begin{lemma}[Growth of the potential on the stable manifold]
\label{lem_Ws} 
There exists a constant $m_0>0$ such that for all $y\in\cW^{\math s}(O)$,
\begin{equation}
 \label{cap_s01a}
 \Vhat(y) \geqs m_0 \norm{y}_{H^1}^2\;.
\end{equation} 
\end{lemma}

\begin{proof}
Let $y_\perp=(y_1,\dots,y_d)$. 
The centre-stable manifold theorem for
differential equations in Banach spaces~\cite[Theorem~1.1]{Gallay_92}
shows that $\cW^{\math s}(O)$ can be locally described by a graph of the form
$y_0=g(y_\perp)$. More precisely, the nonlinear part of $\nabla\Vhat(y)$ being
of order $\norm{y}_{H^s}^2$ for $s>1/4$, there exist constants $\rho,
M>0$ such that 
\begin{equation}
 \label{cap_s01a:1} 
\bigabs{g(y_\perp)} \leqs M \norm{y_\perp}_{H^s}^2 
\qquad
\forall y_\perp \colon \norm{y_\perp}_{H^s} \leqs \rho\;.
\end{equation} 
Since 
\begin{equation}
 \label{cap_s01a:2}
\Vhat(y) = \frac12 \sum_{k=0}^d \lambda_k y_k^2 +
\bigOrder{\norm{y}_{H^s}^3}
\end{equation} 
holds for all $s>1/4$, we have in particular 
\begin{equation}
 \label{cap_s01a:3}
\Vhat(g(y_\perp),y_\perp) = -\frac12 g(y_\perp)^2 
+ \frac12 \sum_{k=1}^d \lambda_k y_k^2 +
\bigOrder{\norm{y}_{H^1}^3}
\end{equation} 
whenever $\norm{y_\perp}_{H^1} \leqs \rho$. Thus using~\eqref{cap_s01a:1}
to bound $g(y_\perp)^2$, we obtain the existence of constants $m_1, \rho_1>0$
such that 
\begin{equation}
 \label{cap_s01a:4}
\Vhat(y) \geqs m_1 \norm{y}_{H^1}^2 
\qquad
\forall y\in\cW^{\math s}(O) \colon \norm{y}_{H^1}\leqs\rho_1\;.
\end{equation} 
We have used the fact that $\norm{y}_{H^1}^2 = \abs{y_0}^2 +
\norm{y_\perp}_{H^1}^2$ and estimated $\abs{y_0}^2$ on the stable
manifold by applying~\eqref{cap_s01a:1} once more. 
A similar computation shows that
\begin{equation}
 \label{cap_s01a:5} 
-\nabla\Vhat(y)\cdot \nabla( \norm{y}_{H^1}^2)<0 
\qquad
\forall y\in\cW^{\math s}(O) \colon \norm{y}_{H^1}\leqs\rho_1\;,
\end{equation} 
that is, the vector field $-\nabla\Vhat(y)$ points inward the ball of radius
$\rho_1$ on the stable manifold. 
By definition of the stable manifold, $\Vhat$ has to decrease on $\cW^{\math
s}(O)$ along orbits of the gradient flow $\dot y=-\nabla\Vhat(y)$. We thus
conclude from~\eqref{cap_s01a:4} and~\eqref{cap_s01a:5} that 
\begin{equation}
 \label{cap_s01a:6} 
\Vhat(y) \geqs m_1 \rho_1^2 
\qquad
\forall y\in\cW^{\math s}(O) \colon \norm{y}_{H^1}\geqs\rho_1\;.
\end{equation} 
Next, recall that by Lemma~\ref{lem_V1}, there exist constants
$\alpha,\beta>0$ such that 
\begin{equation}
 \label{cap_s01a:7}
\Vhat(y) \geqs -\alpha + \beta \norm{y}_{H^1}^2 
\end{equation} 
for all $y\in\R^{d+1}$. Define $\gamma>0$ by $-\alpha+\beta\gamma^2=1$. Then for
all $y\in\cW^{\math s}(O)$ such that
$\rho_1\leqs\norm{y}_{H^1}\leqs\gamma$, we have 
\begin{equation}
 \label{cap_s01a:8}
\Vhat(y) \geqs m_1\rho_1^2 = m_1\frac{\rho_1^2}{\gamma^2}\gamma^2
\geqs  m_1\frac{\rho_1^2}{\gamma^2}\norm{y}_{H^1}^2\;.
\end{equation} 
Together with~\eqref{cap_s01a:4} and~\eqref{cap_s01a:7} for $\norm{y}_{H^1}>\gamma$, this proves~\eqref{cap_s01a}, with the choice 
$m_0=\min\set{m_1, (m_1\rho_1^2/\gamma^2), (1/\gamma^2)}$. 
\end{proof}

\begin{prop}[Upper bound on the capacity]
\label{prop_cap_neumann_smallL_upper}
There exist constants $r_0>0$ and $\eps_0>0$ such that 
\begin{equation}
 \label{cap_s02}
\capacity_A(B) \leqs \frac{\eps}{\sqrt{2\pi\eps}} 
\biggpar{\prod_{k=1}^d \sqrt{\frac{2\pi\eps}{\lambda_k}}\,}
\bigbrak{1 + c_+\eps^{1/2}\abs{\log\eps}^{3/2}}
\end{equation} 
holds for all $r<r_0$, all $\eps<\eps_0$ and all $d\geqs1$, 
where the constant $c_+$ is independent of $\eps$ and $d$.
\end{prop}
\begin{proof}
Choosing the radius $r$ of the balls $A$ and $B$ small enough, we can ensure that $A$ and $B$ lie at a
$L^\infty$-distance of order $1$ from the stable manifold $\cW^{\math s}(O)$. 

By the variational principle~\eqref{capn02}, it is sufficient to construct a
particular function $h_+\in\cH_{A,B}$ for which the claimed upper bound holds. 
We define $h_+$ separately in different sets $D, S$ defined below, and the
remaining part of $\R^{d+1}$. Let 
\begin{equation}
 \label{cap_s02:1}
\delta_k = \sqrt{\frac{c_k\eps\abs{\log\eps}}{\abs{\lambda_k}}}\;, 
\qquad
k=0,\dots d\;, 
\end{equation}
with $c_k=c_0(1+\log(1+k))$. We will choose $c_{0}$ sufficiently large below.
We set 
\begin{equation}
 \label{cap_s02:2}
D = \prod_{k=0}^d [-\delta_k,\delta_k]\;. 
\end{equation}  
Note that for any $s<1/2$ one has
\begin{equation}
 \label{cap_s023}
 \norm{y}_{H^s}^2
= \sum_{k=0}^d \frac{c_k(1+k^2)^s}{\abs{\lambda_k}}\eps\abs{\log\eps} = 
\Order{\eps\abs{\log\eps}}
\end{equation}
for any $y\in D$, uniformly in $d$. By~\eqref{cap_s01a:2} we thus have 
\begin{equation}
 \label{cap_s02:3b}
\Vhat(y) = \frac12 \sum_{k=0}^d \lambda_k y_k^2 +
\bigOrder{\eps^{3/2}\abs{\log\eps}^{3/2}}
\end{equation} 
for all $y\in D$, where again the remainder is uniformly bounded in the
dimension $d$. On $D$, we define $h_+$ by 
\begin{equation}
 \label{cap_s02:4}
h_+(y) = f(y_0) \defby 
\int_{y_0}^{\delta_0} \frac{\e^{\Vhat(t,0,\dots,0)/\eps}}
{\displaystyle
\int_{-\delta_0}^{\delta_0} \e^{\Vhat(s,0,\dots,0)/\eps}\6s}
\6t\;.
\end{equation}
The contribution of $h_{+}$ on $D$ to the Dirichlet form is given by 
\begin{align}
\nonumber
\Phi_D(h_+) &= \eps \int_D f'(y_0)^2 \e^{-\Vhat(y)/\eps} \6y \\
&= \eps \int_D \frac{\e^{-\Vhat(y)/\eps + 2\Vhat(y_0,0,\dots,0)/\eps}}
{\displaystyle
\Biggpar{\int_{-\delta_0}^{\delta_0} \e^{\Vhat(t,0,\dots,0)/\eps}\6t}^2}
\6y\;.
\label{cap_s02:5}
\end{align}
Using the expression~\eqref{cap_s02:3b} of the potential, one readily gets 
\begin{equation}
 \label{cap_s02:6}
\Phi_D(h_+) \leqs \eps 
\Biggpar{\int_{-\delta_0}^{\delta_0} \e^{-y_0^2/2\eps}\6y_0}^{-1}
\prod_{k=1}^d \int_{-\delta_k}^{\delta_k} \e^{-\lambda_k y_k^2/2\eps}\6y_k
\bigbrak{1+\bigOrder{\eps^{1/2}\abs{\log\eps}^{3/2}}}\;,
\end{equation}  
which implies that $\Phi_D(h_+)$ its bounded above by the right-hand side of~\eqref{cap_s02}, 
provided $c_0$ is chosen large enough. 

We now continue $h_+$ outside the set $D$. Let $S$ be a layer of thickness 
of order $\sqrt{\eps\abs{\log\eps}}$ in $\norm{\cdot}_{\ell_{2}}$-norm around the stable manifold $\cW^{\math
s}(O)$. We set $h_+=1$ in the connected component of ${\R^{d+1}\setminus~S}$
containing $A$, $h_+=0$ in the connected component of $\R^{d+1}\setminus S$
containing $B$, and interpolate $h_+$ in an arbitrary way inside $S$, 
requiring only
$\norm{\nabla h_+(y)}_{\ell^2}^2  \leqs M/(\eps\abs{\log\eps})$ for some
constant $M$. Then the contribution $\Phi_{\R^{d+1}\setminus S}(h_+)$ to the
capacity is zero, and it remains to estimate $\Phi_{S\setminus D}(h_+)$. By
Lemma~\ref{lem_Ws}, we have 
\begin{align}
\nonumber
\Phi_{S\setminus D}(h_+)
&\leqs\frac{M\eps}{\eps\abs{\log\eps}} \int_{S\setminus D}
\e^{-m_0\sum_{k=0}^d(1+k^2)y_k^2/\eps}\6y \\
\nonumber
&\leqs\frac{M\eps}{\eps\abs{\log\eps}}
\sum_{k=0}^d \prod_{j\neq k}\int_{-\infty}^\infty
\e^{-m_0(1+j^2)y_j^2/\eps}\6y_j 
\cdot2\int_{\delta_k}^\infty \e^{-m_0(1+k^2)y_k^2/\eps}\6y_k \\
&\leqs\frac{M\eps}{\eps\abs{\log\eps}}\prod_{j=0}^d
\sqrt{\frac{\pi\eps}{m_0(1+j^2)}}
\sum_{k=0}^d \eps^{\kappa c_k}\;,
 \label{cap_s02:10}
\end{align}
where $\kappa>0$ depends only on $m_0$ and $L$.  
Recalling the choice $c_k=c_0(1+\log(1+k))$, we find 
\begin{align}
\nonumber
\sum_{k=0}^d \eps^{\kappa c_k} &\leqs \eps^{\kappa c_0} + \int_0^{d}
\eps^{\kappa c_0(1+\log(1+x))} \6x \\
\nonumber
&= \eps^{\kappa c_0} + \eps^{\kappa c_0}
\int_1^{d+1} x^{-\kappa  c_0\abs{\log\eps}}\6x \\
&\leqs \eps^{\kappa c_0} \Biggbrak{1 + \frac{1}{\kappa
c_0\abs{\log\eps}-1}}\;,
 \label{cap_s02:11} 
\end{align} 
uniformly in $d$, provided $\kappa c_0\abs{\log\eps}>1$. Thus we can ensure
that $\Phi_{S\setminus D}(h_+)$ is negligible by making $c_0$ large enough. 
\end{proof}

\begin{prop}[Lower bound on the capacity]
\label{prop_cap_neumann_smallL_lower} 
There exist $r_0>0$, $\eps_0>0$ and $d_0(\eps)<\infty$ such that 
\begin{equation}
 \label{cap_s03}
\capacity_A(B) \geqs \frac{\eps}{\sqrt{2\pi\eps}} 
\biggpar{\prod_{k=1}^d \sqrt{\frac{2\pi\eps}{\lambda_k}}\,}
\bigbrak{1 - c_-\eps^{1/2}\abs{\log\eps}^{3/2}}
\end{equation} 
holds for all $r<r_0$, all $\eps<\eps_0$ and all $d\geqs d_0(\eps)$, 
where the constant $c_-$ is independent of $\eps$ and $d$.
\end{prop}
\begin{proof}
We write as before $y=(y_0,y_\perp)$, where $y_\perp=(y_1,\dots,y_d)$. Let 
\begin{equation}
 \label{cap_s04:1}
\widehat D_\perp = \prod_{k=1}^d [-\hat\delta_k, \hat\delta_k] 
\qquad
\text{with }
\hat\delta_k = \sqrt{\frac{\hat c_k\eps\abs{\log\eps}}{\lambda_k}}\;, 
\end{equation} 
where the constants $\hat c_k$ are of the form 
$\hat c_k=\hat c_0(1+\log(1+k))$. Note that as in the previous proof, this
implies
$\norm{y_\perp}_{H^s}=\Order{\sqrt{\eps\abs{\log\eps}}\,}$ for
$y_\perp\in\widehat D_\perp$ and all $s<1$. Given $\rho>0$, we set 
\begin{equation}
 \label{cap_s04:3}
\widehat D = [-\rho,\rho] \times \widehat D_\perp\;. 
\end{equation}
Let $h^*=h_{A,B}$ denote the equilibrium potential defined by $\capacity_A(B) = \Phi_{(A\cup B)^c}(h_{A,B})$, cf.~\eqref{capn02}. Then the capacity can  be bounded below as follows:\begin{align}
\nonumber
\capacity_A(B) &= \Phi_{(A\cup B)^c}(h^*) \\
\nonumber 
&\geqs \Phi_{\widehat D}(h^*) \\
\nonumber 
&= \eps\int_{\widehat D_\perp}\int_{-\rho}^{\rho}
\e^{-\Vhat(y_0,y_\perp)/\eps} \norm{\nabla h^*(y_0,y_\perp)}_{\ell^2}^2
\6y_0\6y_\perp \\
\nonumber 
&\geqs \eps\int_{\widehat D_\perp}\int_{-\rho}^{\rho}
\e^{-\Vhat(y_0,y_\perp)/\eps} \biggabs{\dpar{h^*}{y_0}(y_0,y_\perp)}^2
\6y_0\6y_\perp \\
&\geqs \eps\int_{\widehat D_\perp} 
\biggbrak{\inf_{f:f(-\rho)=h^*(-\rho,y_\perp),f(\rho)=h^*(\rho,y_\perp)}
\int_{-\rho}^{\rho}
\e^{-\Vhat(y_0,y_\perp)/\eps} f'(y_0)^2\6y_0}
\6y_\perp\;.
\label{cap_s04:4} 
\end{align}
Solving a one-dimensional Euler--Lagrange problem, we obtain that the infimum
is realised by the function $f$ such that 
\begin{equation}
 \label{cap_s04:5}
f'(y_0) = 
\frac{\bigbrak{h^*(\rho,y_\perp)-h^*(-\rho,y_\perp)}
\e^{\Vhat(y_0,y_\perp)/\eps}}{\displaystyle\int_{ -\rho}^{ \rho}
\e^{\Vhat(t,y_\perp)/\eps}\6t}\;. 
\end{equation} 
Substituting in~\eqref{cap_s04:4} and carrying out the integral over $y_0$, we
obtain
\begin{equation}
 \label{cap_s04:6}
\capacity_A(B) \geqs \eps \int_{\widehat D_\perp} 
\frac{\bigbrak{h^*(\rho,y_\perp)-h^*(-\rho,y_\perp)}^2}{\displaystyle\int_{
-\rho}^{\rho}
\e^{\Vhat(y_0,y_\perp)/\eps}\6y_0}
\6y_\perp \;.
\end{equation}
By \eqref{cap_s01a:1} (which also applies to the infinite-dimensional system)
and the fact that $\norm{y_\perp}_{H^s}=\Order{\sqrt{\eps\abs{\log\eps}}}$, any
point $(\rho,y_\perp)$ lies on the same side of the stable manifold
$\cW^{\math s}(O)$ as $u^*_+$. This implies that $H((\rho,y_\perp),B)=0$ while
$H((\rho,y_\perp),A)\geqs \eta$, where $\eta$ is uniform in $y_\perp\in\widehat
D_\perp$. We can thus apply Proposition~\ref{new_prop_eqpot1} to obtain 
the existence of $H_0>0$ such that 
\begin{equation}
 \label{cap_s04:6A}
 h^*(\rho,y_\perp) = \bigprobin{(\rho,y_\perp)}{\tau_A<\tau_B}
\leqs 4 \e^{-H_0/\eps}\;,
\end{equation} 
provided $\eps$ is small enough and $d$ is larger
than some $d_0(\eps)$. For similar reasons, we also have 
\begin{equation}
 \label{cap_s04:6B}
 h^*(-\rho,y_\perp) = 1 - \bigprobin{(-\rho,y_\perp)}{\tau_B<\tau_A}
\geqs 1 - 4 \e^{-H_0/\eps}\;.
\end{equation} 
Substituting in~\eqref{cap_s04:6}, we obtain 
\begin{equation}
 \label{cap_s04:6C}
\capacity_A(B) \geqs \eps \int_{\widehat D_\perp} 
\frac{\bigbrak{1 - 8 \e^{-H_0/\eps}}^2}{\displaystyle\int_{
-\rho}^{\rho}
\e^{\Vhat(y_0,y_\perp)/\eps}\6y_0}
\6y_\perp \;.
\end{equation}
Consider now, for fixed $y_\perp\in\widehat D_\perp$, the function 
$y_0\mapsto g(y_0)=\Vhat(y_0,y_\perp)$. It satisfies, for all $1/4<s<1/2$, 
\begin{align}
\nonumber
g(y_0) &= -\frac12 y_0^2 + \frac12 \sum_{k=1}^d \lambda_k y_k^2 
+ \bigOrder{\norm{y}_{H^s}^3}\;, \\
\nonumber
g'(y_0) &= -y_0 + \bigOrder{\norm{y}_{H^s}^2}
= -y_0 + \bigOrder{y_0^2} + \bigOrder{\eps\abs{\log\eps}}\;, \\
g''(y_0) &= -1 + \bigOrder{\norm{y}_{H^s}}
= -1 + \bigOrder{y_0} + \bigOrder{\eps^{1/2}\abs{\log\eps}^{1/2}}\;. 
 \label{cap_s04:7}
\end{align}
The assumption on $U$ being a double-well potential, the definitions of $A, B$
and the implicit-function theorem imply that $g$ admits a unique maximum at
$y_0^* = \bigOrder{\eps\abs{\log\eps}}$,
and we have 
\begin{align}
\nonumber
g(y_0^*) &= \frac12 \sum_{k=1}^d \lambda_k y_k^2 
+\bigOrder{\abs{y_0^*}^3}
+\bigOrder{\norm{y_\perp}_{H^s}^3}
= \frac12 \sum_{k=1}^d \lambda_k y_k^2 +
\bigOrder{\eps^{3/2}\abs{\log\eps}^{3/2}}\;, \\
g''(y_0^*) &= -1 +\bigOrder{\eps^{1/2}\abs{\log\eps}^{1/2}}\;.
 \label{cap_s04:9}
\end{align}
Thus by applying standard Laplace asymptotics, we obtain 
\begin{equation}
 \label{cap_s04:10}
 \int_{-\rho}^{\rho} \e^{\Vhat(y_0,y_\perp)/\eps}\6y_0
= \sqrt{2\pi\eps} \exp\biggset{\frac{1}{2\eps}\sum_{k=1}^d \lambda_k
y_k^2}
\bigbrak{1+\bigOrder{\eps^{1/2}\abs{\log\eps}^{3/2}}}\;.
\end{equation} 
Substituting in~\eqref{cap_s04:6C} yields 
\begin{align}
\nonumber
\capacity_A(B) 
&\geqs \frac{\eps}{\sqrt{2\pi\eps}}
\prod_{k=1}^d \int_{-\hat\delta_k}^{\hat\delta_k} \e^{-\lambda_ky_k^2/2\eps}\6y_k 
\bigbrak{1-\bigOrder{\eps^{1/2}\abs{\log\eps}^{3/2}}} \\
\nonumber
&= \frac{\eps}{\sqrt{2\pi\eps}}
\prod_{k=1}^d \biggpar{\sqrt{\frac{2\pi\eps}{\lambda_k}}
\bigbrak{1-\Order{\eps^{\hat c_k/2}}}}
\bigbrak{1-\bigOrder{\eps^{1/2}\abs{\log\eps}^{3/2}}} \\
&= \frac{\eps}{\sqrt{2\pi\eps}}
\biggpar{\prod_{k=1}^d \sqrt{\frac{2\pi\eps}{\lambda_k}}\,}
\biggbrak{1-\biggOrder{\sum_{k=1}^d \eps^{\hat c_k/2}}}
\bigbrak{1-\bigOrder{\eps^{1/2}\abs{\log\eps}^{3/2}}}\;, 
\label{cap_s04:11} 
\end{align}
and the result follows from the same estimate as in~\eqref{cap_s02:11}, taking
$\hat c_0\geqs 1$. 
\end{proof}

\subsubsection{$L$ near $\pi$}

We now turn to the case $\abs{L-\pi}\leqs c$, with $c$ small. Then the
eigenvalue $\lambda_1$ associated with the first Fourier mode satisfies
$\abs{\lambda_1}\leqs \eta$, where we can assume $\eta$ to be small by making
$c$ small. 

Recall from Proposition~\ref{prop_NormalForm} that if the local potential
$U$ is of class $\cC^5$, there exists a change of variables $y=z+g(z)$, with
$\norm{g(z)}_{H^t} = \Order{\norm{z}_{H^s}^2}$ for all
$5/12<s<1/2$ and $t<2s-1/2$, such that 
\begin{equation}
 \label{cap_Lpi_01}
\Vhat(z+g(z)) = \frac12 \sum_{k=0}^d \lambda_k z_k^2 + \frac12 C_4 z_1^4 +
\Order{\norm{z}_{H^s}^5} 
\end{equation} 
with $C_4>0$. 
Note that the factor $1/2$ in front of $C_4$ results from the change from
complex to real Fourier series. In order to localise this change of variables,
it will be
convenient to introduce a $\cC^\infty$ cut-off function
$\theta:\R^{d+1}\to[0,1]$ satisfying 
\begin{equation}
 \label{cap_Lpi_02}
\theta(z) = 
\begin{cases}
1 & \text{for $\norm{z}_{H^s} \leqs 1$\;,} \\
0 & \text{for $\norm{z}_{H^s} \geqs 2$\;.} 
\end{cases} 
\end{equation} 
Given $\rho>0$, we consider the potential 
\begin{equation}
 \label{cap_Lpi_03}
\Vtilde_\rho(z) = \Vhat\biggpar{z+\theta\Bigpar{\frac{z}{\rho}}g(z)}\;,  
\end{equation} 
which is equal to $\Vhat(z)$ for $\norm{z}_{H^s}\geqs 2\rho$, and to the
normal form~\eqref{cap_Lpi_01} for $\norm{z}_{H^s}\leqs \rho$. It what
follows, we will always assume that $\rho>\abs{\lambda_1}$.

The expression~\eqref{cap_Lpi_01} of the normal form shows that for
sufficiently small $\rho$, 
\begin{itemiz}
\item	if $\lambda_1\geqs0$, the origin $O$ is the only stationary point of
$\Vtilde_\rho$ in the ball $\norm{z}_{H^1}<\rho$, and $\lambda_0=-1$
is the only negative eigenvalue of the Hessian of $\Vtilde_\rho$ at $O$;
\item	if $\lambda_1<0$, the origin $O$ is a stationary point with two negative
eigenvalues, and there are two additional stationary points $P_\pm$ with
coordinates 
\begin{equation}
 \label{cap_Lpi_04}
z_1^\pm = \pm\sqrt{2\abs{\lambda_1}/C_4} + \Order{\lambda_1}\;, 
\quad
z_k^\pm = \Order{\lambda_1^2/\lambda_k} 
\quad\text{for $k=0,2\dots,d$\;.} 
\end{equation} 
The symmetry~\eqref{capn_symmetry} implies that $z_k^+=(-1)^k z_k^{-}$ and
$\Vhat(P_+)=\Vhat(P_-)$. 
The eigenvalues of the Hessian of the potential at $P_\pm$ are the same, owing
to the symmetry, and of the form 
\begin{equation}
 \label{cap_Lpi_05}
\mu_1 = - 2\lambda_1 + \Order{\abs{\lambda_1}^{3/2}}\;, 
\quad
\mu_k = \lambda_k + \Order{\abs{\lambda_1}^{3/2}} 
\quad\text{for $k=0,2\dots,d$\;,} 
\end{equation} 
which shows that $P_+$ and $P_-$ are saddles with a one-dimensional unstable
manifold, and a $d$-dimensional stable manifold. The unstable manifolds
necessarily converge to the two local minima of the potential. 
\end{itemiz}

The basins of attraction of the two minima of the potential are separated by a
$d$-dimensional manifold that we will denote $\cW^{\math s}$. For
$\lambda_1>0$, $\cW^{\math s}=\cW^{\math s}(O)$ is the stable manifold of the
origin. For $\lambda_1<0$, we have $\cW^{\math s}=\cW^{\math s}(O) \cup
\cW^{\math s}(P_-) \cup \cW^{\math s}(P_+)$. See for instance~\cite{Jolly_89}
for a picture of the situation. 

\begin{lemma}[Growth of the potential along $\cW^{\math s}$]
\label{lem_Ws_Lpi} 
Let $z_\perp=(z_2,\dots,z_d)$. 
There exist constants $\rho>0$, $m_0>0$ and $\eta>0$ such that for
$\abs{\lambda_1}<\eta$ and all $z\in\cW^{\math s}$, one has 
\begin{equation}
 \label{cap_Lpi_06}
\Vtilde_\rho(z) \geqs 
m_0 \biggbrak{z_0^2 
+ z_1^2 \wedge \Bigpar{\frac12\lambda_1z_1^2 + \frac12 C_4 z_1^4} 
+ \norm{z_\perp}_{H^1}^2}\;.
\end{equation} 
\end{lemma}
\begin{proof}
The manifold $\cW^{\math s}$ can be locally described by a graph 
$z_0=\psi(z_1,z_\perp)$, where 
\begin{equation}
 \label{cap_Lpi_06:1}
\abs{\psi(z_1,z_\perp)} \leqs M(z_1^4 + \norm{z_\perp}_{H^s}^4)
\qquad
\text{whenever } z_1^2 + \norm{z_\perp}_{H^s}^2 \leqs \rho_0^2
\end{equation} 
for any $s>1/4$ and some $M>0$ and $\rho_0>0$. This implies 
\begin{equation}
 \label{cap_Lpi_06:2}
\Vtilde_\rho(\psi(z_1,z_\perp),z_1,z_\perp) = 
\frac12\lambda_1 z_1^2 + \frac12 C_4 z_1^4 + \frac12\sum_{k=2}^d\lambda_k z_k^2 
+ \Order{\abs{z_1^5}+\norm{z_\perp}_{H^s}^5} 
\end{equation} 
for $z_1^2 + \norm{z_\perp}_{H^s}^2 \leqs (\rho_0\wedge\rho)^2 =
\rho_1^2$, and proves~\eqref{cap_Lpi_06} for $\norm{z}_{H^s}\leqs
\rho_1$. In particular, for $z_1^2 + \norm{z_\perp}_{H^1}^2 = \rho_1^2$,
we obtain the existence of a constant $m_2>0$ such that
$\smash{\Vtilde_\rho}(z)\geqs m_2\rho_1^4$, provided $\abs{\lambda_1}$ is small
enough. The remainder of the proof is similar to the proof of
Lemma~\ref{lem_Ws}.
\end{proof}

\begin{prop}[Upper bound on the capacity]
\label{prop_cap_neumann_Lpi_upper} 
There exist constants $\eps_0, \eta, c_+ > 0$ such that 
for $\eps<\eps_0$ and $d\geqs 1$, 
\begin{enum}
\item	If $0\leqs\lambda_1\leqs\eta$, then 
\begin{equation}
 \label{cap_Lpi_07}
\capacity_A(B) \leqs \frac{\eps}{\sqrt{2\pi\eps}}
\int_{-\infty}^\infty \e^{-u_1(y_1)/\eps}\6y_1 
\biggpar{\prod_{k=2}^d \sqrt{\frac{2\pi\eps}{\lambda_k}}\,}
\bigbrak{1+c_+ R(\eps,\lambda_1)}\;,
\end{equation} 
where 
\begin{equation}
 \label{cap_Lpi_08}
u_1(y_1) = \frac12 \lambda_1 y_1^2 + \frac12 C_4 y_1^4 
\end{equation}
and 
\begin{equation}
 \label{cap_Lpi_09}
R(\eps,\lambda) = 
\biggbrak{\frac{\eps\abs{\log\eps}^3}
{\lambda\vee\sqrt{\eps\abs{\log\eps}}}}^{1/2}\;. 
\end{equation} 

\item	If $-\eta < \lambda_1 < 0$, then  
\begin{equation}
 \label{cap_Lpi_10}
\capacity_A(B) \leqs 2\eps
\sqrt{\frac{\abs{\mu_0}}{2\pi\eps}}
\int_{0}^\infty \e^{-u_2(y_1)/\eps}\6y_1 
\Biggpar{\prod_{k=2}^d \sqrt{\frac{2\pi\eps}{\mu_k}}\,}
\e^{-\Vhat(P_\pm)/\eps} 
\bigbrak{1+c_+R(\eps,\mu_1)}\;,
\end{equation} 
where 
\begin{equation}
 \label{cap_Lpi_11}
u_2(y_1) = \frac12 C_4 \biggpar{y_1^2 - \frac{\mu_1}{4C_4}}^2\;.
\end{equation}
\end{enum}
\end{prop}
\begin{proof}
The proof is similar to those of~\cite[Proposition~5.1 and
Theorem~4.1]{BG2010}, the main difference lying in the dimension-dependence of
the domains of integration. 
The capacity can be bounded above by $\Phi_{A\cup B}(h_+)$ for any
$h_+\in\cH_{A,B}$. The change of variables $y=z+g(z)$ and~\eqref{nf11} lead to 
\begin{equation}
 \label{cap_Lpi_11:2}
\Phi_{A\cup B}(h_+) = \eps\int_{(A\cup B)^c} \e^{-\Vtilde_\rho(z)/\eps} 
\norm{\nabla h_+(z)}_{\ell^2}^2 
\Bigbrak{1+\indexfct{\norm{z}_{H^s}\leqs\rho}
\bigOrder{\norm{z}_{H^s}} } \6z\;.
\end{equation} 
Consider first the case $\lambda_1\geqs0$. Let $D$
be a box defined by~\eqref{cap_s02:2}, where we take $\delta_k$ as
in~\eqref{cap_s02:1} for $k\neq1$, while $\delta_1$ is the positive solution of
$u_1(\delta_1)=c_1\eps\abs{\log\eps}$, which satisfies 
\begin{equation}
 \label{cap_Lpi_11:1}
\delta_1^2 =
\biggOrder{\frac{\eps\abs{\log\eps}}{\lambda_1\vee\sqrt{\eps\abs{\log\eps}}}}
\;. 
\end{equation} 
Note that if $z\in D$, then $\norm{z}_{H^s} = \Order{\delta_1}$ for all
$s<1$. This ensures that the potential $\Vtilde_\rho$ is given by the normal
form~\eqref{cap_Lpi_01}. The rest of the proof then proceeds exactly as in
Proposition~\ref{prop_cap_neumann_smallL_upper}. We have slightly overestimated
the logarithmic part of the error terms to get more compact expressions. 

For $-c\sqrt{\eps\abs{\log\eps}}\leqs\lambda_1<0$, the proof is the same, with
$\delta_1$ of order $(\eps\abs{\log\eps})^{1/4}$. Note that in this case, the
potential at the saddles $P_\pm$ has order $\eps\abs{\log\eps}$, so that
$\e^{-\Vtilde_\rho(P_\pm)/\eps}$ is still close to $1$ for small~$c$. 

Finally, for
$-\eta\leqs\lambda_1<-c\sqrt{\eps\abs{\log\eps}}$, we evaluate separately the
capacities on each half-space $\set{z_1<0}$ and $\set{z_1>0}$. Each Dirichlet
form is dominated by the integral over a box around $P_+$, respectively $P_-$,
where the extension of the box in the $z_1$-direction is of order
$\sqrt{\eps\abs{\log\eps}/\mu_1}$. The main point is to notice that  
\begin{equation}
 \label{cap_Lpi_11:3}
u_1(z_1) = \frac12 C_4 \biggpar{y_1^2 - \frac{\mu_1}{48C_4}}^2 +
\Vtilde_\rho(P_\pm)
+ \Order{\mu_1^{3/2}y_1^2}
\end{equation} 
(see~\cite[Proposition~5.4]{BG2010}).
\end{proof}

\begin{remark}
\label{rem_cap_Lpi}
As shown in~\cite[Section~5.4]{BG2010}, the integrals of $\e^{-u_1(y_1)/\eps}$
and $\e^{-u_2^\pm(y_1)/\eps}$ can be expressed in terms of Bessel functions,
yielding the functions $\Psi_\pm$ given in~\eqref{psiplus}
and~\eqref{psiminus}.
\end{remark}

\begin{prop}[Lower bound on the capacity]
\label{prop_cap_neumann_Lpi_lower} 
There exist constants $\eps_0, \eta, c_- > 0$ and $d_0(\eps)<\infty$ such that 
for $\eps<\eps_0$ and $d\geqs d_0(\eps)$, 
\begin{enum}
\item	If $0\leqs\lambda_1\leqs\eta$, then 
\begin{equation}
 \label{cap_Lpi_12}
\capacity_A(B) \geqs \frac{\eps}{\sqrt{2\pi\eps}}
\int_{-\infty}^\infty \e^{-u_1(y_1)/\eps}\6y_1 
\Biggpar{\prod_{k=2}^d \sqrt{\frac{2\pi\eps}{\lambda_k}}\,}
\bigbrak{1-c_-R(\eps,\lambda_1)}\;,
\end{equation} 
where $u_1$ and $R$ are defined in~\eqref{cap_Lpi_08} and~\eqref{cap_Lpi_09}.

\item	If $-\eta < \lambda_1 < 0$, then  
\begin{equation}
 \label{cap_Lpi_13}
\capacity_A(B) \geqs 2\eps
\sqrt{\frac{\abs{\mu_0}}{2\pi\eps}}
\int_{0}^\infty \e^{-u_2(y_1)/\eps}\6y_1 
\Biggpar{\prod_{k=2}^d \sqrt{\frac{2\pi\eps}{\mu_k}}\,}
\e^{-\Vhat(P_\pm)/\eps} 
\bigbrak{1-c_-R(\eps,\mu_1)}\;,
\end{equation} 
where $u_2$ is the function defined in~\eqref{cap_Lpi_11}. 
\end{enum}
\end{prop}
\begin{proof}
For $-c\sqrt{\eps\abs{\log\eps}}\leqs\lambda_1\leqs\eta$, the proof is exactly
the same as the proof of Proposition~\ref{prop_cap_neumann_smallL_lower}, except
that $\hat\delta_1$ is defined in a similar way as $\delta_1$
in~\eqref{cap_Lpi_11:1}, and thus the error terms are larger. 
For $-\eta\leqs\lambda_1<c\sqrt{\eps\abs{\log\eps}}$, the definition of the set
$\widehat D$ has to be slightly modified. Since the same modification is needed
for all $L-\pi$ of order $1$, we postpone that part of the proof to the next
subsection. 
\end{proof}


\subsubsection{$L>\pi$}

We finally consider the case $L\geqs\pi+c$. Recall from Section~\ref{ssec_det}
the following properties of the deterministic system: 
\begin{enum}
\item	The infinite-dimensional system has exactly two saddles of index $1$,
given by functions $u^*_\pm(x)$ of class $\cC^2$ (at least). The fact that
$u^*_\pm(x)\in\cC^2$ implies that their Fourier
components decrease like $k^{-2}$. 
\item	The Hessian of $V$ corresponds to the second Fr\'echet derivative of $V$
at $u^*_\pm$, given by the map 
\begin{align}
\nonumber
(v_1,v_2) &\mapsto \int_0^L 
\Bigbrak{-v_1''(x) + U''(u^*_\pm(x))v_1(x)}v_2(x) \6x \\ 
& = \int_0^L 
\Bigbrak{v_1'(x)v_2'(x) + U''(u^*_\pm(x))v_1(x)v_2(x)} \6x\;. 
 \label{cap_largeL_00}
\end{align} 
The eigenvalues $\mu_k$ of the Hessian are solutions of the Sturm--Liouville
problem $v''(x)=- U''(u^*_\pm(x))v(x)$. They satisfy $\mu_0<0<\mu_1<\dots$ and 
\begin{equation}
 \label{cap_largeL_00A}
 -\gamma_1 + \gamma_2 k^2 \leqs \mu_k \leqs \gamma_3 k^2
\qquad\forall k\;,
\end{equation} 
for some constants $\gamma_1, \gamma_2, \gamma_3 > 0$ (this follows from
expressions for the asymptotics of the eigenvalues of Sturm--Liouville
equations, see for instance~\cite{VinokurovSadovnichi2000}).
In Fourier variables, we have 
\begin{equation}
 \label{cap_largeL_00B}
\hessian\Vhat(u^*_\pm) = \Lambda  + Q(u^*_\pm)\;,
\end{equation} 
where $\Lambda$ is a diagonal matrix with entries $\lambda_k$, and the matrix
$Q$ represents the second summand in the integral~\eqref{cap_largeL_00}. 
Thus if $v$ has Fourier coefficients $z$, we have
\begin{equation}
 \label{cap_largeL_00C}
\pscal{z}{Q(u^*_\pm)z} = \int_0^L U''(u^*_\pm(x))v(x)^2 \6x \;.
\end{equation} 
If $M$ is a constant such that $\abs{U''(u^*_\pm(x))}\leqs M$, for all $x$, we
get 
\begin{equation}
\bigabs{\pscal{z}{Q(u^*_\pm)z}} 
\leqs M \norm{v}_{L^2}^2 \\
= M \norm{z}_{\ell^2}^2\;.
 \label{cap_largeL_00D}
\end{equation} 
\item	As shown in Section~\ref{ssec_tpot}, similar statements hold true for
the finite-dimensional potential for sufficiently large $d$. As above we denote
the two saddles by $P_\pm$, and the eigenvalues of the Hessian by
$\mu_k=\mu_k(d)$. Let $S_\pm$ be the orthogonal change-of-basis matrices such
that 
\begin{equation}
 \label{cap_largeL_00E}
S_\pm\hessian\Vhat(u^*_\pm) \transpose{S_\pm} = \diag(\mu_0,\dots,\mu_d)\;.
\end{equation} 
\end{enum}

\begin{lemma}[Equivalence of norms]
\label{lem_norms_largeL} 
There exists a constant $\beta_0>0$, independent of $d$, such
that 
\begin{equation}
 \label{cap_largeL_01} 
\beta_0 \norm{y}_{H^1}^2 
\leqs \norm{S_\pm y}_{H^1}^2 
\leqs  \frac{1}{\beta_0} \norm{y}_{H^1}^2\;.
\end{equation} 
\end{lemma}
\begin{proof}
On one hand, the $S_\pm$ being orthogonal, $y$ and $z=S_\pm y$ have the same
$\ell^2$-norm and we have the obvious bound 
\begin{equation}
 \label{cap_largeL_01:01}
\norm{y}_{H^1}^2 \geqs \norm{y}_{\ell^2}^2 = \norm{z}_{\ell^2}^2\;. 
\end{equation}  
On the other hand, using $1+k^2 = 1+(1+\lambda_k)L^2/\pi^2$ and again equality
of the $\ell^2$-norms, we get 
\begin{equation}
 \label{cap_largeL_01:02}
\norm{y}_{H^1}^2  = 
\biggpar{1+\frac{L^2}{\pi^2}} \norm{z}_{\ell^2}^2 
+ \frac{L^2}{\pi^2} \sum_{k=0}^d \lambda_k y_k^2\;. 
\end{equation} 
Now by~\eqref{cap_largeL_00B} and~\eqref{cap_largeL_00D}, we have 
\begin{equation}
 \label{cap_largeL_01:03}
\sum_{k=0}^d \mu_k z_k^2 = 
\pscal{y}{\hessian\Vhat(P_+)y} =
\sum_{k=0}^d \lambda_k y_k^2 + \pscal{y}{Qy}
\leqs \sum_{k=0}^d \bigpar{\lambda_k+M} y_k^2\;.
\end{equation} 
It follows that 
\begin{equation}
 \label{cap_largeL_01:04}
\norm{y}_{H^1}^2 
\geqs \sum_{k=0}^d \biggpar{1 + \frac{L^2}{\pi^2} \bigbrak{1+\mu_k-M}} z_k^2
\bydef \sum_{k=0}^d c_k z_k^2\;. 
\end{equation} 
The lower bound~\eqref{cap_largeL_00A} implies that 
\begin{equation}
 \label{cap_largeL_01:05}
c_k \geqs  1 + \frac{L^2}{\pi^2} \bigbrak{1-M-\gamma_1}
+ \gamma_2\frac{L^2}{\pi^2}  k^2\;.
\end{equation} 
Let $k_0$ be the smallest integer such that $c_{k_0}\geqs 1$. We may assume
$d>k_0$, since otherwise there is nothing to prove. It is easy to check that 
\begin{equation}
 \label{cap_largeL_01:06}
c_k \geqs 
\begin{cases}
-\beta_1 & \text{for $0\leqs k \leqs k_0$\;,} \\
1+\beta_2 k^2 & \text{for $k_0+1 \leqs k \leqs d$\;,}
\end{cases} 
\end{equation} 
where $\beta_1=(M+\gamma_1-1)(L^2/\pi^2)-1$ and
$\beta_2=\gamma_2L^2/((k_0^2+1)\pi^2)$. Thus setting 
\begin{equation}
 \label{cap_largeL_01:07}
a = \sum_{k=0}^{k_0} z_k^2\;, \qquad
b_1 = \sum_{k=k_0+1}^d z_k^2\;, \qquad
b_2 = \sum_{k=k_0+1}^d (1+\beta_2 k^2)z_k^2\;,
\end{equation} 
we can write the bounds~\eqref{cap_largeL_01:01} and~\eqref{cap_largeL_01:04}
in the form 
\begin{equation}
 \label{cap_largeL_01:08}
\norm{y}_{H^1}^2 \geqs a + b_1 
\qquad\text{and}\qquad
\norm{y}_{H^1}^2 \geqs -\beta_1 a + b_2\;. 
\end{equation} 
By distinguishing the cases $(1+\beta_1) a \leqs -b_1+b_2$ and $(1+\beta_1) a
> -b_1+b_2$, one can deduce from these two inequalities that 
\begin{equation}
 \label{cap_largeL_01:09}
\norm{y}_{H^1}^2 \geqs \frac{a+b_2}{2+\beta_1}\;, 
\end{equation} 
which implies $\norm{y}_{H^1}^2 \geqs \beta_0\norm{z}_{H^1}^2$
for some $\beta_0>0$. The inequality $\norm{z}_{H^1}^2 \geqs
\beta_0\norm{y}_{H^1}^2$ can be proved in a similar way, using the upper
bound on the $\mu_k$. 
\end{proof}

We denote again by $\cW^{\math s}$ the basin boundary, which is formed by the
closure of the stable manifolds of $P_+$ and $P_-$.

\begin{lemma}[Growth of the potential along $\cW^{\math s}$]
\label{lem_Ws_largeL} 
There exists a constant $m_0>0$ such that for all $y\in\cW^{\math s}$, 
\begin{equation}
 \label{cap_largeL_10}
\Vhat(y) - \Vhat(P_+) = \Vhat(y) - \Vhat(P_-) \geqs m_0 
\bigpar{\norm{y-P_+}_{H^1}^2 \wedge \norm{y-P_-}_{H^1}^2}\;. 
\end{equation} 
\end{lemma}
\begin{proof}
We have 
\begin{equation}
 \label{cap_largeL_10:01}
\Vhat(P_++\transpose{S_+}z) = \Vhat(P_+) + \frac12\sum_{k=0}^d \mu_k z_k^2 +
\Order{\norm{z}_{H^s}^3}
\end{equation} 
for any $s>1/4$. Since the stable manifold can be described locally by an
equation of the form $z_0=g(z_\perp)$, we obtain, as in the proof of
Lemma~\ref{lem_Ws}, the existence of constants $m_1, \rho_1>0$ such that 
\begin{equation}
 \label{cap_largeL_10:02}
\Vhat(P_++\transpose{S_+}z) \geqs \Vhat(P_+) + m_1 \norm{z}_{H^1}^2
\qquad
\forall z\colon P_++\transpose{S}z \in \cW^{\math s}\,,\, 
\norm{z}_{H^1}\leqs\rho_1\;. 
\end{equation} 
By Lemma~\ref{lem_norms_largeL}, this implies 
\begin{equation}
 \label{cap_largeL_10:03}
\Vhat(y) \geqs \Vhat(P_+) + \beta_0 m_1 \norm{y-P_+}_{H^1}^2
\qquad
\forall y \in \cW^{\math s} \colon 
\norm{y-P_+}_{H^1}\leqs\sqrt{\beta_0}\,\rho_1\;. 
\end{equation} 
A similar bound holds in the neighbourhood of $P_-$. 

Now choose a $\gamma>0$ such that $-\alpha+\beta\gamma^2/4\geqs 1$, 
where $\alpha$ and $\beta$ are the constants appearing in~\eqref{cap_s01a:7},
and such that $\gamma\geqs 3\norm{P_+}_{H^1}$. We want to consider the case  of 
$y\in\cW^{\math s}$ satisfying $\sqrt{\beta_0}\rho_1\leqs\norm{y-P_+}_{H^1} \wedge \norm{y-P_-}_{H^1}\leqs\gamma$. Without loss of generality we may assume $\norm{y-P_+}_{H^1} \geqs \norm{y-P_-}_{H^1}$. As in the proof of Lemma~\ref{lem_Ws}, we use the fact that the vector field $-\nabla \Vhat(y)$ is pointing inward. Thus,
\begin{equation}
 \label{cap_largeL_10:04}
 \Vhat(y) - \Vhat(P_+) \geqs m_1\beta_0 (\sqrt{\beta_{0}} \rho_{1})^{2}
\geqs m_1 \beta_{0}^{2} \frac{\rho_1^2}{\gamma^2} 
\bigpar{\norm{y-P_+}_{H^1}^2 \wedge
\norm{y-P_-}_{H^1}^2}\;.
\end{equation} 
Together with~\eqref{cap_s01a:7} for $\norm{y-P_+}_{H^1} \wedge \norm{y-P_-}_{H^1}>\gamma$, this proves~\eqref{cap_largeL_10} for all
$y\in\cW^{\math s}$.  
\end{proof}

\begin{prop}[Upper bound on the capacity]
 \label{prop_cap_neumann_largeL_upper} 
There exist $r_0, \eps_0 >0$ and $d_0<\infty$ such that for $r<r_0$,
$\eps<\eps_0$ and $d\geqs d_0$, 
\begin{equation}
 \label{cap_largeL_11}
\capacity_A(B) \leqs 2\eps\sqrt{\frac{\abs{\mu_0}}{2\pi\eps}}
\Biggpar{\prod_{k=1}^d \sqrt{\frac{2\pi\eps}{\mu_k}}\,}
\e^{-\Vhat(P_\pm)/\eps}
\bigbrak{1 + c_+\eps^{1/2}\abs{\log\eps}^{3/2}}\;,
\end{equation} 
where the constant $c_+$ is independent of $\eps$ and $d$. 
\end{prop}

\begin{proof}
The proof is similar to the proof of
Proposition~\ref{prop_cap_neumann_smallL_upper}. We first compute the Dirichlet
form over a box $D_+$, defined in rotated coordinates $z=S_+(y-P_+)$ by
$\abs{z_k}\leqs\sqrt{c_k\eps\abs{\log\eps}/\abs{\mu_k}}$. Constructing $h_+$ as
a function of $z_0$ as before yields a contribution equal to half the
expression in~\eqref{cap_largeL_11}. The other half comes from a similar
contribution from a box $D_-$ centred in $P_-$. The remaining part of the
Dirichlet form can be shown to be negligible with the help of
Lemmas~\ref{lem_norms_largeL} and~\ref{lem_Ws_largeL}. 
\end{proof}

\begin{prop}[Lower bound on the capacity]
 \label{prop_cap_neumann_largeL_lower} 
There exist $r_0, \eps_0 >0$ and $d_0(\eps)<\infty$ such that for $r<r_0$,
$\eps<\eps_0$ and $d\geqs d_0(\eps)$, 
\begin{equation}
 \label{cap_largeL_12}
\capacity_A(B) \geqs 2\eps\sqrt{\frac{\abs{\mu_0}}{2\pi\eps}}
\Biggpar{\prod_{k=1}^d \sqrt{\frac{2\pi\eps}{\mu_k}}\,}
\e^{-\Vhat(P_\pm)/\eps}
\bigbrak{1 - c_-\eps^{1/2}\abs{\log\eps}^{3/2}}\;,
\end{equation} 
where the constant $c_-$ is independent of $\eps$ and $d$. 
\end{prop}
\begin{proof}
We perform the change of variables $y=P_++\transpose{S_+}z$ in the Dirichlet
form, which is an isometry, and thus of unit Jacobian. Let $A', B'$ denote the
images of $A$ and $B$ under the inverse isometry. 

Let $\Vtilde(z)=\Vhat(S_+(y-P_+))$ be the expression of the potential in the new
variables, given by~\eqref{cap_largeL_10:01}. We define $\widehat D_\perp$ as
in~\eqref{cap_s04:1} and set 
\begin{equation}
\label{cap_largeL12:3} 
\widehat D_+ = \bigsetsuch{z=(z_0,z_\perp)}
{z_\perp\in\widehat D_\perp, -\rho<z_0<\rho}\;.
\end{equation} 
Then by the same computation as in~\eqref{cap_s04:4}--\eqref{cap_s04:6C}, 
we have 
\begin{equation}
\label{cap_largeL12:4} 
\Phi_{\widehat D_+}(h^*) 
\geqs \eps \int_{\widehat D_\perp}
\frac{\bigbrak{1 - \Order{\e^{-H_0/\eps}}}^2}{\displaystyle \int_{-\rho}^{\rho}
\e^{\widetilde V(z_0,z_\perp)/\eps}\6z_0} \6z_\perp\;.
\end{equation} 
The function $z_0\mapsto \widetilde V(z_0,z_\perp)$ admits its
maximum in a point $z_0^*=\Order{\eps\abs{\log\eps}}$. We can thus apply the
Laplace method to obtain 
\begin{equation}
 \label{cap_largeL12:5}
 \int_{-\rho}^{\rho} \e^{\widetilde
V(z_0,\varphi(z_0)+z_\perp)/\eps}\6z_0
= \sqrt{\frac{2\pi\eps}{\abs{\mu_0}}} \e^{\Vhat(P_+)/\eps} 
\exp\biggset{\frac1{2\eps} \sum_{k=1}^d\mu_kz_k^2}
\bigbrak{1+\Order{\eps^{1/2}\abs{\log\eps}^{3/2}}}\;.
\end{equation} 
Substituting this into~\eqref{cap_largeL12:4}, the Dirichlet form
$\Phi_{\widehat D_+}(h^*)$ can be estimated as in~\eqref{cap_s04:11}.
 
Now a similar estimate holds for the Dirichlet form $\Phi_{\widehat D_-}(h^*)$
on a set $\widehat D_-$ constructed around $P_-$. The two sets may overlap, but
the contribution of the overlap to the capacity is negligible. 
\end{proof}


\subsection{Periodic b.c.}
\label{ssec_periodic}

We turn now to the study of capacities for periodic b.c. Since most arguments
are the same as for Neumann b.c., we only give the main results and briefly
comment on a few differences. 

The potential energy~\eqref{capn00} is invariant under translations $u\mapsto
u(\cdot+\varphi)$. As a consequence, when expressed in Fourier variables
it satisfies the symmetry 
\begin{equation}
 \label{capp01}
\Vhat\bigpar{\set{z_k}_{-d\leqs k\leqs d}} = 
\Vhat\bigpar{\set{\e^{2\icx\pi k\varphi/L} z_k}_{-d\leqs k\leqs d}}\;.
\end{equation} 
The eigenvalues of the Hessian of $\Vhat$ at the origin are of the form
\begin{equation}
 \label{capp02}
\lambda_k = -1 + \biggpar{\frac{2k\pi}{L}}^2\;, 
\qquad k = -d,\dots,d\;, 
\end{equation} 
and are thus doubly degenerate for $k\neq 0$. 

The case $L\leqs 2\pi - c$ is treated in exactly the same way as the case $L\leqs \pi - c$ for Neumann
b.c., with the result 
\begin{equation}
 \label{capp03}
\capacity_A(B) = \frac{\eps}{\sqrt{2\pi\eps}} 
\biggpar{\prod_{k=1}^d \frac{2\pi\eps}{\lambda_k} }
\bigbrak{1 + \Order{\eps^{1/2}\abs{\log\eps}^{3/2}}}\;, 
\end{equation} 
where the error term is uniform in $d$. 

For $2\pi - c < L \leqs 2\pi$, the capacity can again be estimated by using
the normal form. The only difference is that the centre manifold is now
two-dimensional, which leads to the expression 
\begin{equation}
 \label{capp04}
 \capacity_A(B) = \frac{\eps}{\sqrt{2\pi\eps}} 
\int_0^{2\pi} \int_{-\infty}^\infty \e^{-u(r_1,\varphi_1)/\eps}
r_1\6r_1\6\varphi_1 \biggpar{\prod_{k=2}^d \frac{2\pi\eps}{\lambda_k} }
\bigbrak{1 + \Order{R(\eps,\lambda_1)}}\;,
\end{equation} 
where 
\begin{equation}
 \label{capp05}
 u(r_1,\varphi_1) = \frac12\lambda_1 r_1^2 + C_4 r_1^4
\end{equation} 
results from the terms in $z_{\pm1}$ written in polar coordinates, and 
$R(\eps,\lambda_1)$ is the same as in \eqref{cap_Lpi_09}. The integral can be
expressed in terms of the distribution function of a Gaussian random variable,
cf.~\cite[Section~5.4]{BG2010}. 

For $2\pi \leqs L \leqs 2\pi + c$, the expression for the capacity is given
by~\eqref{capp04} with an extra term $\e^{V/\eps}$, where
$V\simeq-\lambda_1^2/(16C_4\eps)$ is the value of the potential at the
transition state. 

Finally in the case $L \geqs 2\pi + c$, we have to take into account the fact
that instead of isolated transition states, there is a whole family of
transition states $\set{P(\varphi)}_{0\leqs\varphi<L}$, satisfying by symmetry 
\begin{equation}
 \label{capp06}
P_k(\varphi) = \e^{2\icx\pi k\varphi/L} P_k(0)\;. 
\end{equation} 
The eigenvalues $\mu_k$ of the Hessian at any transition state
satisfy
\begin{equation}
 \label{capp07}
\mu_0 < \mu_{-1}=0 < \mu_{1} < \mu_2, \mu_{-2} <  \dots
\end{equation}
When evaluating the Dirichlet form, we construct an approximation of 
the equilibrium potential in a neighbourhood of the transition states
in a way which is invariant under the symmetry. The result is 
\begin{equation}
 \label{capp08}
 \capacity_A(B) = \frac{\eps}{\sqrt{2\pi\eps\abs{\mu_0}}} 
\ell_{\text{saddle}} \biggpar{\frac{2\pi\eps}{\mu_1}
\prod_{\abs{k}\geqs2}^d \frac{2\pi\eps}{\mu_k} }^{1/2}
\e^{-\Vhat(P(0))/\eps}
\bigbrak{1 + \Order{\eps^{1/2}\abs{\log\eps}^{3/2}}}\;,
\end{equation} 
where $\ell_{\text{saddle}}$ is the \lq\lq length of the saddle\rq\rq, due to
the integration along the direction with vanishing eigenvalue $\mu_{-1}$. It is
given by 
\begin{equation}
 \label{capp09}
\ell_{\text{saddle}} = \int_0^L \biggnorm{\dpar{P}{\varphi}}_{\ell^2}
\6\varphi \;, 
\end{equation}
where~\eqref{capp06} shows that 
\begin{equation}
 \label{capp10}
 \biggnorm{\dpar{P}{\varphi}}_{\ell^2}^2 = 
\sum_{k=-d}^d \biggpar{\frac{2k\pi}{L}}^2 \abs{P_k(0)}^2\;,
\end{equation}  
which converges as $d\to\infty$, by Parseval's identity, to
$\norm{(u^*_{1,0})'}_{L^2}^2$. Hence we have 
\begin{equation}
 \label{capp11}
\lim_{d\to\infty} \ell_{\text{saddle}}(d) 
= L \norm{(u^*_{1,0})'}_{L^2}\;.
\end{equation} 
When $L$ is close to $2\pi$, the normal form shows that $\abs{P_1(0)}^2 =
\abs{\lambda_1}/(2C_4) + \Order{\lambda_1^2}$, while the other components of
$P(0)$ are of order $\lambda_1^2$. Also the eigenvalue $\mu_1$ satisfies
$\mu_1=-2\lambda_1+\Order{\abs{\lambda_1}^{3/2}}$. This shows that 
$\ell_{\text{saddle}}=2\pi\sqrt{\mu_1/(8C_4)} + \Order{\mu_1}$, and allows to
check that the expressions~\eqref{capp08} and~\eqref{capp04} for the capacity
are indeed compatible.


\section{Uniform bounds on expected first-hitting times}
\label{sec_times}


\subsection{Integrating the equilibrium potential against the invariant measure}
\label{ssec_integ}

We define as before the sets $A, B\subset E$ as the open balls 
\begin{align}
\nonumber
A &= \bigsetsuch{u\in E}{\norm{u-u^*_-}_{L^\infty} < r}\;, \\
B &= \bigsetsuch{u\in E}{\norm{u-u^*_+}_{L^\infty} < \rho}\;.
\label{integ02B} 
\end{align}
The aim of this subsection is to obtain sharp upper and lower bounds on the 
integral 
\begin{equation}
 \label{integ01}
J_d(A,B) = 
\int_{E_d\setminus B_d} h^{(d)}_{A_d,B_d}(y) \e^{-\Vhat(y)/\eps} \6y\;, 
\end{equation} 
where $h^{(d)}_{A_d,B_d}$ is the equilibrium potential 
\begin{equation}
 \label{integ02}
 h^{(d)}_{A_d,B_d}(y) = \bigprobin{y}{\tau^{(d)}_{A_d} < \tau^{(d)}_{B_d}}\;.
\end{equation} 
Recall that the local minima $u^*_\pm$ of $V$ are also local minima of the
truncated potential, and that the eigenvalues of the Hessian of the potential at
$u^*_-$ are given by $\nu_k^-=(bk\pi/L)^2+U''(u_-)$, where $b=1$ and $k\in\N_0$
for Neumann b.c., and $b=2$ and $k\in\Z$ for periodic b.c. Recall that $u_{\pm}$ denote the minima of the local potential $U$.

\begin{prop}[Upper bound on the integral]
\label{prop_integ_upper}
There exist constants $r_0>0$, $\eps_0>0$ such that for any $\eps<\eps_0$, there exists a
$d_0=d_0(\eps)<\infty$ such that 
\begin{equation}
 \label{integ03}
J_d(A,B) \leqs \prod_{\abs{k}\leqs d} \sqrt{\frac{2\pi\eps}{\nu_k^-}}
\e^{-V[u^*_-]/\eps}
\bigbrak{1+c_+\eps^{1/2}\abs{\log\eps}^{3/2}} 
\end{equation} 
for all $0<r,\rho<r_0$, and all $d\geqs d_0$, where the
constant $c_+$ is
independent of $\eps$ and $d$. 
\end{prop}
\begin{proof}
Let $\delta_k=\sqrt{c_k\eps\abs{\log\eps}/\nu_k^-}$, where
$c_k=c_0(1+\log(1+\abs{k}))$. 
We introduce two sets 
\begin{align}
\nonumber 
C_d &= [u_- -\delta_0,u_-+\delta_0]\times\prod_{0<\abs{k}\leqs d}
[-\delta_k,\delta_k]\;, \\
D_d &= \bigsetsuch{y\in E_d}{\Vbar(y,u^*_-) > 0}\;,
\label{integ04:1} 
\end{align}
and split the domain of integration into $C_d$, $D_d\setminus B_d$, and the
remaining part of $E_d\setminus B_d$. By Laplace asymptotics (cf.~the quadratic
approximation argument used in the proof
of Proposition~\ref{prop_cap_neumann_smallL_upper}) one obtains that
\begin{equation}
 \label{integ04:2}
\int_{C_d} h^{(d)}_{A_d,B_d}(y) \e^{-\Vhat(y)/\eps} \6y
\leqs \int_{C_d} \e^{-\Vhat(y)/\eps} \6y
\end{equation} 
satisfies the upper bound~\eqref{integ03}. To bound the integral over $D_d$, we
use the bound on the equilibrium potential in
Proposition~\ref{new_prop_eqpot2} to get 
\begin{equation}
 \int_{D_d} h^{(d)}_{A_d,B_d}(y) \e^{-\Vhat(y)/\eps} \6y
\leqs 3 \int_{D_d}
\biggpar{\e^{-[\Vbar(y,A)-\eta+\Vhat(y)]/\eps} +
\e^{-[1/\eta + \Vhat(y)]/\eps}} \;.
 \label{integ04:3}
\end{equation} 
If $u^*_{\text{ts}}$ denotes a transition state, we have 
$\Vbar(y,A)=V[u^*_{\text{ts}}]-\Vhat(y)$. Choosing $\eta$ small enough that
$1/\eta \geqs V[u^*_{\text{ts}}] - V[u_+^*]$, we thus obtain
\begin{equation}
 \int_{D_d} h^{(d)}_{A_d,B_d}(y) \e^{-\Vhat(y)/\eps} \6y
\leqs 6 \e^{-(V[u^*_{\text{ts}}]-\eta)/\eps} \int_{D_d} \6y\;,
 \label{integ04:3B}
\end{equation} 
The lower bound~\eqref{bp05} on the potential implies that $D_d$ is contained
in a set $\set{\norm{y}_{H^1}\leqs M}$ for some $M$. The scaling
$y_k=\sqrt{M/(1+k^2)}z_k$ shows that 
\begin{equation}
 \label{integ04:4}
\int_{D_d} \6y \leqs M^{d+1/2} \prod_{\abs{k}\leqs d} \frac{1}{\sqrt{1+k^2}}
\int_{\fS^{2d}} \6z\;.
\end{equation} 
The volume of the sphere $\fS^{2d}$ is given by $2\pi^d/\Gamma(d)$, which by
Stirling's formula is bounded by $(M_1/d)^d$ for some constant $M_1$. Thus
choosing $d_0$ of order $1/\eps$ or larger ensures that the
integral~\eqref{integ04:3} is negligible if we take $\eta$ small enough. 

Finally, we can bound the integral of $\e^{-\Vhat(y)/\eps}$ over the remaining
space in the same way as in the proof of
Proposition~\ref{prop_cap_neumann_smallL_upper}, using again~\eqref{bp05} to
bound the potential below by a quadratic form. Choosing $c_0$ large enough
ensures that this integral is negligible as well. 
\end{proof}

\begin{prop}[Lower bound on the integral]
\label{prop_integ_lower}
There exist constants $r_0>0$, $\eps_0>0$ such that for any $\eps<\eps_0$, there exists a
$d_0=d_0(\eps)<\infty$ such that 
\begin{equation}
 \label{integ05}
J_d(A,B) \geqs \prod_{\abs{k}\leqs d} \sqrt{\frac{2\pi\eps}{\nu_k^-}}
\e^{-V[u^*_-]/\eps}
\bigbrak{1-c_-\eps^{1/2}\abs{\log\eps}^{3/2}} 
\end{equation} 
for all $0<r,\rho<r_0$, and all $d\geqs d_0$, where the
constant $c_-$ is independent of $\eps$ and $d$. 
\end{prop}
\begin{proof}
We define $C_d$ as in the previous proof. 
The fact that $h^{(d)}_{A_d,B_d}(y) = 1 - h^{(d)}_{B_d,A_d}(y)$ shows that 
\begin{align}
\nonumber
J_d(A,B) &\geqs \int_{C_d} h^{(d)}_{A_d,B_d}(y) \e^{-\Vhat(y)/\eps} \6y \\
&= \int_{C_d} \e^{-\Vhat(y)/\eps} \6y
- \int_{C_d} h^{(d)}_{B_d,A_d}(y) \e^{-\Vhat(y)/\eps} \6y\;.
\label{integ06:1} 
\end{align}
The first term on the right-hand side satisfies the claimed lower bound, by a
computation similar to the one in the proof of
Proposition~\ref{prop_cap_neumann_smallL_lower}, cf.~\eqref{cap_s04:11}.
Proposition~\ref{new_prop_eqpot1} shows that the second term on the right-hand
side is smaller than the first one by an exponentially small term.
\end{proof}


\subsection{Averaged bounds on expected first-hitting times}
\label{ssec_times}

We define the sets $A$ and $B$ as in~\eqref{integ02B}. According
to~\eqref{pot11}, 
\begin{equation}
 \label{times01}
\nu^{(d)}_{A,B}(dz) = \frac{-e_{A_d,B_d}(dz)
\e^{-\Vhat(z)/\eps}}{\capacity_{A_d}(B_d)} 
\end{equation} 
is a probability measure on $\partial A_d$. 

The following result implies Proposition~\ref{prop_bounds_finite}.

\begin{prop}
\label{prop_times_1}
There exist $r_0, \eps_{0}>0$ such that for $0<r,\rho<r_0$ and
$0<\eps<\eps_{0}$, there exists a $d_0=d_{0}(\eps)<\infty$ such that for all
$d\geqs d_0$,
\begin{equation}
 \label{times02}
C(d,\eps) \e^{H(d)/\eps} \bigbrak{1-R^-_{d,B}(\eps)}
\leqs \int_{\partial A_d} \bigexpecin{z}{\tau^{(d)}_{B_d}}\nu^{(d)}_{A,B}(dz)
\leqs 
C(d,\eps) \e^{H(d)/\eps} \bigbrak{1+R^+_{d,B}(\eps)}\;,
\end{equation} 
where the quantities $C(d,\eps)$, $H(d)$ and $R^\pm_{d,B}(\eps)$ are detailed
below.
\end{prop}
\begin{proof}
By~\eqref{pot12}, we have
\begin{equation}
 \label{times03:1}
\int_{\partial A_d} \bigexpecin{z}{\tau^{(d)}_{B_d}}\nu^{(d)}_{A,B}(dz)
= \frac{J_d(A,B)}{\capacity_{A_d}(B_d)}\;.
\end{equation} 
Hence the result follows immediately from
Propositions~\ref{prop_integ_upper}, 
\ref{prop_integ_lower}
and the bounds on capacities obtained in Section~\ref{sec_cap}.
\end{proof}

We end by listing the expressions of the quantities appearing
in~\eqref{times02}. In the case of Neumann b.c., they are of the following
form, depending on the value of $L$. 
\begin{itemiz}
\item	For $L<\pi-c$, Propositions~\ref{prop_cap_neumann_smallL_upper} 
and \ref{prop_cap_neumann_smallL_lower} yield a prefactor 
\begin{equation}
 \label{times04}
C(d,\eps) = 2\pi \biggpar{\frac{1}{\abs{\lambda_0}\nu_0^-} 
\prod_{k=1}^d \frac{\lambda_k}{\nu_k^-}}^{1/2} 
\end{equation} 
(recall that $\lambda_0=-1$). As $d\to\infty$, the product converges to an
infinite product which is finite, due to the fact that both $\lambda_k$
and $\nu_k^-$ grow like $(k\pi/L)^2$. Since the transition state is $u^*_0$, 
the exponent is given by 
\begin{equation}
 \label{times05}
H(d) = V[u^*_0] - V[u^*_-] = \abs{U(u_-)} 
\end{equation} 
and is independent of $d$. The error terms satisfy 
\begin{equation}
 \label{times06}
R^\pm_{d,B}(\eps) = \bigOrder{\eps^{1/2}\abs{\log\eps}^{3/2}} 
\end{equation} 
uniformly in $d$. 

\item	For $L>\pi+c$, Propositions~\ref{prop_cap_neumann_largeL_upper} 
and \ref{prop_cap_neumann_largeL_lower} yield a prefactor 
\begin{equation}
 \label{times07}
C(d,\eps) = \pi \biggpar{\frac{1}{\abs{\mu_0(d)}\nu_0^-} 
\prod_{k=1}^d \frac{\mu_k(d)}{\nu_k^-}}^{1/2}\;, 
\end{equation}
where the eigenvalues $\mu_k(d)$ depend on $d$. They converge, as $d\to\infty$,
to those of the Hessian at the transition state $u^*_{1,+}$, by the
implicit-function theorem argument given in Proposition~\ref{prop_tpot}. The
exponent is given by 
\begin{equation}
 \label{times08}
H(d) = \Vhat(P_\pm(d)) - V[u^*_-]\;,
\end{equation}
and converges to $V[u^*_{1,+}] - V[u^*_-]$ as $d\to\infty$. 
The error terms satisfy~\eqref{times06} as well.

\item	For $L\leqs\pi$, the value of the prefactor follows from
Propositions~\ref{prop_cap_neumann_Lpi_upper}
and~\ref{prop_cap_neumann_Lpi_lower}. Using the computations
of~\cite[Section~5.4]{BG2010} to determine the integral in~\eqref{cap_Lpi_07}
(note that our $C_4$ is equal to half the $C_4$ in that reference), we get 
\begin{equation}
 \label{times09}
C(d,\eps) = 2\pi \biggpar{\frac{1}{\abs{\lambda_0}\nu_0^-} 
\frac{\lambda_1+\sqrt{C_4\eps}}{\nu_1^-} 
\prod_{k=2}^d \frac{\lambda_k}{\nu_k^-}}^{1/2}
\frac{1}{\Psi_+(\lambda_1/\sqrt{C_4\eps}\,)}\;, 
\end{equation} 
where $\Psi_+$ is the function defined in~\eqref{psiplus}. The exponent is
still given by~\eqref{times05}, while the error terms are of the form 
\begin{equation}
 \label{times10}
R^\pm_{d,B}(\eps) = \biggOrder{\biggbrak{\frac{\eps\abs{\log\eps}^3}
{\lambda_1\vee\sqrt{\eps\abs{\log\eps}}}}^{1/2}}\;.
\end{equation} 

\item	For $L\geqs\pi$, again by 
Propositions~\ref{prop_cap_neumann_Lpi_upper}
and~\ref{prop_cap_neumann_Lpi_lower} and~\cite[Section~5.4]{BG2010}, 
\begin{equation}
 \label{times11}
C(d,\eps) = 2\pi \biggpar{\frac{1}{\abs{\mu_0(d)}\nu_0^-} 
\frac{\mu_1(d)+\sqrt{C_4\eps}}{\nu_1^-} 
\prod_{k=2}^d \frac{\mu_k(d)}{\nu_k^-}}^{1/2}
\frac{1}{\Psi_-(\mu_1(d)/\sqrt{C_4\eps}\,)}\;, 
\end{equation} 
where $\Psi_+$ is the function defined in~\eqref{psiminus}. The exponent is
again given by~\eqref{times08}, and the error terms satisfy~\eqref{times10}. 
\end{itemiz}

The expressions are similar for periodic b.c.



\appendix

\section{Monotonicity of the period}
\label{app_A} 

Consider the Hamiltonian system defined by the Hamiltonian~\eqref{ODE02}. 
Let $T(E)$ be the period of its periodic solution with energy $E$, given
by~\eqref{ODE04}. The following lemma provides a sufficient condition for $T$
being increasing in $E$.

\begin{lemma}
\label{lem_appendixA}
Assume that 
\begin{equation}
 \label{appA01}
U'(u)^2 - 2U(u)U''(u) > 0 
\qquad
\text{for all $u\in(u_-,u_+)\setminus\set{0}$\;.} 
\end{equation} 
Then $T(E)$ is strictly increasing on $[0,E_0)$. 
\end{lemma}
\begin{proof}
We parametrize the upper half of the periodic orbit by 
\begin{align}
\nonumber
u' &= \sqrt{2E} \sin\varphi\;, \\
-U(u) &= E \cos^2\varphi\;,
\label{appA02} 
\end{align}
where $\varphi\in[0,\pi]$. The second relation can be inverted, writing
$u=f_E(\varphi)$, where the function $f_E:[0,\pi]\to[u_2(E),u_3(E)]$ is
increasing and maps $[0,\pi/2]$ on $[u_2,0]$ and $[\pi/2,\pi]$ on $[0,u_3]$. 
Differentiating the relation $E\cos^2\varphi=-U(f_E(\varphi))$ shows that 
\begin{equation}
 \label{appA03}
\dpar{f_E}{\varphi} = \frac{2E\sin\varphi\cos\varphi}{U'(f_E(\varphi))}\;, 
\qquad
\dpar{f_E}{E} = -\frac{\cos^2\varphi}{U'(f_E(\varphi))}\;.
\end{equation} 
The period is given by  
\begin{equation}
\label{appA04} 
\frac{T(E)}{2} = \int_{u_2(E)}^{u_3(E)} \frac{\6u}{u'}
= \int_0^\pi \frac{\sqrt{2E}\cos\varphi}{U'(f_E(\varphi)))}
\6\varphi\;.
\end{equation} 
By~\eqref{appA03} we have 
\begin{equation}
 \label{appA05}
\dtot{}{E} \biggpar{\frac{\sqrt{2E}\cos\varphi}{U'(f_E(\varphi))} }
= \bigbrak{U'(f_E(\varphi))^2 - 2U(f_E(\varphi))U''(f_E(\varphi))}
\frac{\cos\varphi }{\sqrt{2E}\,U'(f_E(\varphi))^3}\;.
\end{equation} 
Since $\cos\varphi$ and $-U'(f_E(\varphi)$ have the same sign, the
assumption~\eqref{appA01} implies that the integral~\eqref{appA04} is strictly
increasing in $E$. 
\end{proof}

\small
\bibliography{../../BFG}

\def\cprime{$'$}
\providecommand{\bysame}{\leavevmode\hbox to3em{\hrulefill}\thinspace}
\providecommand{\MR}{\relax\ifhmode\unskip\space\fi MR }
\providecommand{\MRhref}[2]{%
  \href{http://www.ams.org/mathscinet-getitem?mr=#1}{#2}
}
\providecommand{\href}[2]{#2}
\begin{thebibliography}{BEGK04}

\bibitem[Bar12]{Barret_2012}
Florent Barret, \emph{Sharp asymptotics of metastable transition times for one
  dimensional {SPDE}s}, arXiv:1201.4440, 2012.

\bibitem[BBM10]{BBM2010}
Florent Barret, Anton Bovier, and Sylvie M\'el\'eard, \emph{Uniform estimates
  for metastable transition times in a coupled bistable system}, Electron. J.
  Probab. \textbf{15} (2010), 323--345.

\bibitem[BEGK04]{BEGK}
Anton Bovier, Michael Eckhoff, V{\'e}ronique Gayrard, and Markus Klein,
  \emph{Metastability in reversible diffusion processes. {I}. {S}harp
  asymptotics for capacities and exit times}, J. Eur. Math. Soc. (JEMS)
  \textbf{6} (2004), no.~4, 399--424.

\bibitem[Ber11]{Berglund_Kramers_11}
Nils Berglund, \emph{{K}ramers' law: Validity, derivations and
  generalisations}, arXiv:1106.5799v1, 2011.

\bibitem[BFG07a]{BFG06a}
Nils Berglund, Bastien Fernandez, and Barbara Gentz, \emph{Metastability in
  interacting nonlinear stochastic differential equations: {I}. {F}rom weak
  coupling to synchronization}, Nonlinearity \textbf{20} (2007), no.~11,
  2551--2581.

\bibitem[BFG07b]{BFG06b}
\bysame, \emph{Metastability in interacting nonlinear stochastic differential
  equations~{II}: {L}arge-${N}$ behaviour}, Nonlinearity \textbf{20} (2007),
  no.~11, 2583--2614.

\bibitem[BG09]{BG09a}
Nils Berglund and Barbara Gentz, \emph{Anomalous behavior of the {K}ramers rate
  at bifurcations in classical field theories}, J. Phys. A: Math. Theor
  \textbf{42} (2009), 052001.

\bibitem[BG10]{BG2010}
\bysame, \emph{The {E}yring--{K}ramers law for potentials with nonquadratic
  saddles}, Markov Processes Relat. Fields \textbf{16} (2010), 549--598.

\bibitem[BGK05]{BGK}
Anton Bovier, V\'eronique Gayrard, and Markus Klein, \emph{Metastability in
  reversible diffusion processes. {II}. {P}recise asymptotics for small
  eigenvalues}, J. Eur. Math. Soc. (JEMS) \textbf{7} (2005), no.~1, 69--99.

\bibitem[BJ13]{Blomker_Jentzen_09}
Dirk Bl{\"o}mker and Arnulf Jentzen, \emph{Galerkin approximations for the
  stochastic {B}urgers equation}, to appear in SIAM J.~Numer.\ Anal., 2013.

\bibitem[CdV99]{ColinVerdiere1999}
Yves Colin~de Verdi{\`e}re, \emph{D\'eterminants et int\'egrales de {F}resnel},
  Ann. Inst. Fourier (Grenoble) \textbf{49} (1999), no.~3, 861--881, Symposium
  {\`a} la M{\'e}moire de Fran{\c{c}}ois Jaeger (Grenoble, 1998).

\bibitem[Cer96]{Cerrai_1996}
Sandra Cerrai, \emph{Elliptic and parabolic equations in {${\bf R}^n$} with
  coefficients having polynomial growth}, Comm. Partial Differential Equations
  \textbf{21} (1996), no.~1-2, 281--317.

\bibitem[Cer99]{Cerrai_1999}
\bysame, \emph{Smoothing properties of transition semigroups relative to {SDE}s
  with values in {B}anach spaces}, Probab. Theory Related Fields \textbf{113}
  (1999), no.~1, 85--114.

\bibitem[CI75]{ChafeeInfante_74}
N.~Chafee and E.~F. Infante, \emph{A bifurcation problem for a nonlinear
  partial differential equation of parabolic type}, Applicable Anal. \textbf{4}
  (1974/75), 17--37.

\bibitem[CM97]{ChenalMillet1997}
Fabien Chenal and Annie Millet, \emph{Uniform large deviations for parabolic
  {SPDE}s and applications}, Stochastic Process. Appl. \textbf{72} (1997),
  no.~2, 161--186.

\bibitem[dH04]{denHollander04}
F.~den Hollander, \emph{Metastability under stochastic dynamics}, Stochastic
  Process. Appl. \textbf{114} (2004), no.~1, 1--26.

\bibitem[DPZ92]{DaPrato_Jabczyk_92}
Giuseppe Da~Prato and Jerzy Zabczyk, \emph{Stochastic equations in infinite
  dimensions}, Encyclopedia of Mathematics and its Applications, vol.~44,
  Cambridge University Press, Cambridge, 1992.

\bibitem[DS88]{Dunford_Schwartz_I}
Nelson Dunford and Jacob~T. Schwartz, \emph{Linear operators. {P}art {I}},
  Wiley Classics Library, John Wiley \& Sons Inc., New York, 1988.

\bibitem[Eyr35]{Eyring}
H.~Eyring, \emph{The activated complex in chemical reactions}, Journal of
  Chemical Physics \textbf{3} (1935), 107--115.

\bibitem[FJL82]{Faris_JonaLasinio82}
William~G. Faris and Giovanni Jona-Lasinio, \emph{Large fluctuations for a
  nonlinear heat equation with noise}, J. Phys. A \textbf{15} (1982), no.~10,
  3025--3055.

\bibitem[For87]{Forman1987}
Robin Forman, \emph{Functional determinants and geometry}, Invent. Math.
  \textbf{88} (1987), no.~3, 447--493.

\bibitem[Fre88]{Freidlin88}
Mark~I. Freidlin, \emph{Random perturbations of reaction-diffusion equations:
  the quasideterministic approximation}, Trans. Amer. Math. Soc. \textbf{305}
  (1988), no.~2, 665--697. \MR{924775 (89f:35110)}

\bibitem[FW98]{FW}
M.~I. Freidlin and A.~D. Wentzell, \emph{Random perturbations of dynamical
  systems}, second ed., Springer-Verlag, New York, 1998.

\bibitem[Gal93]{Gallay_92}
Th. Gallay, \emph{A center-stable manifold theorem for differential equations
  in {B}anach spaces}, Comm. Math. Phys. \textbf{152} (1993), no.~2, 249--268.

\bibitem[Hai09]{Hairer_LN_2009}
Martin Hairer, \emph{An introduction to stochastic {PDE}s}, Lecture notes,
  2009, {\tt http://arxiv.org/abs/0907.4178}.

\bibitem[Jet86]{Jetschke_86}
G.~Jetschke, \emph{On the equivalence of different approaches to stochastic
  partial differential equations}, Math. Nachr. \textbf{128} (1986), 315--329.

\bibitem[Jol89]{Jolly_89}
Michael~S. Jolly, \emph{Explicit construction of an inertial manifold for a
  reaction diffusion equation}, J. Differential Equations \textbf{78} (1989),
  no.~2, 220--261.

\bibitem[Kra40]{Kramers}
H.~A. Kramers, \emph{Brownian motion in a field of force and the diffusion
  model of chemical reactions}, Physica \textbf{7} (1940), 284--304.

\bibitem[Liu03]{Liu_CMS_2003}
Di~Liu, \emph{Convergence of the spectral method for stochastic
  {G}inzburg-{L}andau equation driven by space-time white noise}, Commun. Math.
  Sci. \textbf{1} (2003), no.~2, 361--375.

\bibitem[MOS89]{Martinelli_Olivieri_Scoppola_89}
Fabio Martinelli, Enzo Olivieri, and Elisabetta Scoppola, \emph{Small random
  perturbations of finite- and infinite-dimensional dynamical systems:
  unpredictability of exit times}, J. Statist. Phys. \textbf{55} (1989),
  no.~3-4, 477--504.

\bibitem[MS01]{Maier_Stein_PRL_01}
Robert~S. Maier and D.~L. Stein, \emph{Droplet nucleation and domain wall
  motion in a bounded interval}, Phys. Rev. Lett. \textbf{87} (2001),
  270601--1.

\bibitem[MS03]{Maier_Stein_SPIE_2003}
\bysame, \emph{The effects of weak spatiotemporal noise on a bistable
  one-dimensional system}, Noise in complex systems and stochastic dynamics
  (L.~Schimanski-Geier, D.~Abbott, A.~Neimann, and C.~Van~den Broeck, eds.),
  SPIE Proceedings Series, vol. 5114, 2003, pp.~67--78.

\bibitem[MT95]{McKane_Tarlie_1995}
A.~J. McKane and M.B. Tarlie, \emph{Regularization of functional determinants
  using boundary conditions}, J. Phys. A \textbf{28} (1995), 6931--6942.

\bibitem[OV05]{OlivieriVares05}
Enzo Olivieri and Maria~Eul{\'a}lia Vares, \emph{Large deviations and
  metastability}, Encyclopedia of Mathematics and its Applications, vol. 100,
  Cambridge University Press, Cambridge, 2005.

\bibitem[Sai00]{Saitoh_2000}
Saburou Saitoh, \emph{Weighted {$L\sb p$}-norm inequalities in convolutions},
  Survey on classical inequalities, Math. Appl., vol. 517, Kluwer Acad. Publ.,
  Dordrecht, 2000, pp.~225--234.

\bibitem[Ste04]{Stein_JSP_04}
D.~L. Stein, \emph{Critical behavior of the {K}ramers escape rate in asymmetric
  classical field theories}, J. Stat. Phys. \textbf{114} (2004), 1537--1556.

\bibitem[Ste05]{Stein04}
\bysame, \emph{Large fluctuations, classical activation, quantum tunneling, and
  phase transitions}, Braz. J. Phys. \textbf{35} (2005), 242--252.

\bibitem[VF69]{VF69}
A.~D. Ventcel{\cprime} and M.~I. Fre\u{\i}dlin, \emph{Small random
  perturbations of a dynamical system with stable equilibrium position}, Dokl.
  Akad. Nauk SSSR \textbf{187} (1969), 506--509.

\bibitem[VS00]{VinokurovSadovnichi2000}
V.~A. Vinokurov and V.~A. Sadovnichi{\u\i}, \emph{Asymptotics of arbitrary
  order of the eigenvalues and eigenfunctions of the {S}turm-{L}iouville
  boundary value problem in an interval with a summable potential}, Izv. Ross.
  Akad. Nauk Ser. Mat. \textbf{64} (2000), no.~4, 47--108.

\end{thebibliography}
\bibliographystyle{amsalpha}               

\goodbreak
\bigskip\bigskip\noindent
{\small 
\noindent
Nils Berglund \\ 
Universit\'e d'Orl\'eans, Laboratoire {\sc Mapmo} \\
{\sc CNRS, UMR 6628} \\
F\'ed\'eration Denis Poisson, FR 2964 \\
B\^atiment de Math\'ematiques, B.P. 6759\\
45067~Orl\'eans Cedex 2, France \\
{\it E-mail address: }{\tt nils.berglund@univ-orleans.fr}

\bigskip\noindent
Barbara Gentz \\ 
Faculty of Mathematics, University of Bielefeld \\
P.O. Box 10 01 31, 33501~Bielefeld, Germany \\
{\it E-mail address: }{\tt gentz@math.uni-bielefeld.de}

}


\end{document}